\documentclass[12pt]{article}
\usepackage{xcolor,graphicx,comment}
\usepackage{algorithm,algorithmic}
\usepackage{amsmath,amssymb,amsthm,a4wide,color}
\usepackage{graphicx,xcolor,pictex,epstopdf}
\input prepictex
\input pictex
\input postpictex
\input colordvi
\font\tenrm=cmr10
\font\teni=cmmi10 \skewchar\teni='177
\font\tensy=cmsy10 \skewchar\tensy='60
\font\tenex=cmex10
\font\tenit=cmti10
\font\tensl=cmsl10
\font\tenbf=cmbx10
\font\tentt=cmtt10
\font\ninerm=cmr9
\font\ninei=cmmi9 \skewchar\ninei='177
\font\ninesy=cmsy9 \skewchar\ninesy='60
\font\nineit=cmti9
\font\ninesl=cmsl9
\font\ninebf=cmbx9
\font\ninett=cmtt9
\font\eightrm=cmr8
\font\eighti=cmmi8 \skewchar\eighti='177
\font\eightsy=cmsy8 \skewchar\eightsy='60
\font\eightit=cmti8
\font\eightsl=cmsl8
\font\eightbf=cmbx8
\font\eighttt=cmtt8
\font\sevenrm=cmr7
\font\seveni=cmmi7 \skewchar\seveni='177
\font\sevensy=cmsy7 \skewchar\sevensy='60
\font\sevenbf=cmbx7
\font\sevenit=cmmi7
\font\sevensl=cmmi7
\font\seventt=cmr7
\font\sixrm=cmr6
\font\sixi=cmmi6 \skewchar\sixi='177
\font\sixsy=cmsy6 \skewchar\sixsy='60
\font\sixbf=cmbx6
\font\fiverm=cmr5
\font\fivei=cmmi5 \skewchar\fivei='177
\font\fivesy=cmsy5 \skewchar\fivesy='60
\font\fivebf=cmbx5
\def\tenpoint{\def\rm{\fam0\tenrm}%
        \textfont0=\tenrm \scriptfont0=\sevenrm \scriptscriptfont0=\fiverm
        \textfont1=\teni \scriptfont1=\seveni \scriptscriptfont1=\fivei
        \textfont2=\tensy \scriptfont2=\sevensy \scriptscriptfont2=\fivesy
        \textfont3=\tenex \scriptfont3=\tenex \scriptscriptfont3=\tenex
        \def\it{\fam\itfam\tenit}%
        \textfont\itfam=\tenit
        \def\sl{\fam\slfam\tensl}%
        \textfont\slfam=\tensl
        \def\bf{\fam\bffam\tenbf}%
        \textfont\bffam=\tenbf \scriptfont\bffam=\sevenbf
                \scriptscriptfont\bffam=\fivebf
        \def\tt{\fam\ttfam\tentt}%
        \textfont\ttfam=\tentt
        \normalbaselineskip=12pt%
        \let\sc=\eightrm        
        \setbox\strutbox=\hbox{\vrule height8.5pt depth3.5pt width0pt}%
        \normalbaselines\rm}
\def\ninepoint{\def\rm{\fam0\ninerm}%
        \textfont0=\ninerm \scriptfont0=\sixrm \scriptscriptfont0=\fiverm
        \textfont1=\ninei \scriptfont1=\sixi \scriptscriptfont1=\fivei
        \textfont2=\ninesy \scriptfont2=\sixsy \scriptscriptfont2=\fivesy
        \textfont3=\tenex \scriptfont3=\tenex \scriptscriptfont3=\tenex
        \def\it{\fam\itfam\nineit}%
        \textfont\itfam=\nineit
        \def\sl{\fam\slfam\ninesl}%
        \textfont\slfam=\ninesl
        \def\bf{\fam\bffam\ninebf}%
        \textfont\bffam=\ninebf \scriptfont\bffam=\sixbf
                \scriptscriptfont\bffam=\fivebf
        \def\tt{\fam\ttfam\ninett}%
        \textfont\ttfam=\ninett
        \normalbaselineskip=11pt%
        \let\sc=\sevenrm        
        \setbox\strutbox=\hbox{\vrule height8pt depth3pt width0pt}%
        \normalbaselines\rm}
\def\eightpoint{\def\rm{\fam0\eightrm}%
        \textfont0=\eightrm \scriptfont0=\sixrm \scriptscriptfont0=\fiverm
        \textfont1=\eighti \scriptfont1=\sixi \scriptscriptfont1=\fivei
        \textfont2=\eightsy \scriptfont2=\sixsy \scriptscriptfont2=\fivesy
        \textfont3=\tenex \scriptfont3=\tenex \scriptscriptfont3=\tenex
        \def\it{\fam\itfam\eightit}%
        \textfont\itfam=\eightit
        \def\sl{\fam\slfam\eightsl}%
        \textfont\slfam=\eightsl
        \def\bf{\fam\bffam\eightbf}%
        \textfont\bffam=\eightbf \scriptfont\bffam=\sixbf
                \scriptscriptfont\bffam=\fivebf
        \def\tt{\fam\ttfam\eighttt}%
        \textfont\ttfam=\eighttt
        \normalbaselineskip=9pt%
        \let\sc=\sixrm  
        \setbox\strutbox=\hbox{\vrule height7pt depth2pt width0pt}%
        \normalbaselines\rm}
\def\sevenpoint{\def\rm{\fam0\sevenrm}%
        \textfont0=\sevenrm \scriptfont0=\fiverm \scriptscriptfont0=\fiverm
        \textfont1=\seveni \scriptfont1=\fivei \scriptscriptfont1=\fivei
        \textfont2=\sevensy \scriptfont2=\fivesy \scriptscriptfont2=\fivesy
        \textfont3=\tenex \scriptfont3=\tenex \scriptscriptfont3=\tenex
        \def\it{\fam\itfam\sevenit}%
        \textfont\itfam=\sevenit
        \def\sl{\fam\slfam\sevensl}%
        \textfont\slfam=\sevensl
        \def\bf{\fam\bffam\sevenbf}%
        \textfont\bffam=\sevenbf \scriptfont\bffam=\fivebf
                \scriptscriptfont\bffam=\fivebf
        \def\tt{\fam\ttfam\seventt}%
        \textfont\ttfam=\seventt
        \normalbaselineskip=8pt%
        \let\sc=\fiverm  
        \setbox\strutbox=\hbox{\vrule height6pt depth2pt width0pt}%
        \normalbaselines\rm}

\newbox\figureone
\newbox\figuretwo
\newbox\figurethree
\newbox\figurefour
\newbox\figurefive
\newbox\figuresix
\newbox\figureseven
\newbox\figureeight
\newbox\figurenine
\newbox\figureten
\newbox\figureeleven
\newbox\figureutwo
\newbox\figureuthree
\newbox\figurefourutwo
\newbox\figurequarteruthree
\newbox\figurelegend
\newbox\figurelegendtexttwo
\newbox\figurelegendtext
\newbox\figurelegendvalues
\newdimen\xfiglen
\newdimen\yfiglen
\newdimen\xfigdim
\newdimen\yfigdim
\newdimen\legendheight
\newdimen\xfiglenshort \newdimen\yfiglenshort
\newtheorem{theorem}{Theorem}[section]

\newtheorem{remark}{Remark}[section]
\newcommand{\uld}[1]{\underline{d#1}}
\newcommand{\dete}{\text{det}}
\newcommand{\pact}{p^{\text{act}}}
\newcommand{\qact}{q^{\text{act}}}
\newcommand{\uact}{u^{\text{act}}}
\newcommand{\ainf}{a^\infty}
\newcommand{\Binf}{B^\infty}
\newcommand{\pinf}{p^\infty}
\newcommand{\rinf}{r^\infty}
\newcommand{\uinf}{u^\infty}
\newcommand{\usim}{u^\sim}

\newcommand{\decayfunc}{\Psi}
\newcommand{\rrr}{\mathfrak{r}}
\newcommand{\sss}{\mathfrak{s}}
\newcommand{\altil}{{\tilde{\alpha}}}
\definecolor{grey}{rgb}{0.5,0.5,0.5}

\title{
Identification
of space-dependent coefficients in two 
competing 
terms of a nonlinear subdiffusion equation
}
\author{Barbara Kaltenbacher\footnote{
Department of Mathematics,
Alpen-Adria-Universit\"at Klagenfurt.
barbara.kaltenbacher@aau.at}
\and
William Rundell\footnote{
Department of Mathematics,
Texas A\&M University, College Station, Texas 77843 USA.
rundell@math.tamu.edu}
}
\begin{document}
\maketitle
\begin{abstract}
We consider a (sub)diffusion equation with a nonlinearity of the form $pf(u)-qu$, where $p$ and $q$ are space dependent functions.
Prominent examples are the Fisher-KPP, the Frank-Kamenetskii-Zeldovich and the Allen-Cahn equations. 
We devise a fixed point scheme for reconstructing the spatially varying coefficients from interior observations a) at final time under two different excitations b) at two different time instances under a single excitation. Convergence of the scheme as well as local uniqueness of these coefficients is proven. 
Numerical experiments illustrate the performance of the reconstruction scheme.  
\end{abstract}
keywords: coefficient identification, reaction-(sub)diffusion equation, fixed point scheme, generalised Fisher-KPP
\section{Introduction}

\subsection{The model}
The equation
\begin{equation}\label{pq_equation}
u_t - D \triangle u = q(x)u - p(x)f(u) 
\end{equation}
has become a standard model in a variety of
applications including population and chemical reaction models.
The usual assumption is that both $p$ and $q$ are non-negative.
It is critical that $p(x)$ is so, else the solution can ``blow up'' in finite
time if $f(u) \approx u^{1+\epsilon}$. Typically $f(u)$
is taken to be a simple power $u^k$ with $k$ and integer $\geq 2$,
but general smooth functions $f$ make sense in many applications and have in fact already been considered early on in the mathematical literature \cite{AronsonWeinberger}.

Such equations have a long history.
The study of the ordinary differential (logistic) equation
$u' = qu(1-\frac{1}{K}u)$ with the growth rate $q$ and the carrying capacity $K$, 
modelling population growth began with the work of Verhulst, first in 1838
and then with an analysis of the model in 1845, \cite{Verhulst1838,Verhulst:1845}.
Similar quadratic term additions for interactions between and within species was considered by Lotka beginning in 1913 and later modified in 1920 \cite{Lotka1920a}.
This work was later extended, especially mathematically, by Volterra
\cite{volterra1926variazioni},
giving rise to a coupled system with
linear growth and quadratic interactions that has been used in
modelling of a wide range of application including
biological/ecological systems, business and market competition and chemical kinetics.

This was extended to include ``diffusion'' almost a hundred years later
by Fisher in 1937, \cite{Fisher:1937}  with the specific case of $f(u) = u^2$.
This gives, in one spatial variable,
$u_t - u_{xx} = qu(1-\frac{1}{K}u)$.
Fisher proposed it as a model for the propagation of a mutant gene,
but the number of applications have multiplied considerably.
In this same year Kolmogorov, Petrovsky and Piskunov, \cite{KPP}
introduced the more general reaction-diffusion equation
and the combination of the
models is now referred to as the Fisher-KPP equation.
In another direction, that of chemical reactions, a classical model is
due to Frank-Kamenetskii and Zeldovich, \cite{FrankZeldovich} which
takes $f(u)=u^3$. 
This is the same nonlinearity as in the Allen-Cahn equation modelling phase separation in many physical applications \cite{AllenCahn:1972,PenroseFife:1990}.

The original Fisher paper showed that the eponymous equation on the real line
$x$ has travelling wave solutions of the form $u(x,t) = U(z) = U(x-ct)$ 
where $c$ is the wave speed.
A typical wavefront solution $U$ would have $\lim_{z\to\infty}U(z) = 0$
and $\lim_{z\to-\infty}U(z) = 1$.
Fisher originally showed that such travelling waves exist provided $c\geq 2$
and this had clear import on the intended application as the title of Fisher's
paper was ``The Wave of Advance of Advantageous Genes.''
There has been much further work in this direction including some
exact solutions. One such is $U(z) = 1/(1+ae^{bz})^s$ where $a$, $b$ and $s$
are positive constants.
An excellent, if now partly outdated, survey here is Chapter~11
of \cite{Murray:2002}.

The case of bounded spatial domains has also received attention but at
a more restricted level as travelling wave solutions are the frequently
the sought-after type.
One such case is the Fisher-Stefan case where for example $u(0,t)=0$
and $u_x(s(t),t) = 0$ for some curve $s(t)$, \cite{AblowitzZeppetella}.
See also \cite{TysonBrazhnikPavel}.

Perhaps surprisingly, there are few papers in the  literature
for the Fisher equation where the coefficients $q$ and $K$
are spatially-dependent but there are some with these time-dependent,
see \cite{Hammond}.
Although this is a very obvious {\sc {pde}} question on its own footing,
this can be partly explained given the search
is often for travelling wave solutions.
A generalisation to include spatial dependence of $q(x)$ and $\frac{q}{K}=p(x)$
adds considerable mathematical
complexity including the possibility of either more complex travelling wave
solutions or perhaps none at all.

The large time limit even for the setting \eqref{pq_equation} of general $f$ but with constant coefficients $p$, $q$, has already been studied in \cite{AronsonWeinberger}.
Steady state solutions of the Fisher-KPP equation as well as their relations to travelling wave solutions have been subject of intensive research, see, e.g., \cite{BerestyckiNadinPerthameRyzhik:2009,CabreRoquejoffre:2012} for the non-local setting. 
The same holds true for the steady state solutions of the Allen-Cahn equation, that can be characterised as minimisers of an appropriate energy functional, see, e.g., \cite{AkagiSchimpernaSegatti:2016,NinomiyaHirokazuTaniguchi:2005}; in particular, we point to \cite{DuYangZhou:2020}, \cite[Section 6.4.2]{Jin:2021} for the time-fractional case.

Generalisation of the parabolic setting to a time-fractional one is in fact highly relevant in the mentioned applications, as it allows to extend the classical diffusion model (Fick's law) to one of a so-called anomalous diffusion, thus leading to a non-local, time-fractional equation \eqref{ibvp} below.


\medskip

The 
main task of this paper will be to study reconstruction of the space dependent functions $p(x)$, $q(x)$ from appropriate observations.

Since we have two spatial coefficients $p(x)$ and $q(x)$ to determine we
will require at least two measurements.
These can be obtained from a ``two-run'' situation where we are free to
add a right hand side input $r$ in \eqref{pq_equation} 
with two different sources $r_i(t,x)$ for $i=1,\,2$
and make measurements for each at a later time $T$.

An alternative situation (and one more in line with observing rather than
as an experimental set-up) is for a single source input to measure
data $u(x,t)$ at $t=T_1$ and $t=T_2$; that is ``two time'' measurements.
This corresponds to the usual terminology in the literature for the
recovery of a spatially-dependent unknown from spatial data.
Of course, an obvious extension is to have $N$ sampling points at $T_i$,
for $i=1,\;\ldots\; N$,
and follow Verhulst \cite{Verhulst:1845} with a least squares approach.
However, our focus here is on the mathematical question of what can be achieved
from the minimum amount of information sufficient to guarantee uniqueness.
Some further observation settings are discussed in Section~\ref{subsec:reconstructions}, but not subject to further analysis here.

Multi-coefficient inverse problems for subdiffusion equations have been recently studied, e.g., in \cite{CenZhou:2024,BB4}, which consider the inverse problem of recovering both the unknown, spatially-dependent conductivity $a(x)$ and the potential $q(x)$ from overposed data consisting of the value of solution profiles taken at a later time $T$.

Inverse problems for reaction diffusion equations where the term $f(u)$
has to be recovered also have a history dating back at least 40 years.
Early examples are \cite{AronsonWeinberger,DuChateauRundell1985,PilantRundell1986}.
More recently models involving an unknown $f(u)$ and additional coefficients
to also be determined have been studied in the context of subdiffusion models in place of the parabolic one above:
In \cite{BB5} reconstructing reaction-diffusion systems, in \cite{BB6}
the simultaneous recovery of the spatially-dependent conductivity $a(x)$ and the reaction term $f(u)$, and in \cite{BB7} the reconstruction of a nonlinear conductivity $a(u)$ was considered.
The challenging case of a term $q(x)f(u)$
where both $q$ and $f$ are unknown was studied in \cite{BB18}.
While we assume $f(u)$ to be fixed and known here, its recovery, alongside with $p(x)$ and $q(x)$ could constitute an interesting follow-up project of this paper.

\subsection{The inverse problem}
We aim to determine $q(x)$, $p(x)$ in the initial boundary value problem
\begin{equation}\label{ibvp}
\begin{aligned}
\partial_t^\alpha u_i-\mathbb{L}u_i&=qu_i-pf(u_i)+r_i \text{ in }(0,T)\times\Omega, \\ 
B_i u_i&=a_i \text{ on }(0,T)\times\partial\Omega, \\
u_i&=u^0_i\text{ in }\{0\}\times\Omega\\
&i\in\{1,2\}
\end{aligned}
\end{equation}
from observations
\begin{equation}\label{obs}
g_i(x)=u_i(T,x), \quad x\in\Omega, 
\qquad i\in\{1,2\}
\end{equation}

Alternatively, we consider the experimental setup of applying a single excitation $r(t,x)$ and measuring at two different time instances $T_1$, $T_2$, 
\begin{equation}\label{obs2}
g_i(x)=u(T_i,x), \quad x\in\Omega, 
\qquad i\in\{1,2\}.
\end{equation}

\medskip

Typical examples for $f$ in \eqref{ibvp} are
\[
\begin{aligned}
&f(u)= u^2 \text{\ldots Fisher-KPP}, 
\\
&f(u)= u^3 \text{\ldots Allen Cahn}, 
\end{aligned}
\]
but we will allow for quite general $C^2$ functions $f$ as nonlinearities here.

In \eqref{ibvp}, $\Omega\subseteq\mathbb{R}^d$, $d\in\{1,2,3\}$ is assumed to be a smooth and bounded domain and the operators $B_i$ defining the boundary conditions are assumed to be linear but otherwise quite general; typical examples are Dirichlet, Neumann, or impedance conditions.

Without loss of generality we assume 
\begin{equation}\label{f0fprime0}
f(0)=0,\qquad f'(0)=0,
\end{equation}
for otherwise, in view of the identity 
\[
-p f(u)+qu+r=-p(f(u)-f'(0)u-f(0))-(q+f'(0)p)u+r-f(0)p
\]
the replacements $f\to f(u)-f'(0)u-f(0)$, $q\to q+f'(0)p$, $r\to r-f(0)p$ could be made. Note that we will require $p$ and $f$ to be positive, but do not impose any sign conditions on $q$ and $r$.

The elliptic operator is assumed to be of the form
\begin{equation}\label{ellipticityL}
\mathbb{L}=-\nabla\cdot(D\nabla)+d
\text{ with }d(x)\geq0, \ \lambda_{\min}(D(x))\geq \underline{c}_D>0 \quad \text{for a.e. }x\in\Omega,
\end{equation}
satisfying (along with the boundary conditions) elliptic regularity
\begin{equation}\label{ellipticregularityL}
\|v\|_{H^2(\Omega)}\leq C_{\text{ell}}^\Omega
\Bigl(\|\mathbb{L}v\|_{L^2(\Omega)}
+\|v\|_{L^2(\Omega)}
\Bigr), \quad 
\text{ for all }v\in C^\infty(\Omega), \ Bv=0\text{ on }\partial\Omega. 
\end{equation}
The Djirbashian-Caputo derivative in \eqref{ibvp} can be written as 
\[
(\partial_t^\alpha v)(t) = (k^\alpha*v_t)(t)=\int_0^t k^\alpha(t-s)v_t(s)\, ds
\text{ with }k^\alpha(s)= \frac{1}{\Gamma(1-\alpha)} \, s^{-\alpha}
\]
and satisfies (among other properties) the coercivity estimate \cite[Lemma 4.18]{BBB}
\begin{equation}\label{coercivityDC}
v(t)\,(\partial_t^\alpha v)(t)\geq\frac12 \partial_t^\alpha [v^2](t),
\qquad v\in W^{1,1}(0,T)
\end{equation}

\medskip

\noindent
The remainder of this paper is organised as follows.\\
In Section~\ref{subsec:fixedpoint}, we derive a fixed point scheme for simultaneously reconstructing $p(x)$ and $q(x)$ in the observation settings \eqref{obs} or \eqref{obs2}, whose convergence 
is proven in Section~\ref{subsec:conv-fixedpoint} by a contraction argument for sufficiently large $T$ or $T_2$, while $T_1$ is supposed to be close enough to zero.
This also implies local uniqueness of $p$ and $q$ from the given observations.
Numerical reconstruction results are provided in Section~\ref{subsec:reconstructions}.

\section{Fixed point schemes}\label{subsec:fixedpoint}
We evaluate the PDE in \eqref{ibvp} on the observation set $\{T\}\times\Omega$  
\[
\text{res}_i(p,q)=\text{res}_i:=\partial_t^\alpha u_i(p,q)(T)-r_i(T)-\mathbb{L}g_i\stackrel{!}{=}qg_i-pf(g_i) \text{ in }\Omega
\qquad i\in\{1,2\}
\]
where $u_i(p,q)$ denotes the solution of \eqref{ibvp}, and resolve this 
\begin{itemize}
\item for $p$ by multiplying the $i=1$ equation with $g_2$, the $i=2$ equation with $g_1$ and subtracting;
\item for $q$ by multiplying the $i=1$ equation with $f(g_2)$, the $i=2$ equation with $f(g_1)$ and subtracting.
\end{itemize}
With 
\begin{equation}\label{defdet}
\dete^g(x)=\dete\left(\begin{array}{cc} g_1(x) &-f(g_1(x))\\ g_2(x) &-f(g_2(x))\end{array}\right)
=g_2(x)f(g_1(x))-g_1(x)f(g_2(x)), \quad x\in\Omega, 
\end{equation}
the result of this elimination procedure can be written as the fixed point equation
\begin{equation}\label{defcalT}
\left(\begin{array}{c} p\\q\end{array}\right)
=\frac{1}{\dete^g}
\left(\begin{array}{c} 
g_1\,\text{res}_2(p,q)-g_2\,\text{res}_1(p,q)\\
f(g_1)\,\text{res}_2(p,q)-f(g_2)\,\text{res}_1(p,q)
\end{array}\right)
=:\mathcal{T}(p,q).
\end{equation}
This yields a pointwise in space update formula and thus (as opposed to Newton's method) does not require any basis representation of the unknowns $p$ and $q$.

We here study convergence of the fixed point scheme under the premise of existence of steady state solutions $\uinf_i=\lim_{t\to\infty} u_i(\pact,\qact)(t)$ at the actual coefficients $(\pact,\qact)$, so that the limiting determinant \eqref{defdet} is bounded away from zero
\begin{equation}\label{detgbdd}
|\uinf_2(x)f(\uinf_1(x))-\uinf_1(x)f(\uinf_2(x))|\geq c_d>0\quad x\in\Omega.
\end{equation}

Using the defining PDE for $u_i(p^{(k)},q^{(k)})$ to substitute 
\[
\begin{aligned}
&\partial_t^\alpha u_i(p^{(k)},q^{(k)})(T)-r_i(T)\\
&=(\mathbb{L}+q^{(k)})u_i(p^{(k)},q^{(k)})(T)-p^{(k)}f(u_i(p^{(k)},q^{(k)})(T)),
\end{aligned}
\]
the corresponding fixed point iteration $\left(\begin{array}{c} p^{(k+1)}\\q^{(k+1)}\end{array}\right)=\mathcal{T}(p^{(k)},q^{(k)})$ can be written as a discrepancy driven iteration with increment
\[
\left(\begin{array}{c} \uld{p}^{(k)}\\\uld{q}^{(k)}\end{array}\right)
= \frac{1}{\dete^g}
\left(\begin{array}{c} 
g_1\,\text{dsc}_2^{(k)}-g_2\,\text{dsc}_1^{(k)}\\
f(g_1)\,\text{dsc}_2^{(k)}-f(g_2)\,\text{dsc}_1^{(k)}\\
\end{array}\right),
\]
where 
\[
\text{dsc}_i^{(k)}=(\mathbb{L}+q^{(k)})[u_i(p^{(k)},q^{(k)})(T)-g_i]
-p^{(k)}[f(u_i(p^{(k)},q^{(k)})(T))-f(g_i)].
\]
This allows us to equip the method with a stepsize control for choosing $\mu_k>0$ in 
\[
\left(\begin{array}{c} p^{(k+1)}\\q^{(k+1)}\end{array}\right)=
\left(\begin{array}{c} p^{(k)}\\q^{(k)}\end{array}\right)+
\mu_k\left(\begin{array}{c} \uld{p}^{(k)}\\\uld{q}^{(k)}\end{array}\right).
\] 
An elementary one is to impose a monotone decrease of the data misfit 
\begin{equation}\label{mu_mon}
\mu_k=\max\{\theta^\ell\, : \, \sum_{i=1}^2\|u_i(p^{(k)}+\theta^\ell\uld{p},q^{(k)}+\theta^\ell\uld{q})(T)-g_i\|\leq \sum_{i=1}^2\|u_i(p^{(k)},q^{(k)})(T)-g_i\| 
\}
\end{equation}
for some fixed $\theta\in(0,1)$.
To avoid applying second order space derivatives, one can alternatively compute the increment as 
$\left(\begin{array}{c} \uld{p}^{(k)}\\\uld{q}^{(k)}\end{array}\right)
=\mathcal{T}(p^{(k)},q^{(k)})-\left(\begin{array}{c} p^{(k)}\\q^{(k)}\end{array}\right)$
and compute the next iterate with stepsize control from
\[
\left(\begin{array}{c} p^{(k+1)}\\q^{(k+1)}\end{array}\right)
= (1-\mu_k)\left(\begin{array}{c} p^{(k)}\\q^{(k)}\end{array}\right)
+\mu_k\mathcal{T}(p^{(k)},q^{(k)})
\]

We summarise the scheme in Algorithm~\ref{algo1}.

\begin{algorithm}
\caption{(iterative reconstruction of $p$, $q$)}\label{algo1}
\begin{algorithmic}[1]
\STATE choose tol$>0$, $k_{\max}>0$, $p^{(0)}$, $q^{(0)}$; set $k=0$; 
\WHILE{$k=0\ \vee \ \bigl((k<k_{\max})\wedge(\|(\uld{p},\uld{q})\|>\text{tol})\bigr)$}
\STATE solve \eqref{ibvp} with $p=p^{(k)}$, $q=q^{(k)}$ for $u_i=u_i(p,q)$, $i=1,2$;
\STATE evaluate $\text{res}_i=\partial_t^\alpha u_i(T)-r_i(T)-\mathbb{L}g_i$, $i=1,2$;
\STATE compute 
$\left(\begin{array}{c} \uld{p}\\\uld{q}\end{array}\right)
=\frac{1}{\dete^g}
\left(\begin{array}{c} 
g_1\,\text{res}_2-g_2\,\text{res}_1\\
f(g_1)\,\text{res}_2-f(g_2)\,\text{res}_1
\end{array}\right) - \left(\begin{array}{c} p^{(k)}\\q^{(k)}\end{array}\right)
$
\STATE select step size $\mu\in(0,1]$ (e.g. according to \eqref{mu_mon})
\STATE set $\left(\begin{array}{c} p^{(k+1)}\\q^{(k+1)}\end{array}\right)=
\left(\begin{array}{c} p^{(k)}\\q^{(k)}\end{array}\right)+
\mu_k\left(\begin{array}{c} \uld{p}^{(k)}\\\uld{q}^{(k)}\end{array}\right).
$
\STATE $k\ \to \ k+1$
\ENDWHILE
\end{algorithmic}
\end{algorithm}

Due to contractivity of $\mathcal{T}$, Algorithm~\ref{algo1} stops after finitely many steps $k_*$ with an error norm that is bounded by $C\text{tol}$, where $C$ only depends on the contraction constant.

\medskip

In the alternative observation setting of \eqref{obs2} in place of \eqref{obs}
the fixed point iteration reads exactly like \eqref{defcalT}
\begin{equation}\label{defcalT_til}
\left(\begin{array}{c} p\\q\end{array}\right)
=\frac{1}{\dete^g}
\left(\begin{array}{c} 
g_1\,\widetilde{\text{res}}_2(p,q)-g_2\,\widetilde{\text{res}}_1(p,q)\\
f(g_1)\,\widetilde{\text{res}}_2(p,q)-f(g_2)\,\widetilde{\text{res}}_1(p,q)
\end{array}\right)
=:\widetilde{\mathcal{T}}(p,q)
\end{equation}
just upon redefining the residual as
\[
\widetilde{\text{res}}_i=
:=\partial_t^\alpha u(p,q)(T_i)-r(T_i)-\mathbb{L}g_i \text{ in }\Omega.
\qquad i\in\{1,2\}
\]
According to the results in \cite{frac_pq_decay}, we have convergence of $u_i(T,x)\to\uinf(x)$ as $T\to\infty$ and assume, in place of \eqref{detgbdd},
\begin{equation}\label{detgbdd_til}
|\uinf(x)f(u_0(x))-u_0(x)f(\uinf(x))|\geq c_d>0\quad x\in\Omega,
\end{equation}
which allows to show local uniqueness as well as convergence of the fixed point scheme, provided $T_1$ is close to zero and $T_2$ is large enough.
A similar observation setting was considered for a different two-coefficient identification problem in \cite{CenZhou:2024}.

\section{Local uniqueness and convergence of the fixed point scheme}\label{subsec:conv-fixedpoint}
Well-definedness of the forward operator 
\begin{equation}\label{S}
\mathcal{S}_i:(p,q)\mapsto u_i\text{ solving \eqref{ibvp}},
\end{equation}
and thus well-definedness of the fixed point operator $\mathcal{T}$ follows from 
\cite[Corollary 2.1]{frac_pq_decay}.

To prove a contractivity estimate of $\mathcal{T}$ for large enough $T$ at the actual solution $(\pact,\qact)$, we write the difference of values as 
\begin{equation}\label{diffT}
\mathcal{T}(p,q)-\mathcal{T}(\pact,\qact)
=\frac{1}{\dete^g}
\left(\begin{array}{c} 
g_2\,\partial_t^\alpha \hat{u}_1(T)
-g_1\,\partial_t^\alpha \hat{u}_2(T)\\
f(g_2)\,\partial_t^\alpha \hat{u}_1(T)
-f(g_1)\,\partial_t^\alpha \hat{u}_2(T)
\end{array}\right)
\end{equation}
where $\hat{u}_i=u_i(p,q)-u_i(\pact,\qact)$.

We here focus on the practically relevant low regularity setting for the coefficients
\begin{equation}\label{LrrrLsss}
p\in L^\rrr(\Omega), \qquad q\in L^\sss(\Omega).
\end{equation}
While $\sss\geq2$ sufficies for all results here, we need some further restrictions, related to the growth of $f$, on $\rrr$ cf.  \eqref{growth}, \eqref{cond_r_kappa2_hi} below.

Note that $\dete^g(x)=\uact_2(T,x)f(\uact_1(T,x))-\uact_1(T,x)f(\uact_2(T,x))$ remains bounded away from zero for large $T$, provided \eqref{detgbdd} holds, due to the convergence of $\|\uinf_i-\uact_i(T)\|_{H^2(\Omega)}=\|\usim_i(T)\|_{H^2(\Omega)}$ to zero according to 
\cite[Theorem 2.1 or 2.3]{frac_pq_decay}, where estimates of the type
\begin{equation}\label{decay_usim}
\|\uact_i(t)-\uinf_i\|_{H^s(\Omega)}^2\leq C \decayfunc(t) =
C\begin{cases}t^{-\alpha} &\text{ in case }\alpha\in(0,1)\\
e^{-\omega t} &\text{ in case }\alpha=1\end{cases}
\quad\text{ for all }t\in(0,T),
\end{equation} 
for some $C$, $\omega>0$ independent of $T$ and $s\in\{0,1,2\}$ are shown.
Here, $\uinf_i$ are steady state solutions 
\begin{equation}\label{bvp}
\begin{aligned}
-\mathbb{L}\uinf_i&=\qact\uinf_i-\pact f(\uinf_i)+\rinf_i \text{ in }\Omega \\ 
\Binf \uinf_i&=\ainf_i \text{ on }\partial\Omega. 
\end{aligned}
\end{equation}
Given solutions $\uinf_i$ of \eqref{bvp} (which exist under mild conditions see, e.g., \cite[Theorem 4.4]{Troeltzsch2010}), like in \cite{frac_pq_decay} we impose a sign condition on the potential in the linearisation of the steady state equation \eqref{bvp}
\begin{equation}\label{pinfty}
\pinf_i(x):=\pact(x)f'(\uinf_i(x))-\qact(x)\geq\underline{c}>0\quad x\in\Omega.
\end{equation}
Alternatively to \eqref{pinfty}, the assumption 
\begin{equation}\label{boundpqf_int}
\pact(x)f'(\theta \uact(t,x)+(1-\theta)\uinf(x))-\qact(x)\geq \underline{c}>0 \quad \text{ for a.e. }x\in\Omega, \ t\in(0,T)
\end{equation}
is also considered in \cite{frac_pq_decay}.

For the convenience of the reader, we shortly recap the convergence results from \cite{frac_pq_decay}. (In doing so, we focus on the high state regularity setting  and skip the low state regularity results from \cite{frac_pq_decay}.)
The conditions additional to \eqref{pinfty} or \eqref{boundpqf_int} needed for this purpose are as follows. 
\begin{itemize}
\item $f\in C^2(\mathbb{R})$ with a growth condition
\begin{equation}\label{growth}
|f''(\xi)|\leq c_2+C_2|\xi|^{\kappa_2}, \quad
\xi\in\mathbb{R},
\end{equation}
with $C_2,\,c_2>0$;
\item a related requirement on the summability index $\rrr$ of the function $p$
\begin{equation}\label{cond_r_kappa2_hi}
\begin{aligned}
&d=1:\quad&& \rrr\geq2 \quad&& \kappa_2\in[0,\infty]\\
&d=2:\quad&& \rrr>2 \quad&& \kappa_2\in[0,\infty)\\
&d=3:\quad&& \rrr\geq6 \quad&& \kappa_2\in[0,2-\tfrac{6}{\rrr}],
\end{aligned}
\end{equation}
\item a decay condition on $r-\rinf$ 
\begin{equation}\label{decay_r}
\|r(t)-\rinf\|_{L^2(\Omega)}^2\leq C_r 
\decayfunc(t)
\quad t>0
\end{equation} 
with $\decayfunc$ as in \eqref{decay_usim}, $\omega=c_1>0$. 
\end{itemize}
For simplicity, we restrict the setting to Dirichlet or Neumann conditions with time constant data, that is 
\begin{equation}\label{bndy_const}
a(t,x)=\ainf(x), \quad B u=\Binf u=\partial_\nu u\text{ or }B u=\Binf u=u, 
\end{equation}

\begin{theorem}[Theorem 2.1 in \cite{frac_pq_decay}]\label{thm:decay_usim}
For $u_0\in H^1(\Omega)$, $f\in C^2(\mathbb{R})$, under  conditions \eqref{f0fprime0} \eqref{ellipticityL}, \eqref{ellipticregularityL},  
\eqref{pinfty}, \eqref{growth}, \eqref{cond_r_kappa2_hi}, 
$r\in L^\infty(0,\infty;L^2(\Omega))$, \eqref{decay_r}, \eqref{bndy_const},
there exist $\rho_0>0$, $C>0$, $0<c_0<c_1$ independent of $T$ such that $\usim=u-\uinf$ (where $u$ solves \eqref{ibvp} and $\uinf$ solves \eqref{bvp}) satisfies the bound
\begin{equation}\label{bound_usim}
\begin{aligned}
&\frac{c}{2}\, \left(k^{1-\alpha}*\|u-\uinf\|_{H^2(\Omega)}^2\right)(t)
+\|u(t)-\uinf\|_{H^1(\Omega)}^2\\
&\leq \|u(0)-\uinf\|_{H^1(\Omega)}^2 + \frac{C_r}{\Gamma(\alpha)\,\alpha(1-\alpha)} 
\end{aligned}
\end{equation} 
as well as the decay estimate
\eqref{decay_usim}
provided 
\begin{equation}\label{smallness_init}
\|u(0)-\uinf\|_{H^1(\Omega)}<\rho_0. 
\end{equation}
In case 
$T=\infty$, we also have 
\begin{equation}\label{decay_usim_Tauber}
\|\partial_t^\alpha u(t)\|_{L^2(\Omega)}^2+\|u(t)-\uinf\|_{H^2(\Omega)}^2\lesssim C t^{-\alpha}
\quad \text{ as }t\to\infty.
\end{equation} 
Here $\eta(t) \lesssim C\,t^{-\gamma}$ as $t\to\infty$ means 
$\limsup_{t\to\infty} \eta(t)\, t^\gamma\leq C$.
\end{theorem}

\begin{theorem}[Theorem 2.2 in \cite{frac_pq_decay}]\label{prop:decay_usim_t} 
Let the conditions of Theorem~\ref{thm:decay_usim} and additionally 
\begin{equation}\label{utinit}
u_t(0)\in L^2(\Omega), \quad u_t(t_0)\in H^1(\Omega) \ \text{ for some }t_0\geq0
\end{equation}
be satisfied.
\begin{itemize}
\item[(a)] If $T=\infty$, $\alpha=1$ and 
$r_t\in L^2(0,\infty;L^2(\Omega))$, with $\|r_t(t)\|_{L^2(\Omega)}\to0$ as $t\to\infty$,
then  $\|u_t(t)\|_{H^1(\Omega)}\to0$ as $t\to\infty$.
\item[(b)] If $T=\infty$, $\alpha\in(0,1]$ and $\|r_t(t)\|_{L^2(\Omega)}^2=O(t^{-\beta})$ then there exists $C>0$ independent of $t_0$, such that
\[
\|u_t(t)\|_{H^2(\Omega)}^2\lesssim C (t-t_0)^{-\min\{\alpha,\beta\}}.
\]
\item[(c)] If 
$\|r_t(t)\|_{L^2(\Omega)}^2\leq C_r 
\decayfunc(t)$
for some $\omega:=c_1>0$, 
then there exist $C,\,\omega>0$ independent of $T$, $t_0$, such that 
\begin{equation}\label{expdecay_u_t}
\|u_t(t)\|_{H^2(\Omega)}^2\leq C 
\decayfunc(t-t_0)
\quad t\in(t_0,T)
\end{equation} 
\end{itemize}
\end{theorem}

\begin{theorem}[Theorem 2.3 in \cite{frac_pq_decay}]\label{thm:decay_usim_nosmallness}
The results of Theorems~\ref{thm:decay_usim}, \ref{prop:decay_usim_t} remain valid if \eqref{pinfty} is replaced by \eqref{boundpqf_int} and the smallness condition \eqref{smallness_init} on $\uact_i(0)-\uinf_i$ is dropped.
\end{theorem}
 
According to \cite[Remark 2.1]{frac_pq_decay}, under the conditions of Theorem~\ref{prop:decay_usim_t} (d) (or the corresponding part of Theorem~\ref{thm:decay_usim_nosmallness}), also decay of 
\[
\|\partial_t^\alpha u(t)\|_{H^2(\Omega)}
\leq \tilde{C} t^{1-3\alpha/2}
\]
holds for $\alpha\in(\frac23,1)$.

\medskip

Returning to the contractivity proof for $\mathcal{T}$, cf., \eqref{diffT}, our task is thus to estimate $\partial_t^\alpha \hat{u}_i(T)$ appearing in \eqref{diffT} in terms of a small multiple of $\hat{p}$, $\hat{q}$, where we use the decay of $\partial_t^\alpha \hat{u}_i$ with time in order to make the factor small.  
Here we estimate the decay of $\partial_t^\alpha \hat{u}_i(t)$ by means of energy methods. 
Note that the approach from \cite{BB2} using explicit solution representations via Mittag-Leffler functions does not seem to be amenable here, due to the nonlinearity/time dependence of PDE coefficients.


\medskip

We can make use of the decay estimates of 
Theorems~\ref{thm:decay_usim}, \ref{prop:decay_usim_t}, \ref{thm:decay_usim_nosmallness}
with the replacements
\begin{equation}\label{replacements}
\text{$p$, $q$, $\uinf$, $u$, replaced by $\pact$, $\qact$, $\uinf_i$, $\uact_i$};
\end{equation}
That is, we assume existence of steady states for the actual coefficients $\pact$, $\qact$, satisfying
\[
\begin{aligned}
-\mathbb{L}\uinf_i&=\qact\uinf_i-\pact f(\uinf_i)+\rinf_i \text{ in }\Omega \\ 
\Binf \uinf_i&=\ainf_i \text{ on }\partial\Omega. 
\end{aligned}
\]
$i\in\{1,2\}$, but not necessarily for other $p$, $q$ (such as the iterates generated by the fixed point scheme).
 
To this end, in the spirit of 
Theorem~\ref{thm:decay_usim}, Proposition~\ref{prop:decay_usim_t}, 
we consider estimates relying on a positivity assumption \eqref{pinfty} only at the steady states $\uinf_i(x)=\lim_{t\to\infty}\uact_i(t,x)$ 
and on decay of the difference $\usim_i(t,x)=\uact_i(t,x)-\uinf_i(x)$ according to 
\cite[Theorems~2.1,2.2]{frac_pq_decay}.
The difference $\hat{u}_i=u_i(p,q)-u_i(\pact,\qact)=u_i-\uact_i$ and its time derivative of order $\altil\in\{\alpha,1\}$ satisfies
\begin{equation}\label{ibvp_uhat_inf}
\begin{aligned}
\partial_t^\alpha \hat{u}_i+(-\mathbb{L}+\pinf_i)\hat{u}_i&+\text{rest}_{0,i}=\hat{q}(\uact_i+\hat{u}_i)-\hat{p}f(\uact_i+\hat{u}_i) \text{ in }(0,T)\times\Omega, \\ 
B_i \hat{u}_i&=0 \text{ on }(0,T)\times\partial\Omega, \\
\hat{u}_i&=0\text{ in }\{0\}\times\Omega.
\end{aligned}
\end{equation}
and
\begin{equation}\label{ibvp_uhat_t_inf}
\begin{aligned}
\partial_t^\alpha \,\partial_t^\altil \hat{u}_i+ (-\mathbb{L}+\pinf_i)\partial_t^\altil \hat{u}_i 
&\,+\text{rest}_i\\
=&\,\hat{q}\partial_t^\altil (\uact_i+\hat{u}_i)-\hat{p}\,\partial_t^\altil[f(\uact_i+\hat{u}_i)] 
\text{ in }(0,T)\times\Omega, \\ 
B_i \partial_t^\altil\hat{u}_i=&\,0 \text{ on }(0,T)\times\partial\Omega, \\
\end{aligned}
\end{equation}
with 
\[
\begin{aligned}
&\tilde{\bar{f}}_i(t,x) = \int_0^1 f'(\uact_i(t,x)+\theta\hat{u}_i(t,x))\,d\theta-f'(\uinf_i(x))\\
&\phantom{\tilde{\bar{f}}_i(t,x)}= \int_0^1\int_0^1f''(\uinf_i(x)+\theta\usim_i(t,x)+s\theta\hat{u}_i(t,x))(\usim_i(t,x)+s\hat{u}_i(t,x))\, ds\,d\theta\\
&\hat{\bar{f}}_i(t,x) = f'(u_i(t,x))-f'(\uact(t,x))
=\int_0^1f''(\uact_i(t,x)+\theta\hat{u}_i)\,d\theta\,\hat{u}_i(t,x)\\
&\bar{\bar{f}}_i(t,x) = f'(\uact_i(t,x))-f'(\uinf_i(x))
=\int_0^1f''(\uinf_i(x)+\theta\usim_i(t,x))\,d\theta\,\usim_i(t,x)\\
&\text{rest}_{0,i}=\pact\,\tilde{\bar{f}}_i\,\hat{u}_i\\
&\text{rest}_i
=\pact\,k^\altil*\left(\hat{\bar{f}}_i\, u_{i,t} + \bar{\bar{f}}_i\, \hat{u}_{i,t}\right) =\pact\,k^\altil*\left(\hat{\bar{f}}_i\, \uact_{i,t} + (\hat{\bar{f}}_i+\bar{\bar{f}}_i)\, \hat{u}_{i,t}\right)+\chi_{\altil=1}k^\alpha(t)\hat{u}_{i,t}(0). 
\end{aligned}
\]
Note that initial conditions on $\partial_t^\altil\hat{u}_i$ can be bootstrapped from the PDE only in case $\altil=\alpha$, in which we have $\partial_t^\alpha\hat{u}_i(0)=\hat{q}u_{0,i}-\hat{p}f(u_{0,i})$.

To derive the equations \eqref{ibvp_uhat_inf}, \eqref{ibvp_uhat_t_inf} we have used the identities
\[
\begin{aligned}
&p f(u_i)-\pact f(\uact_i)-\pact f'(\uinf_i)(u_i-\uact_i)-(p-\pact)f(u_i)\\
&=\pact\bigl(f(u_i)-f(\uact_i)-f'(\uinf_i)(u_i-\uact_i)\bigr)
\end{aligned}
\]
\[
\begin{aligned}
&p \partial_t^\altil f(u_i)-\pact \partial_t^\altil f(\uact_i)-\pact f'(\uinf_i)\partial_t^\altil(u_i-\uact_i)-(p-\pact)\partial_t^\altil f(u_i)\\
&=\pact k^\altil*\bigl(f'(u_i)u_{i,t}-f'(\uact_i)\uact_{i,t}-f'(\uinf_i)(u_{i,t}-\uact_{i,t})\\
&=\pact k^\altil*\bigl((f'(u_i)-f'(\uact_i))u_{i,t}+(f'(\uact_i)-f'(\uinf_i))(u_{i,t}-\uact_{i,t})\bigr)
\end{aligned}
\]

\medskip

Analogously to 
the proofs of \cite[Theorems~2.1,2.2]{frac_pq_decay},
we test 
\eqref{ibvp_uhat_inf} with $\mu\hat{u}_i-\mathbb{L}\hat{u}_i$ and 
\eqref{ibvp_uhat_t_inf} with $\mu\partial_t^\altil \hat{u}_i-\mathbb{L}\partial_t^\altil \hat{u}_i$, 
to obtain, using \eqref{ellipticityL}, \eqref{ellipticregularityL}, \eqref{coercivityDC},
with
\begin{equation}\label{cstar}
c_{i,*}^2=\min\left\{\frac{1+2\|\pinf_i\|_{L^\infty(\Omega)}^2}{2\underline{c}},\, \underline{c}_D\right\}, \qquad
C_{i,*}^2=\left(\frac{1+2\|\pinf_i\|_{L^\infty(\Omega)}^2}{2\underline{c}^2}+2\right)
\end{equation}
that
\begin{equation}\label{enest_uhat}
\begin{aligned}
&c_{i,*}^2\, \partial_t^\alpha\|\hat{u}_i(t)\|_{H^1(\Omega)}^2
+\frac{1}{2(C_{\text{ell}}^\Omega)^2}\, \|\hat{u}_i(t)\|_{H^2(\Omega)}^2
\\
&\leq 
C_{i,*}^2\,
\|\hat{p}f(\uact_i+\hat{u}_i)-\hat{q}(\uact_i+\hat{u}_i)
+\pact\,\tilde{\bar{f}}_i\,\hat{u}_i\|_{L^2(\Omega)}^2,
\end{aligned}
\end{equation}
and
\begin{equation}\label{enest_uhat_t}
\begin{aligned}
&c_{i,*}^2\, \partial_t^\alpha\|\partial_t^\altil\hat{u}_i(t)\|_{H^1(\Omega)}^2
+\frac{1}{2(C_{\text{ell}}^\Omega)^2}\, \|\partial_t^\altil\hat{u}_i(t)\|_{H^2(\Omega)}^2
\\
&\leq 
C_{i,*}^2\,
\|\hat{p}\,\partial_t^\altil[f(\uact_i+\hat{u}_i)]-\hat{q}\partial_t^\altil (\uact_i+\hat{u}_i)\\
&
\qquad\quad
+\pact\,k^\altil*\left(\hat{\bar{f}}_i\, \uact_{i,t} + (\hat{\bar{f}}_i+\bar{\bar{f}}_i)\, \hat{u}_{i,t}\right)
+\chi_{\altil=1}k^\alpha(t)\hat{u}_{i,t}(0)\|_{L^2(\Omega)}^2,
\end{aligned}
\end{equation}
where the $f$ terms will be further estimated by means of the growth estimate \eqref{growth}. 

\medskip

We first derive boundedness of $\hat{u}(t)$ in terms of $\hat{p}$ and $\hat{q}$ from \eqref{enest_uhat}.
Note that since the right hand side $\hat{q}(\uact_i(t)+\hat{u}_i(t))-\hat{p}f(\uact_i(t)+\hat{u}_i(t))$ 
of \eqref{ibvp_uhat_inf} cannot be expected to tend to zero, since $\uact_i(t)$ is assumed to tend to a nonzero steady state $\uinf_i$, we do not obtain decay of $\hat{u}(t)$ itself, but only of its time derivative. 

The right hand side in \eqref{enest_uhat} under the growth condition \eqref{growth} and condition \eqref{cond_r_kappa2_hi} on $\rrr$ can be estimated as follows.
We have 
\[
\begin{aligned}
&\|\pact\,\tilde{\bar{f}}_i(t)\,\hat{u}_i(t)\|_{L^2(\Omega)}
\leq \|\pact\|_{L^{\rrr}(\Omega)}\,\|\tilde{\bar{f}}_i(t)\|_{L^{2\rrr/(\rrr-2)}(\Omega)}\,\|\hat{u}_i(t)\|_{L^\infty(\Omega)}^2\\
&\leq \|\pact\|_{L^{\rrr}(\Omega)}
\Bigl(\tilde{c}_{2,i}+\|\usim_i(t)\|_{H^1(\Omega)}^{\kappa_2}+\|\hat{u}_i(t)\|_{H^1(\Omega)}^{\kappa_2}\Bigr)
\bigl(\|\usim_i(t)\|_{H^1(\Omega)}+\|\hat{u}_i(t)\|_{H^1(\Omega)}\bigr)
\|\hat{u}_i(t)\|_{H^2(\Omega)}
\end{aligned}
\]
with
\[
\begin{aligned}
&\tilde{c}_{2,i}
=\Bigl(c_2C_{H^1\to L^{2\rrr/(\rrr-2)}}^\Omega+C_2C_{\kappa_2}C_{H^1\to L^{(\kappa_2+1)2\rrr/(\rrr-2)}}^\Omega\|\uinf_i\|_{L^{(\kappa_2+1)2\rrr/(\rrr-2)}(\Omega)}^{\kappa_2}\Bigr)
 C_{H^2\to L^\infty}^\Omega\\
&\tilde{C}_{2,i}=C_2 C_{\kappa_2} C_{\kappa_2+1} 
(C_{H^1\to L^{(\kappa_2+1)2\rrr/(\rrr-2)}}^\Omega)^{\kappa_2+1} C_{H^2\to L^\infty}^\Omega,
\end{aligned}
\]
where we have estimated
\[
\begin{aligned}
&\|\tilde{\bar{f}}_i(t)\|_{L^{2\rrr/(\rrr-2)}(\Omega)}
\leq c_2(\|\usim_i(t)\|_{L^{2\rrr/(\rrr-2)}(\Omega)}+\|\hat{u}_i(t)\|_{L^{2\rrr/(\rrr-2)}(\Omega)})\\
&+C_2C_{\kappa_2}\bigl(\|\uinf_i\|_{L^{(\kappa_2+1)2\rrr/(\rrr-2)}(\Omega)}^{\kappa_2}+
(\|\usim_i(t)\|_{L^{(\kappa_2+1)2\rrr/(\rrr-2)}(\Omega)}+\|\hat{u}_i(t)\|_{L^{(\kappa_2+1)2\rrr/(\rrr-2)}(\Omega)})^{\kappa_2}\Bigr)\\
&\hspace*{3cm}(\|\usim_i(t)\|_{L^{(\kappa_2+1)2\rrr/(\rrr-2)}(\Omega)}+\|\hat{u}_i(t)\|_{L^{(\kappa_2+1)2\rrr/(\rrr-2)}(\Omega)})
\end{aligned}
\]
and used
\[
(a+b)^{\kappa_2}\leq C_{\kappa_2}(a^{\kappa_2}+b^{\kappa_2}), \quad a,\, b\,\geq0 \quad\text{ for }C_{\kappa_2}=\max\{1,2^{\kappa_2-1}\}.
\]

For the other right hand side term under the assumption \eqref{f0fprime0}
we have 
\[
|f(\xi)|\leq \frac12(c_2+C_2|\xi|^{\kappa_2})|\xi|^2
\]
and with $\uact_i=\uinf_i+\usim_i$ obtain  
\[
\begin{aligned}
&\|\hat{p}f(\uact_i(t)+\hat{u}_i(t))\|_{L^2(\Omega)}\\
&\leq \|\hat{p}\|_{L^{\rrr}(\Omega)}\Bigl(
\bar{c}_{2,i}+\bar{C}_{2,i} \bigl(\|\uinf_i\|_{H^1(\Omega)}^{\kappa_2}+\|\usim_i(t)\|_{H^1(\Omega)}^{\kappa_2}+\|\hat{u}_i(t)\|_{H^1(\Omega)}^{\kappa_2}\bigr)\Bigr)\\
&\qquad\,\bigl(\|\uinf_i\|_{H^1(\Omega)}+\|\usim_i(t)\|_{H^1(\Omega)}+\|\hat{u}_i(t)\|_{H^1(\Omega)}\bigr)
\,(\|\uinf_i\|_{H^2(\Omega)}+\|\usim_i(t)\|_{H^2(\Omega)}+\|\hat{u}_i(t)\|_{H^2(\Omega)})
\end{aligned}
\]
with
\[
\begin{aligned}
&\bar{c}_{2,i}=
\frac12 c_2C_{H^1\to L^{2\rrr/(\rrr-2)}}^\Omega
C_{H^2\to L^\infty}^\Omega\\
&\bar{C}_{2,i}=
\frac12 C_2 C_{\kappa_2+1} (C_{H^1\to L^{(\kappa_2+1)2\rrr/(\rrr-2)}}^\Omega)^{\kappa_2+1} 
C_{H^2\to L^\infty}^\Omega,
\end{aligned}
\]

Here according to 
\eqref{bound_usim} in Theorem~\ref{thm:decay_usim} 
with $\|\usim_i(0)\|_{H^1(\Omega)}<\rho$ (or according to 
Theorem~\ref{thm:decay_usim_nosmallness} 
without smallness assumption on the initial data) we have boundedness of $\|\usim_i(t)\|_{H^2(\Omega)}^2$ and thus also of $\|\uact_i(t)\|_{H^2(\Omega)}$ (but not necessarily decay to zero).
Moving all terms containing $\|\hat{u}_i(t)\|_{H^2(\Omega)}$ to the left hand side in \eqref{enest_uhat}, we arrive at
\begin{equation}\label{enest_uhat_1}
\begin{aligned}
&
c_{i,*}^2\partial_t^\alpha\|\hat{u}_i(t)\|_{H^1(\Omega)}^2
+\Bigl(\frac{1}{2(C_{\text{ell}}^\Omega)^2}-C_{i,*}^2(D_1(t)+D_2(t)+D_3(t))^2\Bigr)\, \|\hat{u}_i(t)\|_{H^2(\Omega)}^2
\\
&\leq 
C_{i,*}^2\,D_2(t)^2\Bigl(\|\uinf_i\|_{H^2(\Omega)}+\|\usim_i(t)\|_{H^2(\Omega)}\Bigr)^2
\end{aligned}
\end{equation}
with
\begin{equation}\label{D1D2}
\begin{aligned}
&D_1(t)=
\|\pact\|_{L^{\rrr}(\Omega)}
\Bigl(\tilde{c}_{2,i}+\|\usim_i(t)\|_{H^1(\Omega)}^{\kappa_2}+\|\hat{u}_i(t)\|_{H^1(\Omega)}^{\kappa_2}\Bigr)
\bigl(\|\usim_i(t)\|_{H^1(\Omega)}+\|\hat{u}_i(t)\|_{H^1(\Omega)}\bigr)
\\
&D_2(t)=
\|\hat{p}\|_{L^{\rrr}(\Omega)}
\Bigl(\bar{c}_{2,i}+\bar{C}_{2,i} \bigl(\|\uinf_i\|_{H^1(\Omega)}^{\kappa_2}+\|\usim_i(t)\|_{H^1(\Omega)}^{\kappa_2}+\|\hat{u}_i(t)\|_{H^1(\Omega)}^{\kappa_2}\bigr)\Bigr)\\
&\phantom{D_2(t)=C_{i,*}\|\hat{p}\|_{L^{\rrr}(\Omega)}}\
\bigl(\|\uinf_i\|_{H^1(\Omega)}+\|\usim_i(t)\|_{H^1(\Omega)}+\|\hat{u}_i(t)\|_{H^1(\Omega)}\bigr)
\\
&D_3(t)=\|\hat{q}\|_{L^2(\Omega)}\,C_{H^2\to L^\infty}^\Omega\\
\end{aligned}
\end{equation}
To guarantee positivity (and boundedness away from zero) of the factor multiplying $\|\hat{u}_i(t)\|_{H^2(\Omega)}$ in \eqref{enest_uhat_1}, we assume 
\[
\|\hat{p}\|_{L^{\rrr}(\Omega)}^2+\|\hat{q}\|_{L^2(\Omega)}\leq\rho
\]
with $\rho>0$ chosen small enough, use the fact that $\hat{u}_i(0)=0$ 
and $\|\usim_i(0)\|_{H^1(\Omega)}\leq\rho_0$ small to achieve
\[
C_{i,*}(D_1(0)+D_2(0)+D_3(0))
\leq \frac{1}{2C_{\text{ell}}^\Omega}
\] 
and apply a barrier argument (like the one in the proof of \cite[Theorem 2.1]{frac_pq_decay}) to conclude that $D_1(t)+D_2(t)+D_3(t)\leq \frac{1}{2C_{\text{ell}}^\Omega}$ for all $t>0$.

We can thus use \cite[Lemma 3.2]{frac_pq_decay} to conclude
\begin{equation}\label{bound_uhat}
\|\hat{u}_i\|_{H^1(\Omega)}^2(t)+\|\hat{u}_i(t)\|_{H^2(\Omega)}^2
\leq
C \bigl(\|\hat{p}\|_{L^\rrr(\Omega)}^2+\|\hat{q}\|_{L^2(\Omega)}^2\bigr).
\end{equation} 
for some $C>0$ depending only on $\Omega$, $c_2$, $C_2$, $\kappa_2$ and $\rrr$.

\medskip

To obtain a decay estimate on $\partial_t^\altil \hat{u}_i(t)$ from 
\eqref{enest_uhat_t}, we need commutator estimates of the terms
\[
(k^\altil*(h\,v_t)-h\,k^\altil*v_t)(t) \quad \text{ for }\ 
h\in \{\hat{\bar{f}}_i(x),\, \bar{\bar{f}}_i(x), f'(\uact_i+\hat{u}_i)(x)\}, \quad
v\in \{\uact_i(x),\,\hat{u}_i(x)\},
\]
which in case $\altil<1$ are not available in a local form that would allow us to prove decay with time.
We thus focus on the case $\altil=1$, in which the right hand side of \eqref{enest_uhat_t}
simplifies and under conditions \eqref{f0fprime0}, \eqref{growth}, \eqref{cond_r_kappa2_hi} can be estimated by
\[
\begin{aligned}
&\|(\hat{p}\,f'(\uact_i+\hat{u}_i)-\hat{q}) (\uact_{i,t}+\hat{u}_{i,t})
+\pact\,\left(\hat{\bar{f}}_i\, \uact_{i,t} + (\hat{\bar{f}}_i+\bar{\bar{f}}_i)\, \hat{u}_{i,t}\right)\|_{L^2(\Omega)}^2\\
&\leq (D_1(t)+D_2(t)+D_3(t))\|\hat{u}_{i,t}(t)\|_{H^2(\Omega)}
+(D_2(t)+D_3(t)) \|\usim_{i,t}(t)\|_{H^2(\Omega)}
\end{aligned}
\]
with 
$D_1$, $D_2$, $D_3$ as in \eqref{D1D2}. 
Here we already have $C_{i,*}(D_1(t)+D_2(t)+D_3(t))\leq \frac{1}{2C_{\text{ell}}^\Omega}$ from above and $\|\usim_{i,t}(t)\|_{H^2(\Omega)}\leq C\decayfunc(t)$ due to Theorem~\ref{prop:decay_usim_t}.
Thus, from \eqref{enest_uhat_t}, assuming $\hat{u}_{i,t}(0)\in L^2(\Omega)$ to hold, we conclude an estimate of the form 
\begin{equation}\label{enest_uhat_t_1}
\begin{aligned}
&\partial_t^\alpha\|\partial_t^\altil\hat{u}_i(t)\|_{H^1(\Omega)}^2
+\|\partial_t^\altil\hat{u}_i(t)\|_{H^2(\Omega)}^2
\\
&\leq \tilde{C}\decayfunc(t)
(\|\hat{p}\|_{L^\rrr(\Omega)}^2+\|\hat{q}\|_{L^2(\Omega)}^2 \quad t>0
\end{aligned}
\end{equation}
to hold. 
Thus, using \cite[Lemma 3.2]{frac_pq_decay} and assuming $\hat{u}_{i,t}(0)\in L^2(\Omega)$, $\hat{u}_{i,t}(t_0)\in H^1(\Omega)$ for some $t_0\geq0$ we conclude
\begin{equation}\label{bound_uhat_t}
\|\hat{u}_i\|_{H^1(\Omega)}^2(t)+\|\hat{u}_i(t)\|_{H^2(\Omega)}^2
\leq 
C \decayfunc(t) \bigl(\|\hat{p}\|_{L^\rrr(\Omega)}^2+\|\hat{q}\|_{L^2(\Omega)}^2\bigr)
\quad t>t_0
\end{equation} 
for some $C>0$ depending only on $\Omega$, $c_2$, $C_2$, $\kappa_2$, $\rrr$, $\|\hat{u}_{i,t}(0)\|_{L^2(\Omega)}$, $\|\hat{u}_{i,t}(t_0)\|_{H^1(\Omega)}$.

\begin{theorem}\label{thm:decay_uhat}
Under the conditions of Theorem~\ref{prop:decay_usim_t} (d)
with the replacements \eqref{replacements}, 
as well as \eqref{utinit} for $u=u_i(p,q)$, $u=\uact_i$, 
there exists $\rho>0$ such that for all $(p,q)\in\mathcal{B}_\rho(\pact,\qact)^{L^{\rrr}\times L^{\sss}}$
the time derivative of $\hat{u}_i=u_i(p,q)-u_i(\pact,\qact)$ satisfies
\begin{equation}\label{expdecay_uhat_t}
\|\hat{u}_i(t)\|_{H^2(\Omega)}^2+\|\partial_t \hat{u}_i(t)\|_{H^1(\Omega)}^2\leq C 
\decayfunc(t)
\quad\text{ for all }t\in(0,T]
\end{equation} 
for some $C>0$.
\end{theorem}

\bigskip

The bounds in Theorem~\ref{thm:decay_uhat} still depend on $p$ and $q$ through the norms of 
$\|u_{i,t}(0)\|_{L^2(\Omega)}$, $\|u_{i,t}(t_0)\|_{H^1(\Omega)}$ in \eqref{utinit} for $u=u_i(p,q)$, $u=\uact_i$.
We get rid of this dependence by only imposing this assumption on the actual states $\uact_i$ 
\begin{equation}\label{utinit_act}
\uact_{i,t}(0)\in L^2(\Omega), \quad \uact_{i,t}(t_0)\in H^1(\Omega) \ \text{ for some }t_0\geq0
\end{equation}
and rather than determining $u_i(p,q)$ as a PDE solution $\text{PDE}_i[p,q,u(p,q)]=0$
with  
\[
\text{PDE}_i[p,q,u]:=\partial_t^\alpha u-\mathbb{L}u+pf(u)-qu+r_i
\]
define it by the constrained minimisation problem
\begin{equation}\label{uipq_min}
\begin{aligned}
u_i(p,q)\in\ &\text{argmin}_{u\in H^\alpha(0,T;L^2(\Omega))\cap L^2(0,T;H^2(\Omega))} 
J_{\text{PDE}_i}(u)\\
&\text{ such that }u(0)=u_0, \quad \|u_t(0)\|_{L^2(\Omega)}^2+\|u_t(t_0)\|_{H^1(\Omega)}^2
\leq \bar{\varrho}_i 
\end{aligned}
\end{equation}
where $\bar{\varrho}_i\geq \|\uact_{i,t}(0)\|_{L^2(\Omega)}^2+\|\uact_{i,t}(t_0)\|_{H^1(\Omega)}^2$ 
\[
\begin{aligned}
J_{\text{PDE}_i}(u):= 
\int_0^T\decayfunc(t)^{-1}\Bigl(\|\text{PDE}_i[p,q,u](t)\|_{L^2(\Omega)}^2+\|\partial_t\text{PDE}_i[p,q,u](t)\|_{L^2(\Omega)}^2\Bigr)\, dt.
\end{aligned}
\]
This, due to minimality of $u_i(p,q)$ and feasibility of $\uact_i$ provides us with the estimate 
\[
\begin{aligned}
&J_{\text{PDE}_i}(u_i(p,q))\leq J_{\text{PDE}_i}(\uact_i)\\
&=\int_0^T\decayfunc(t)^{-1}\Bigl(\|\text{PDE}_i[p,q,\uact_i](t)-\text{PDE}[\pact,\qact,\uact_i](t)\|_{L^2(\Omega)}^2\\
&\phantom{\int_0^T\decayfunc(t)^{-1}\Bigl(}+\|\partial_t\bigl(\text{PDE}_i[p,q,\uact_i]-\text{PDE}_i[\pact,\qact,\uact_i]\bigr)(t)\|_{L^2(\Omega)}^2\Bigr)\, dt
\end{aligned}
\]
where we have used the fact that $\text{PDE}_i[\pact,\qact,\uact_i](t)=0$.
In here
\[
\text{PDE}_i[p,q,\uact_i]-\text{PDE}_i[\pact,\qact,\uact_i]
=\hat{p}f(\uact_i)-\hat{q}\uact_i
\]
and therefore altogether we obtain
\[
\begin{aligned}
&\int_0^T\decayfunc(t)^{-1}\Bigl(\|\text{PDE}_i[p,q,u_i(p,q)](t)\|_{L^2(\Omega)}^2+\|\partial_t\text{PDE}[p,q,u_i(p,q)](t)\|_{L^2(\Omega)}^2\Bigr)\, dt\\
&\leq 
2\|\hat{p}\|_{L^\rrr(\Omega)}^2
\int_0^T\decayfunc(t)^{-1}\Bigl(\|f(\uact_i)(t)\|_{L^{2\rrr/(\rrr-2)}(\Omega)}^2
+\|f'(\uact_i)(t)\uact_{i,t}(t)\|_{L^{2\rrr/(\rrr-2)}(\Omega)}^2\Bigr)\, dt\\
&\qquad+2\|\hat{q}\|_{L^2(\Omega)}^2
\int_0^T\decayfunc(t)^{-1}\Bigl(\|\uact_i(t)\|_{L^\infty(\Omega)}^2
+\|\uact_{i,t}(t)\|_{L^\infty(\Omega)}^2\Bigr)\, dt
\end{aligned}
\]
As a consequence, only the right hand side of \eqref{enest_uhat_1} changes by addition of a term of the form $\tilde{C}\decayfunc(t)(\|\hat{p}\|_{L^\rrr(\Omega)}^2+\|\hat{q}\|_{L^2(\Omega)}^2)$
and therefore the estimate \eqref{bound_uhat} on $\hat{u}$ remains valid. 

Likewise, \eqref{enest_uhat_t_1} is preserved with a possibly modified (but completely $(p,q)$ independent) constant and therefore \eqref{bound_uhat_t} holds, with a constant $C$ whose dependence on $\|\hat{u}_{i,t}(0)\|_{L^2(\Omega)}$, $\|\hat{u}_{i,t}(t_0)\|_{H^1(\Omega)}$ can be 
replaced by dependence on the uniform 
bound $\bar{\varrho}_i$.

Note that the modification \eqref{uipq_min} only affects computation of $u_i$ in line 3 of Algorithm~\ref{algo1}. The pointwise in space update of the coefficient $(p,q)$ in lines 4--7 of Algorithm~\ref{algo1} remains.

\bigskip

Now we are in the position to prove contractivity of the fixed point operator $\mathcal{T}$ defined by \eqref{defcalT}.

In order to guarantee that iterates remain in the parameter set $\mathcal{M}$ by contractivity, we define it to be a ball with the actual parameters as centres 
\begin{equation}\label{parspace}
(p,q)\in\mathcal{M}:=\mathcal{B}_\rho(\pact,\qact)^{L^{\rrr}\times L^{\sss}}\subseteq L^{\rrr}(\Omega)\times L^{\sss}(\Omega) 
\end{equation}
with $\rho>0$ small enough (but independent of $T$).
Contractivity of the fixed point operator $\mathcal{T}$ will thus imply that it is a self-mapping on $\mathcal{M}$.
Another consequence of contractivity will be local uniqueness of $(p,q)$, that is, the fact that $(\pact,\qact)$ is the only solution to the inverse problem in $\mathcal{M}$.

It is important to note that all constants above are independent of $T$. 

We have thus proven the following contractivity result.
\begin{theorem}\label{thm:contractivity}
For $\alpha\in(2/3,1]$, under the conditions of Theorem~\ref{prop:decay_usim_t} (d) with the replacements \eqref{replacements}, as well as \eqref{detgbdd}, \eqref{utinit_act},
there exist constants $C$, $\omega$, $\rho>0$ independent of $T$
such that  
\begin{equation}\label{normdiffT}
\begin{aligned}
&\|\mathcal{T}(p,q)-\mathcal{T}(\pact,\qact)\|_{L^{\rrr}(\Omega)\times L^{\sss}(\Omega)}
\leq C \,
\decayfunc(T-t_0)
\|(p,q)-(\pact,\qact)\|_{L^{\rrr}(\Omega)\times L^{\sss}(\Omega)}\\
&\text{ for all }(p,q)\in\mathcal{B}_\rho(\pact,\qact)^{L^{\rrr}\times L^{\sss}}.
\end{aligned}
\end{equation}
\end{theorem}
\begin{proof}
In \eqref{diffT} we have $g_i=\uact_i(T)$ and can thus estimate 
\[
\begin{aligned}
&\|\mathcal{T}(p,q)-\mathcal{T}(\pact,\qact)\|_{L^{\rrr}(\Omega)\times L^{\sss}(\Omega)}\\
&\leq \frac{1}{\dete^{\uact(T)}}
\max_{i\in\{1,2\}} f_{\max}\bigl(\|\uact_i(T)\|_{L^\infty(\Omega)}\bigr) 
\max_{i\in\{1,2\}} \bigl(\|\hat{u}_{i,t}(T)\|_{L^{\rrr}(\Omega)}+\|\hat{u}_{i,t}(T)\|_{L^{\sss}(\Omega)}\bigr)\\
&\leq \frac{(C_{H^1,L^{\rrr}}^\Omega+C_{H^1,L^{\sss}}^\Omega)}{\dete^{\uinf}+|\dete^{\uact(T)}-\dete^{\uinf}|}  
\max_{i\in\{1,2\}} f_{\max}\bigl(\|\uact_i(T)\|_{L^\infty(\Omega)}\bigr) 
\max_{i\in\{1,2\}} \|\hat{u}_{i,t}(T)\|_{H^1(\Omega)},
\end{aligned}
\]
where we have used the monotonically increasing function 
\[
f_{\max}(\zeta):=\zeta+\sup_{|\xi|\leq \zeta} |f(\xi)|
\quad \zeta\in\mathbb{R}^+,
\]
so that 
$\|v\|_{L^\infty(\Omega)}+\|f(v)\|_{L^\infty(\Omega)}\leq f_{\max}(\|v\|_{L^\infty(\Omega)})$.
Thus, using \eqref{expdecay_uhat_t} as well as \eqref{detgbdd} and 
\[
|\dete^{\uact(T,x)}-\dete^{\uinf(x)}|\leq 
\sum_{i\in\{1,2\}}(1+|\int_0^1 f'(\uinf_i(x)+\theta(\uact_i(T,x)-\uinf_i(x)))|) |\uact_i(T,x)-\uinf_i(x)|
\]
with $\|\uact_i(T)-\uinf_i\|_{L^\infty(\Omega)}\leq C_{H^2\to L^\infty}^\Omega \|\usim_i\|_{H^2(\Omega)}$,
we arrive at the assertion.
\end{proof}

\subsubsection*{Alternative observation setting}
To prove a contractivity estimate like \eqref{diffT} for $\widetilde{\mathcal{T}}$ (cf. \eqref{defcalT_til}) according to  
\begin{equation}\label{diffT_til}
\widetilde{\mathcal{T}}(p,q)-\widetilde{\mathcal{T}}(\pact,\qact)
=\frac{1}{\dete^g}
\left(\begin{array}{c} 
g_2\,\partial_t^\alpha \hat{u}(T_1)
-g_1\,\partial_t^\alpha \hat{u}(T_2)\\
f(g_2)\,\partial_t^\alpha \hat{u}(T_1)
-f(g_1)\,\partial_t^\alpha \hat{u}(T_2)
\end{array}\right)
\end{equation}
where $\hat{u}=u(p,q)-u(\pact,\qact)$,
we consider the limits $T_1\to0$, $T_2\to\infty$. 
For the latter, the results above (Theorem~\ref{thm:decay_uhat}) can be applied and for the former, continuity of $\partial_t^\alpha \hat{u}$ at $t=0$ is relevant.
Due to the known singularity at initial time of solutions to subdiffusion equations, this can only be expected to hold in case $\alpha=1$.
\begin{theorem}\label{thm:contractivity_til}
For $\alpha=1$, $t_0=0$, under the conditions of Theorem~\ref{thm:contractivity} (with the subscript ${}_i$ removed) and \eqref{detgbdd_til}, 
with additionally $r$, $r_t$ continuous at $t=0$ with values in $L^2(\Omega)$,
there exist constants $C$, $\omega$, $\rho>0$ independent of $0<T_1<T_2$
such that  
\begin{equation}\label{normdiffT_til}
\begin{aligned}
&\|\widetilde{\mathcal{T}}(p,q)-\widetilde{\mathcal{T}}(\pact,\qact)\|_{L^{\rrr}(\Omega)\times L^{\sss}(\Omega)}\\
&\leq C \,
\max\{e^{-c_0T_2},\,
\phi(T_1)\} \|(p,q)-(\pact,\qact)\|_{L^{\rrr}(\Omega)\times L^{\sss}(\Omega)}\\
&\text{ for all }(p,q)\in\mathcal{B}_\rho(\pact,\qact)^{L^{\rrr}\times L^{\sss}}.
\end{aligned}
\end{equation}
with a function $\phi$ that 
satisfies $\phi(t)\searrow0$ as $t\searrow0$.  
\end{theorem}

\begin{remark}\label{rem:noisydata}
Since $\triangle g_i$ is used in the reconstruction scheme (as well as in the proof of the above convergence theorems), realistically given noisy measurements, with some noise level $\tilde{\delta}>0$ in $L^2(\Omega)$, first of all have to be subject to some smoothing procedure, so that the filtered version lies in $H^2(\Omega)$, with an $H^2(\Omega)$ noise level $\delta>0$ that is typically larger than the $L^2$ noise level $\tilde{\delta}$.
An inspection of the above convergence proofs shows that the error in the reconstruction will then be $O(\delta)$ after a sufficiently large number $O(|\log(\delta)|)$ of fixed point iterations. In fact, the above convergence results also show that the inverse problem is well-posed when considering the data in the space $H^2(\Omega)$, so that no early stopping is required.
\end{remark}
 
\newpage
\section{Reconstruction results}\label{subsec:reconstructions}

As a parallel to the spatially independent situation \cite{Verhulst:1845} we will take as given an initial value $u_0 = u_0(x)$ 
as well as a source $r=r(t,x)$ 
and make measurements at spatial points for given values of time.
These values will correspond to the spatial grid taken for the direct solver.
However, since the solution $u(x,t)$ is at least twice differentiable in space
a more sparse sample can be taken and extrapolated to the final grid.
Under assumptions of low data error this is quite effective.

The observation setting \eqref{obs} corresponding to the application of two different sources will be referred to as the ``two run'' case. For the alternative situation \eqref{obs2} of observing $u(x,t)$ at two times $t=T_1$ and $t=T_2$
we will also indicate some of the delicate issues that can arise in the
choice of $T_1$ and $T_2$.
In particular, we will demonstrate the differences in reconstructions
between the cases of $T_2-T_1$ being small or large, that is, $T_2$ being
{\it near\/} or {\it far\/} from $T_1$.
In short, while from a practical viewpoint more data measurements are obviously
desirable, our focus is on describing what is feasible to obtain from
the least information that is sufficient to guarantee uniqueness.

In view of the decay rate of the Lipschitz constant of $\mathcal{T}$ to zero, convergence and its speed heavily depends on the length $T$ or $T_2$ of the time interval. While the values we are choosing in our computations below look moderate, they must be seen relative to the diffusion coefficient (which is here set to unity $D\equiv I$), which is coupled to time by a simple rescaling.

As will be demonstrated in this section,
a judicious choice of the two 
sources
as opposed to two time measurements gives in general superior reconstructions of $p(x)$ and $q(x)$.
In addition, if two time values are the prescribed data then
a certain gap in time between measurements is important.
In the reconstructions to be shown we take several subcases depending
on the location/type of later time data.
Specifically, we shall separate the single run from 
a source $r$
case into when $T_2-T_1$ is either small or large.

The parameter $\alpha$ -- the order of the fractional derivative -- turns out to play an only minor role, at least for the range of times considered;
but see Table~\ref{table:alpha} for a comparison of different $\alpha$ values.
Thus all the reconstructions to be shown utilise the
same value of $\alpha$ and its value was taken to be $0.8$.
Of course, if we were in the situation of considering very large times as part
of our reconstruction information this may be a larger factor.

Will will also show the effect of the function $f(u)$ or, more precisely,
the term $f(u)/u$ that directly couples the unknowns $q(x)$ and $p(x)$.
In the reconstructions to be shown we take four choices of $f(u)$:
\begin{equation}\label{fu1234}
f_1(u) = u^2, \quad
f_2(u) = \frac{1}{4}u^3, \quad
f_3(u) = 4u^2, \quad
f_4(u) = u^3.
\end{equation}
As can be seen from the graphs to follow,
there is a remarkable difference in the quality of the  reconstructions
both in the difference between the final iteration and the actual
$p(x)$ and $q(x)$ and in the number of iterations taken to
achieve this under both the different data cases and the terms $f(u)$.

In each of these figures, the first column shows those from the two-run case, the second from the
case of a single run with  $T_1$ and $T_2$ being ``far'' ($T_1=0.05$, $T_2=0.3$)
and the third from these being ``near'' ($T_1=0.2$, $T_2=0.3$).
 
The figure below shows,
for each of the four choices \eqref{fu1234} of the function $f(u)$,
the error history ($p$ blue, $q$ red).
\vspace*{1cm}
\newbox\figureonenew
\newbox\figuretwonew
\newbox\figurethreenew
\newbox\figurefournew
\newbox\figurefivenew
\newbox\figuresixnew
\newbox\figuresevennew
\newbox\figureeightnew
\newbox\figureninenew
\newbox\figuretennew
\newbox\figureelevennew
\newbox\figuretwelvenew
\newbox\figurethirteennew
\newbox\figurefourteennew
\newbox\figurefifteennew
\newbox\figurelegendnew
\newbox\figurelegendtextnew
\newbox\figurelegendvalues
\newdimen\xfiglen
\newdimen\yfiglen
\newdimen\xfigdim \newdimen\yfigdim \newdimen\legendheight
\newdimen\xfiglenshort \newdimen\yfiglenshort
\xfigdim = 3true in
\yfiglen = 1.25true in
\xfiglen = 2true in
\yfigdim = 2true in
\legendheight = 1.0true in
\setbox\figurelegendtextnew=
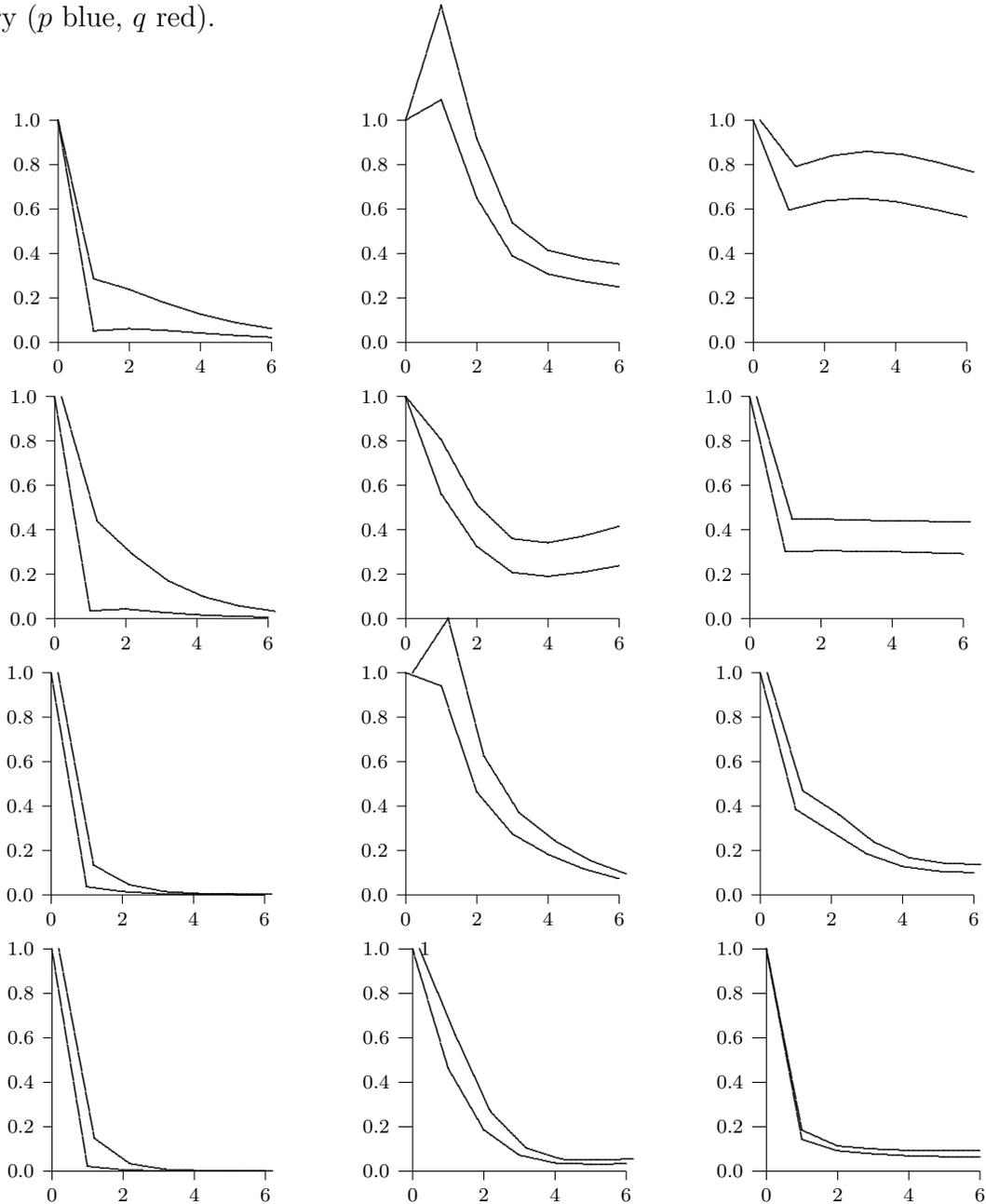
\vbox{
 \beginpicture
   \setcoordinatesystem units <0.25\xfigdim,0.19\legendheight> 
   \setplotarea x from 0 to 7, y from 0 to 4
 \footnotesize
 \put{Row 1: $f(u)=u^2$} [l] at  0 3.2
 \put{Row 2: $f(u)=\frac{1}{4}u^3$} [l] at 0 2
 \put{Row 3: $f(u)=4u^2$}  [l] at 2 3.2
 \put{Row 4: $f(u)=u^3$} [l] at 2 2
 \put{Column 1:\ \, two runs,  $\,T=0.3$} [l] at 4.3 3.5
 \put{Column 2:\ \, $T_1=0.05$, $T_2=0.3$} [l] at 4.3 2.5
 \put{Column 3:\ \, $T_1=0.2$, $T_2=0.3$} [l] at 4.3 1.5
 \endpicture
 }
%

\setbox\figureonenew=\vbox{\hsize=\xfiglen
\beginpicture
\scriptsize
  \setcoordinatesystem units <\xfiglen,\yfiglen>  point at 0 -1.5
  \setplotarea x from 0 to 0.6, y from 0 to 1
  \axis bottom shiftedto y=0 ticks short withvalues $0$ $2$ $4$ $6$ / quantity 4 /
  \axis left ticks short numbered from 0 to 1 by 0.2 /
\setlinear
\setlinear
\setsolid
\Blue{\relax  
\plot
        0          1
    0.1000    0.051546
    0.2000    0.061163
    0.3000    0.053878
    0.4000    0.041753
    0.5000    0.030832
    0.6000    0.022535
/ }\relax
\Red{\relax   
\plot 
        0          1
    0.1000    0.286295
    0.2000    0.238260
    0.3000    0.177653
    0.4000    0.126678
    0.5000    0.088639
    0.6000    0.061799 
/}\relax
\endpicture
}   
\setbox\figuretwonew=\vbox{\hsize=\xfiglen 
\beginpicture
\scriptsize
  \setcoordinatesystem units <\xfiglen,\yfiglen>  point at 0 -1.5
  \setplotarea x from 0 to 0.6, y from 0 to 1
  \axis bottom shiftedto y=0 ticks short withvalues $0$ $2$ $4$ $6$  / quantity 4 /
  \axis left ticks short numbered from 0 to 1 by 0.2 /
\setlinear
\setsolid
\Blue{\relax  
\plot
        0          1
    0.1000    1.091351
    0.2000    0.650000
    0.3000    0.389947
    0.4000    0.307143
    0.5000    0.274368
    0.6000    0.249633
/ \relax}\relax
\Red{\relax  
\plot
        0          1
    0.1000    1.515702
    0.2000    0.918321
    0.3000    0.537259
    0.4000    0.414928
    0.5000    0.375526
    0.6000    0.352553
/\relax}\relax
\endpicture
}
\setbox\figurethreenew=\vbox{\hsize=\xfiglen   
\beginpicture
\scriptsize
  \setcoordinatesystem units <\xfiglen,\yfiglen>  point at 0 -1.5
  \setplotarea x from 0 to 0.6, y from 0 to 1
  \axis bottom shiftedto y=0 ticks short withvalues $0$ $2$ $4$ $6$  / quantity 4 /
  \axis left ticks short numbered from 0 to 1 by 0.2 /
\setlinear
\setsolid
\Blue{\relax  
\plot
0         1
0.1000    0.595010
0.2000    0.635641
0.3000    0.648183 
0.4000    0.633163
0.5000    0.601400
0.6000    0.564639
/\relax\relax
}
\setsolid
\Red{\relax
\plot
0.0000 1.0
0.1000 0.790981
0.2000 0.839237
0.3000 0.859279
0.4000 0.845678
0.5000 0.809310
0.6000 0.766105
/\relax}\relax
\endpicture
}
%


%
\setbox\figurefournew=\vbox{\hsize=\xfiglen
\beginpicture
\scriptsize
  \setcoordinatesystem units <\xfiglen,\yfiglen>  point at 0 -1.5
  \setplotarea x from 0 to 0.6, y from 0 to 1
  \axis bottom shiftedto y=0 ticks short withvalues $0$ $2$ $4$ $6$  / quantity 4 /
  \axis left ticks short numbered from 0 to 1 by 0.2 /
\setlinear
\setlinear
\setsolid
\Blue{\relax  
\plot
    0         1.0
    0.1000    0.034988
    0.2000    0.043296
    0.3000    0.028581
    0.4000    0.017374
    0.5000    0.010688
    0.6000    0.006934
/}
\Red{\relax   
\plot
        0          1
    0.1000    0.438421
    0.2000    0.289019
    0.3000    0.170169
    0.4000    0.098319
    0.5000    0.056884
    0.6000    0.033355
/
 }\relax
\endpicture
}   
\setbox\figurefivenew=\vbox{\hsize=\xfiglen 
\beginpicture
\scriptsize
  \setcoordinatesystem units <\xfiglen,\yfiglen>  point at 0 -1.5
  \setplotarea x from 0 to 0.6, y from 0 to 1
  \axis bottom shiftedto y=0 ticks short withvalues $0$ $2$ $4$ $6$  / quantity 4 /
  \axis left ticks short numbered from 0 to 1 by 0.2 /
\setlinear
\setsolid
\Blue{\relax  
\plot
    0         1
    0.1000    0.562000
    0.2000    0.325013
    0.3000    0.207499
    0.4000    0.190064
    0.5000    0.209256
    0.6000    0.237973
/\relax}\relax
\Red{\relax  
\plot
        0          1
    0.1000    0.804946
    0.2000    0.513148 
    0.3000    0.359587
    0.4000    0.341260
    0.5000    0.371589
    0.6000    0.415089
/\relax}\relax
\endpicture
}
\setbox\figuresixnew=\vbox{\hsize=\xfiglen  
\beginpicture
\scriptsize
  \setcoordinatesystem units <\xfiglen,\yfiglen>  point at 0 -1.5
  \setplotarea x from 0 to 0.6, y from 0 to 1
  \axis bottom shiftedto y=0 ticks short withvalues $0$ $2$ $4$ $6$  / quantity 4 /
  \axis left ticks short numbered from 0 to 1 by 0.2 /
\setlinear
\setsolid
\Blue{\relax  
\plot
0         1
0.1000    0.302625
0.2000    0.304730
0.3000    0.303962
0.4000    0.301304
0.5000    0.296648
0.6000    0.291484
/\relax
}
\setsolid
\Red{\relax
\plot
0.0000 1.0
0.1000 0.446934
0.2000 0.446755
0.3000 0.442177
0.4000 0.440012
0.5000 0.437182
0.6000 0.434625
/\relax
}
\endpicture
}
%
\setbox\figuresevennew=\vbox{\hsize=\xfiglen  
\beginpicture
\scriptsize
  \setcoordinatesystem units <\xfiglen,\yfiglen>  point at 0 -1.5
  \setplotarea x from 0 to 0.6, y from 0 to 1
  \axis bottom shiftedto y=0 ticks short withvalues $0$ $2$ $4$ $6$  / quantity 4 /
  \axis left ticks short numbered from 0 to 1 by 0.2 /
\setlinear
\setsolid
\Blue{\relax 
\plot
0.0000 1.0
0.1000 0.036572
0.2000 0.015331
0.3000 0.005555
0.4000 0.002535
0.5000 0.001789
0.6000 0.001623
/\relax
}
%
\Red{\relax
\plot
0.0000 1.0
0.1000 0.133935
0.2000 0.045520 
0.3000 0.014996
0.4000 0.005598
0.5000 0.003132
0.6000 0.002626
/\relax
}
\endpicture
} 
\setbox\figureeightnew=\vbox{\hsize=\xfiglen  
\beginpicture
\scriptsize
  \setcoordinatesystem units <\xfiglen,\yfiglen>  point at 0 -1.5
  \setplotarea x from 0 to 0.6, y from 0 to 1
  \axis bottom shiftedto y=0 ticks short withvalues $0$ $2$ $4$ $6$  / quantity 4 /
  \axis left ticks short numbered from 0 to 1 by 0.2 /
\setlinear
\setsolid
\Blue{\relax 
\plot
0         1
0.1000    0.940460
0.2000    0.463716
0.3000    0.274468
0.4000    0.182588
0.5000    0.116919
0.6000    0.073009
/\relax
}
\setsolid
\Red{\relax
\plot
0         1
0.1000   1.244179
0.2000   0.626720
0.3000   0.368447
0.4000   0.243446
0.5000   0.153865
0.6000   0.094956
/\relax
}
\endpicture
}  
\setbox\figureninenew=\vbox{\hsize=\xfiglen  
\beginpicture
\scriptsize
  \setcoordinatesystem units <\xfiglen,\yfiglen>  point at 0 -1.5
  \setplotarea x from 0 to 0.6, y from 0 to 1
  \axis bottom shiftedto y=0 ticks short withvalues $0$ $2$ $4$ $6$  / quantity 4 /
  \axis left ticks short numbered from 0 to 1 by 0.2 /
\setlinear
\setsolid
\Blue{\relax 
\plot
0         1
0.1000    0.384513
0.2000    0.284676
0.3000    0.184196
0.4000    0.127657
0.5000    0.106230
0.6000    0.099974
/\relax
}
\setsolid
\Red{\relax
\plot
0.0000 1.0
0.1000 0.469081
0.2000 0.364222
0.3000 0.238103
0.4000 0.166311
0.5000 0.142017
0.6000 0.136888
/
}\relax
\endpicture
} 
%
\setbox\figuretennew=\vbox{\hsize=\xfiglen  
\beginpicture
\scriptsize
  \setcoordinatesystem units <\xfiglen,\yfiglen>  point at 0 -1.5
  \setplotarea x from 0 to 0.6, y from 0 to 1
  \axis bottom shiftedto y=0 ticks short withvalues $0$ $2$ $4$ $6$  / quantity 4 /
  \axis left ticks short numbered from 0 to 1 by 0.2 /
\setlinear
\setsolid
\Blue{\relax 
\plot
0         1
0.1000    0.021463
0.2000    0.005790
0.3000    0.001624
0.4000    0.000818
0.5000    0.000692
0.6000    0.000669
/\relax
}
\setsolid
\Red{\relax
\plot
0         1
0.1000    0.147844
0.2000    0.033595
0.3000    0.007709
0.4000    0.002251
0.5000    0.001421
0.6000    0.001337
/\relax
}
\endpicture
} 
%
%
\setbox\figureelevennew=\vbox{\hsize=\xfiglen
\beginpicture
\scriptsize
  \setcoordinatesystem units <\xfiglen,\yfiglen>  point at 0 -1.5
  \setplotarea x from 0 to 0.6, y from 0 to 1
  \axis bottom shiftedto y=0 ticks short withvalues $0$ $2$ $4$ $6$  / quantity 4 /
  \axis left ticks numbered from 0 to 1 by 0.2 /
  \put {1} [l] at 0.02 1
\setlinear
\setlinear
\setsolid
\Blue{\relax  
\plot
0         1
0.1000    0.463348
0.2000    0.186541
0.3000    0.071922
0.4000    0.036877
0.5000    0.031351
0.6000    0.033245
/}
\Red{\relax   
\plot
0         1
0.1000    0.615062
0.2000    0.266196
0.3000    0.103732
0.4000    0.053622 
0.5000    0.049550 
0.6000    0.055364
/}\relax
\endpicture
}    
\setbox\figuretwelvenew=\vbox{\hsize=\xfiglen 
\beginpicture
\scriptsize
  \setcoordinatesystem units <\xfiglen,\yfiglen>  point at 0 -1.5
  \setplotarea x from 0 to 0.6, y from 0 to 1
  \axis bottom shiftedto y=0 ticks short withvalues $0$ $2$ $4$ $6$  / quantity 4 /
  \axis left ticks numbered from 0 to 1 by 0.2 /
\setlinear
\setsolid
\Blue{\relax  
\plot
    0         1
    0.1000    0.143375
    0.2000    0.091723
    0.3000    0.077785
    0.4000    0.068910
    0.5000    0.065366
    0.6000    0.064204
/\relax}\relax
\Red{\relax  
\plot
        0          1
    0.1000    0.184910
    0.2000    0.113422
    0.3000    0.100096
    0.4000    0.093179
    0.5000    0.091318
    0.6000    0.091008
/\relax}\relax
\endpicture
}  
\relax
\relax
\smallskip
\hbox to \hsize{\hss\copy\figureonenew\hss\copy\figuretwonew\hss\copy\figurethreenew\hss}
\medskip

\hbox to \hsize{\hss\copy\figurefournew\hss\copy\figurefivenew\hss\copy\figuresixnew\hss}
\medskip

\hbox to \hsize{\hss\copy\figuresevennew\hss\copy\figureeightnew\hss\copy\figureninenew\hss}
\medskip

\hbox to \hsize{\hss\copy\figuretennew\hss\copy\figureelevennew\hss\copy\figuretwelvenew\hss}


\vfill\eject
The next set of figures shows the reconstructions of both $p$ and $q$.
In each picture, the first (red), second (green) and sixth iterates (blue) are plotted, for 
$f(u)=u^2$, (rows 1 and 2)
$f(u)=\frac14 u^3$, (rows 3 and 4)
$f(u)=4u^2$, (rows 5 and 6)
$f(u)=u^3$, (rows 7 and 8).

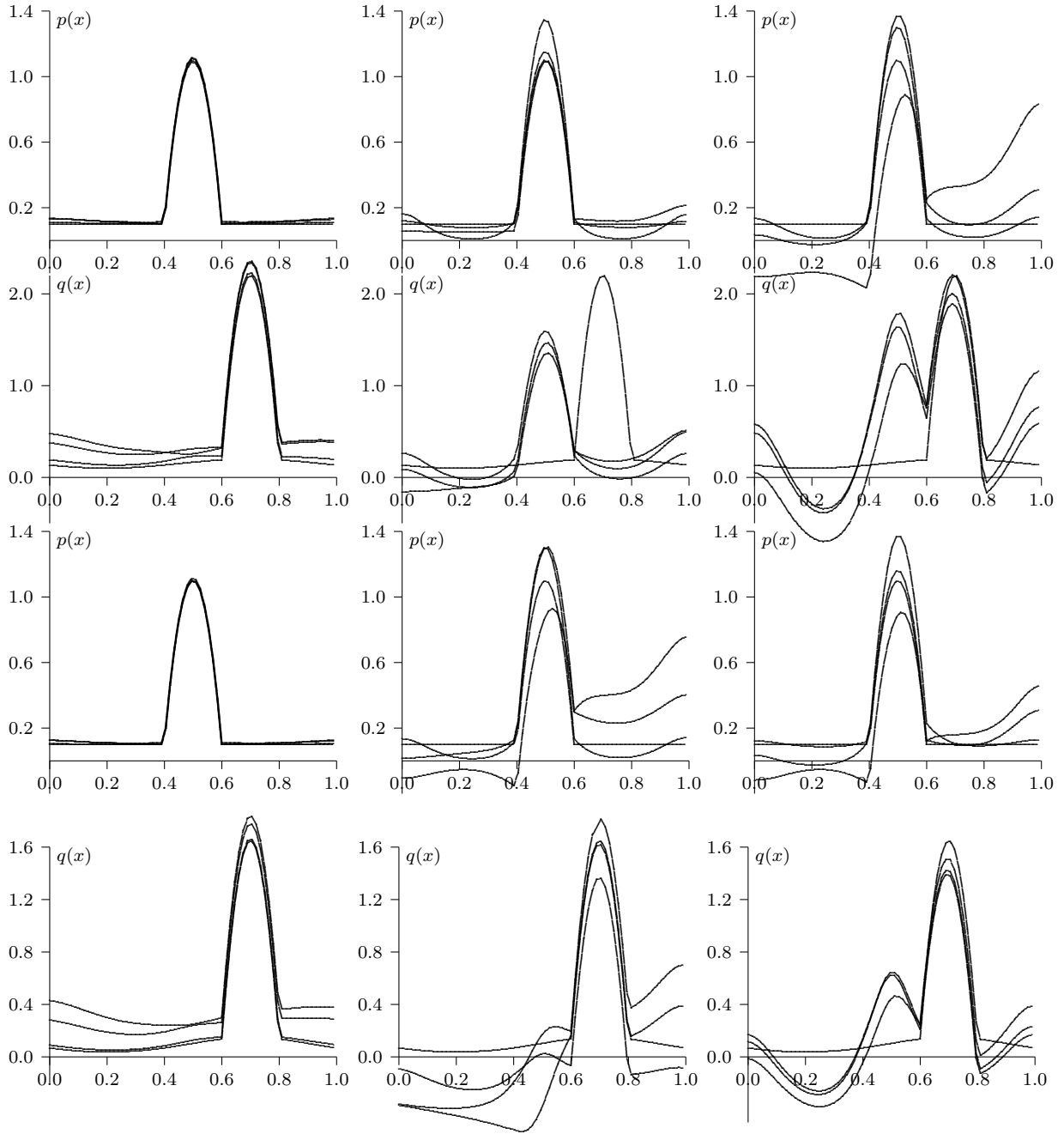
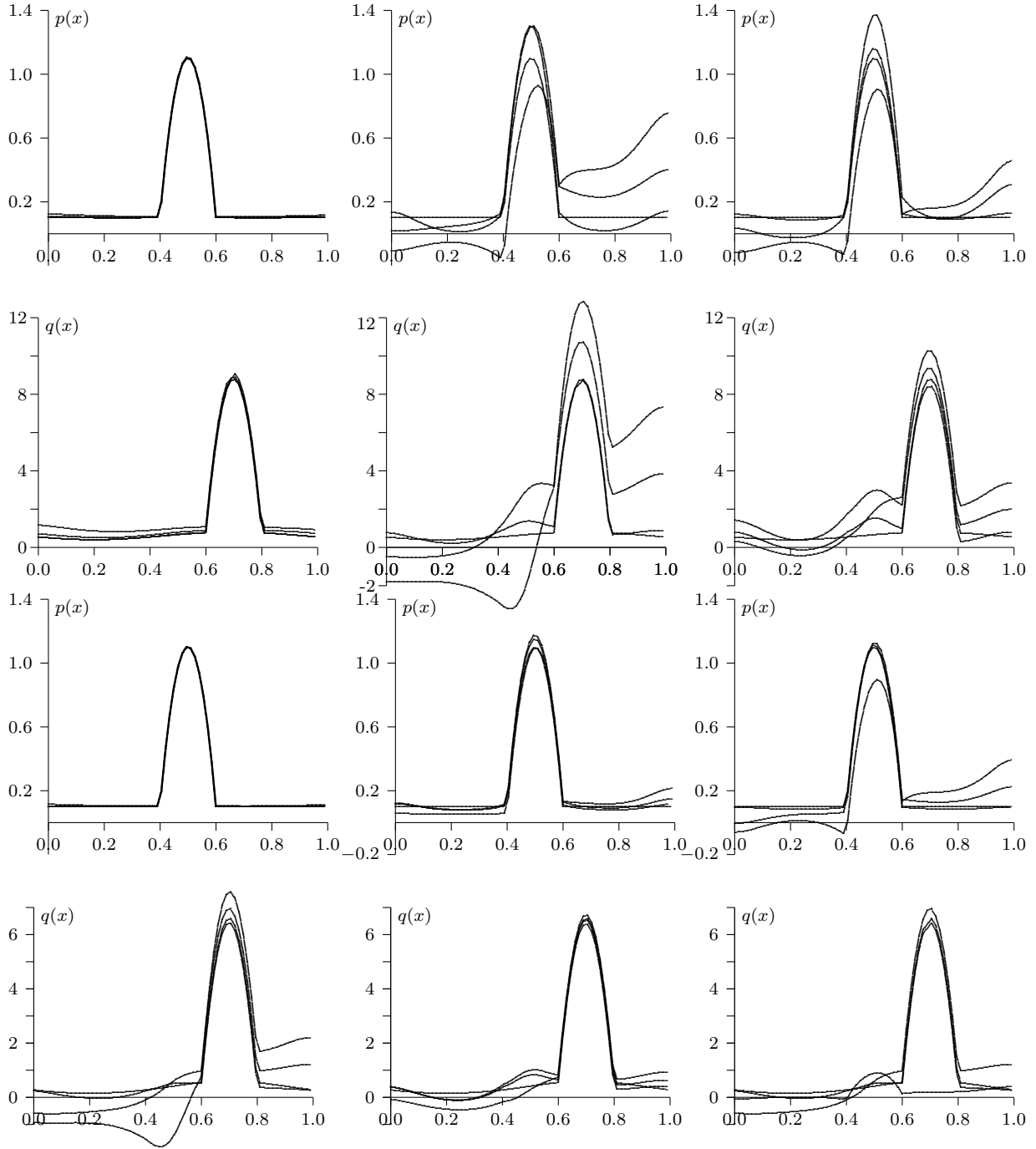
\setbox\figurelegendtext=\vbox{
\beginpicture
  \setcoordinatesystem units <0.5\xfigdim,1.2\legendheight> 
  \setplotarea x from 2 to 6, y from 0 to 4
\scriptsize
\put{Rows 1,\ 2 $\ \,f(u) = u^2$:} [l] at 0 2
\put{Rows 3,\ 4 $\,f(u) = \frac{1}{4}u^3$:} [l] at 2.5 2
\put{Rows 5,\ 6 $\ \,f(u) = 4u^2$:} [l] at 0 0.7
\put{Rows 7,\ 8 $\,f(u) = u^3$:} [l] at 2.5 0.7
\put{Columns (left to right) :\quad two runs, \quad $T_1$-$T_2$ far,\quad $T_1$-$T_2$ near.} [l] at 0 -0.7
\put{Starting values: $p,\,q\equiv0$, 1st (red), 2nd (green), 6th (blue) iterate versus actual coefficients (black)} [l] at 0 -2
\endpicture
}
\setbox\figurelegendtexttwo=\vbox{
\beginpicture
\put{{\bf Row 1:} $f(u) = u^2$, $p$ from: 2run data,\ \ $T_1$-$T_2$ far,\ \ $T_1$-$T_2$ near} [l] at 0 3
\put{{\bf Row 2:} $f(u) = u^2$, $q$ from: 2run data\quad $T_1$-$T_2$ far\quad $T_1$-$T_2$ near} [l] at 6.0 3
\put{{\bf Row 3:} $f(u) = u^3$, $p$ from: 2run data\quad $T_1$-$T_2$ far\quad $T_1$-$T_2$ near} [l] at 0 1
\put{{\bf Row 4:} $f(u) = u^3$, $q$ from: 2run data\quad $T_1$-$T_2$ far\quad $T_1$-$T_2$ near} [l] at 6 1
\endpicture
}
\xfigdim = 1.75true in
\yfigdim = 1.0true in
\setbox\figureone=\vbox{\hsize=\xfiglen
\beginpicture
\linethickness=0.25pt
\scriptsize
  \setcoordinatesystem units <\xfigdim,\yfigdim>  point at 0 0
  \setplotarea x from 0 to 1, y from -0.2 to 1.4
  \axis bottom shiftedto y=0 ticks short numbered from 0 to 1 by 0.2 /
  \axis left ticks short numbered from 0.2 to 1.4 by 0.4 /
  \put {$p(x)$} [lt] at 0.025 1.4
\setquadratic
\setlinear
\setsolid
\Black{\relax  
\plot
  0.0000   0.1000
  0.0150   0.1000
  0.0300   0.1000
  0.0450   0.1000
  0.0600   0.1000
  0.0750   0.1000
  0.0900   0.1000
  0.1050   0.1000
  0.1200   0.1000
  0.1350   0.1000
  0.1500   0.1000
  0.1650   0.1000
  0.1800   0.1000
  0.1950   0.1000
  0.2100   0.1000
  0.2250   0.1000
  0.2400   0.1000
  0.2550   0.1000
  0.2700   0.1000
  0.2850   0.1000
  0.3000   0.1000
  0.3150   0.1000
  0.3300   0.1000
  0.3450   0.1000
  0.3600   0.1000
  0.3750   0.1000
  0.3900   0.1000
  0.4050   0.1975
  0.4200   0.4600
  0.4350   0.6775
  0.4500   0.8500
  0.4650   0.9775
  0.4800   1.0600
  0.4950   1.0975
  0.5100   1.0900
  0.5250   1.0375
  0.5400   0.9400
  0.5550   0.7975
  0.5700   0.6100
  0.5850   0.3775
  0.6000   0.1000
  0.6150   0.1000
  0.6300   0.1000
  0.6450   0.1000
  0.6600   0.1000
  0.6750   0.1000
  0.6900   0.1000
  0.7050   0.1000
  0.7200   0.1000
  0.7350   0.1000
  0.7500   0.1000
  0.7650   0.1000
  0.7800   0.1000
  0.7950   0.1000
  0.8100   0.1000
  0.8250   0.1000
  0.8400   0.1000
  0.8550   0.1000
  0.8700   0.1000
  0.8850   0.1000
  0.9000   0.1000
  0.9150   0.1000
  0.9300   0.1000
  0.9450   0.1000
  0.9600   0.1000
  0.9750   0.1000
  0.9900   0.1000
/}\relax\relax
%
\linethickness=0.15pt
\thinlines
\Red {\relax   
\plot
  0.0000   0.1340
  0.0150   0.1336
  0.0300   0.1330
  0.0450   0.1320
  0.0600   0.1308
  0.0750   0.1293
  0.0900   0.1276
  0.1050   0.1259
  0.1200   0.1241
  0.1350   0.1223
  0.1500   0.1207
  0.1650   0.1190
  0.1800   0.1175
  0.1950   0.1161
  0.2100   0.1148
  0.2250   0.1136
  0.2400   0.1124
  0.2550   0.1113
  0.2700   0.1103
  0.2850   0.1093
  0.3000   0.1082
  0.3150   0.1072
  0.3300   0.1060
  0.3450   0.1048
  0.3600   0.1035
  0.3750   0.1020
  0.3900   0.1004
  0.4050   0.1962
  0.4200   0.4569
  0.4350   0.6726
  0.4500   0.8434
  0.4650   0.9695
  0.4800   1.0510
  0.4950   1.0882
  0.5100   1.0811
  0.5250   1.0296
  0.5400   0.9336
  0.5550   0.7929
  0.5700   0.6072
  0.5850   0.3765
  0.6000   0.1006
  0.6150   0.1021
  0.6300   0.1034
  0.6450   0.1045
  0.6600   0.1055
  0.6750   0.1064
  0.6900   0.1072
  0.7050   0.1079
  0.7200   0.1086
  0.7350   0.1094
  0.7500   0.1102
  0.7650   0.1110
  0.7800   0.1119
  0.7950   0.1128
  0.8100   0.1139
  0.8250   0.1150
  0.8400   0.1163
  0.8550   0.1176
  0.8700   0.1190
  0.8850   0.1205
  0.9000   0.1220
  0.9150   0.1235
  0.9300   0.1249
  0.9450   0.1261
  0.9600   0.1272
  0.9750   0.1279
  0.9900   0.1283
/}\relax\relax
\Green{\relax   
\plot
  0.0000   0.1330
  0.0150   0.1327
  0.0300   0.1320
  0.0450   0.1310
  0.0600   0.1297
  0.0750   0.1281
  0.0900   0.1264
  0.1050   0.1247
  0.1200   0.1230
  0.1350   0.1214
  0.1500   0.1198
  0.1650   0.1184
  0.1800   0.1171
  0.1950   0.1160
  0.2100   0.1150
  0.2250   0.1142
  0.2400   0.1135
  0.2550   0.1129
  0.2700   0.1124
  0.2850   0.1121
  0.3000   0.1119
  0.3150   0.1117
  0.3300   0.1118
  0.3450   0.1119
  0.3600   0.1121
  0.3750   0.1125
  0.3900   0.1130
  0.4050   0.2112
  0.4200   0.4744
  0.4350   0.6927
  0.4500   0.8659
  0.4650   0.9942
  0.4800   1.0772
  0.4950   1.1150
  0.5100   1.1075
  0.5250   1.0546
  0.5400   0.9566
  0.5550   0.8134
  0.5700   0.6252
  0.5850   0.3920
  0.6000   0.1138
  0.6150   0.1133
  0.6300   0.1128
  0.6450   0.1125
  0.6600   0.1124
  0.6750   0.1123
  0.6900   0.1124
  0.7050   0.1125
  0.7200   0.1128
  0.7350   0.1132
  0.7500   0.1137
  0.7650   0.1144
  0.7800   0.1151
  0.7950   0.1161
  0.8100   0.1171
  0.8250   0.1183
  0.8400   0.1197
  0.8550   0.1212
  0.8700   0.1228
  0.8850   0.1245
  0.9000   0.1263
  0.9150   0.1280
  0.9300   0.1297
  0.9450   0.1312
  0.9600   0.1325
  0.9750   0.1334
  0.9900   0.1339
/}\relax\relax
\Blue{\relax  
\plot
  0.0000   0.1111
  0.0150   0.1108
  0.0300   0.1105
  0.0450   0.1100
  0.0600   0.1094
  0.0750   0.1087
  0.0900   0.1079
  0.1050   0.1072
  0.1200   0.1065
  0.1350   0.1058
  0.1500   0.1052
  0.1650   0.1046
  0.1800   0.1042
  0.1950   0.1038
  0.2100   0.1035
  0.2250   0.1032
  0.2400   0.1031
  0.2550   0.1030
  0.2700   0.1029
  0.2850   0.1030
  0.3000   0.1031
  0.3150   0.1033
  0.3300   0.1036
  0.3450   0.1040
  0.3600   0.1045
  0.3750   0.1052
  0.3900   0.1060
  0.4050   0.2045
  0.4200   0.4680
  0.4350   0.6867
  0.4500   0.8603
  0.4650   0.9888
  0.4800   1.0720
  0.4950   1.1099
  0.5100   1.1023
  0.5250   1.0494
  0.5400   0.9511
  0.5550   0.8075
  0.5700   0.6189
  0.5850   0.3853
  0.6000   0.1068
  0.6150   0.1059
  0.6300   0.1052
  0.6450   0.1046
  0.6600   0.1041
  0.6750   0.1037
  0.6900   0.1035
  0.7050   0.1033
  0.7200   0.1032
  0.7350   0.1032
  0.7500   0.1033
  0.7650   0.1034
  0.7800   0.1036
  0.7950   0.1039
  0.8100   0.1043
  0.8250   0.1047
  0.8400   0.1052
  0.8550   0.1058
  0.8700   0.1064
  0.8850   0.1071
  0.9000   0.1079
  0.9150   0.1087
  0.9300   0.1094
  0.9450   0.1102
  0.9600   0.1107
  0.9750   0.1112
  0.9900   0.1114
/}\relax\relax
\endpicture
}   
\setbox\figuretwo=\vbox{\hsize=\xfiglen 
\beginpicture
\scriptsize
  \setcoordinatesystem units <\xfigdim,\yfigdim>  point at 0 0
  \setplotarea x from 0 to 1, y from -0.2 to 1.4
  \axis bottom shiftedto y=0.0 ticks short numbered from 0 to 1 by 0.2 /
  \axis left ticks short numbered from 0.2 to 1.4 by 0.4 /
  \put {$p(x)$} [lt] at 0.025 1.4
\setlinear
\linethickness=0.5pt
\setsolid
\Black{\relax  
\plot
  0.0000   0.1000
  0.0150   0.1000
  0.0300   0.1000
  0.0450   0.1000
  0.0600   0.1000
  0.0750   0.1000
  0.0900   0.1000
  0.1050   0.1000
  0.1200   0.1000
  0.1350   0.1000
  0.1500   0.1000
  0.1650   0.1000
  0.1800   0.1000
  0.1950   0.1000
  0.2100   0.1000
  0.2250   0.1000
  0.2400   0.1000
  0.2550   0.1000
  0.2700   0.1000
  0.2850   0.1000
  0.3000   0.1000
  0.3150   0.1000
  0.3300   0.1000
  0.3450   0.1000
  0.3600   0.1000
  0.3750   0.1000
  0.3900   0.1000
  0.4050   0.1975
  0.4200   0.4600
  0.4350   0.6775
  0.4500   0.8500
  0.4650   0.9775
  0.4800   1.0600
  0.4950   1.0975
  0.5100   1.0900
  0.5250   1.0375
  0.5400   0.9400
  0.5550   0.7975
  0.5700   0.6100
  0.5850   0.3775
  0.6000   0.1000
  0.6150   0.1000
  0.6300   0.1000
  0.6450   0.1000
  0.6600   0.1000
  0.6750   0.1000
  0.6900   0.1000
  0.7050   0.1000
  0.7200   0.1000
  0.7350   0.1000
  0.7500   0.1000
  0.7650   0.1000
  0.7800   0.1000
  0.7950   0.1000
  0.8100   0.1000
  0.8250   0.1000
  0.8400   0.1000
  0.8550   0.1000
  0.8700   0.1000
  0.8850   0.1000
  0.9000   0.1000
  0.9150   0.1000
  0.9300   0.1000
  0.9450   0.1000
  0.9600   0.1000
  0.9750   0.1000
  0.9900   0.1000
/}\relax
\setsolid
%
\linethickness=0.15pt
\Red{\relax   
\plot
  0.0000   0.0586
  0.0150   0.0584
  0.0300   0.0577
  0.0450   0.0567
  0.0600   0.0556
  0.0750   0.0545
  0.0900   0.0537
  0.1050   0.0530
  0.1200   0.0527
  0.1350   0.0525
  0.1500   0.0526
  0.1650   0.0528
  0.1800   0.0531
  0.1950   0.0534
  0.2100   0.0536
  0.2250   0.0538
  0.2400   0.0539
  0.2550   0.0539
  0.2700   0.0538
  0.2850   0.0537
  0.3000   0.0536
  0.3150   0.0535
  0.3300   0.0536
  0.3450   0.0541
  0.3600   0.0550
  0.3750   0.0567
  0.3900   0.0592
  0.4050   0.1584
  0.4200   0.4204
  0.4350   0.6400
  0.4500   0.8173
  0.4650   0.9519
  0.4800   1.0433
  0.4950   1.0904
  0.5100   1.0921
  0.5250   1.0476
  0.5400   0.9566
  0.5550   0.8191
  0.5700   0.6355
  0.5850   0.4064
  0.6000   0.1323
  0.6150   0.1306
  0.6300   0.1285
  0.6450   0.1263
  0.6600   0.1243
  0.6750   0.1224
  0.6900   0.1208
  0.7050   0.1195
  0.7200   0.1185
  0.7350   0.1178
  0.7500   0.1175
  0.7650   0.1177
  0.7800   0.1183
  0.7950   0.1196
  0.8100   0.1222
  0.8250   0.1262
  0.8400   0.1312
  0.8550   0.1374
  0.8700   0.1448
  0.8850   0.1535
  0.9000   0.1632
  0.9150   0.1737
  0.9300   0.1845
  0.9450   0.1949
  0.9600   0.2039
  0.9750   0.2108
  0.9900   0.2148
/
}\relax
\Green{\relax   
\plot
  0.0000   0.1619
  0.0150   0.1586
  0.0300   0.1496
  0.0450   0.1359
  0.0600   0.1192
  0.0750   0.1013
  0.0900   0.0836
  0.1050   0.0674
  0.1200   0.0532
  0.1350   0.0414
  0.1500   0.0318
  0.1650   0.0243
  0.1800   0.0187
  0.1950   0.0146
  0.2100   0.0118
  0.2250   0.0101
  0.2400   0.0094
  0.2550   0.0096
  0.2700   0.0109
  0.2850   0.0134
  0.3000   0.0175
  0.3150   0.0236
  0.3300   0.0323
  0.3450   0.0444
  0.3600   0.0608
  0.3750   0.0825
  0.3900   0.1104
  0.4050   0.2404
  0.4200   0.5382
  0.4350   0.7964
  0.4500   1.0117
  0.4650   1.1788
  0.4800   1.2914
  0.4950   1.3439
  0.5100   1.3332
  0.5250   1.2600
  0.5400   1.1284
  0.5550   0.9445
  0.5700   0.7143
  0.5850   0.4426
  0.6000   0.1321
  0.6150   0.1008
  0.6300   0.0759
  0.6450   0.0567
  0.6600   0.0422
  0.6750   0.0316
  0.6900   0.0240
  0.7050   0.0186
  0.7200   0.0149
  0.7350   0.0126
  0.7500   0.0112
  0.7650   0.0108
  0.7800   0.0111
  0.7950   0.0124
  0.8100   0.0153
  0.8250   0.0198
  0.8400   0.0260
  0.8550   0.0340
  0.8700   0.0443
  0.8850   0.0568
  0.9000   0.0716
  0.9150   0.0883
  0.9300   0.1060
  0.9450   0.1236
  0.9600   0.1394
  0.9750   0.1516
  0.9900   0.1588
/}\relax
\Blue{\relax  
\plot
  0.0000   0.1185
  0.0150   0.1177
  0.0300   0.1154
  0.0450   0.1119
  0.0600   0.1076
  0.0750   0.1030
  0.0900   0.0984
  0.1050   0.0941
  0.1200   0.0904
  0.1350   0.0872
  0.1500   0.0846
  0.1650   0.0826
  0.1800   0.0810
  0.1950   0.0799
  0.2100   0.0791
  0.2250   0.0786
  0.2400   0.0785
  0.2550   0.0785
  0.2700   0.0789
  0.2850   0.0797
  0.3000   0.0808
  0.3150   0.0825
  0.3300   0.0850
  0.3450   0.0883
  0.3600   0.0927
  0.3750   0.0986
  0.3900   0.1059
  0.4050   0.2105
  0.4200   0.4783
  0.4350   0.7031
  0.4500   0.8840
  0.4650   1.0198
  0.4800   1.1087
  0.4950   1.1495
  0.5100   1.1413
  0.5250   1.0843
  0.5400   0.9795
  0.5550   0.8285
  0.5700   0.6328
  0.5850   0.3936
  0.6000   0.1116
  0.6150   0.1041
  0.6300   0.0981
  0.6450   0.0934
  0.6600   0.0898
  0.6750   0.0870
  0.6900   0.0850
  0.7050   0.0834
  0.7200   0.0822
  0.7350   0.0813
  0.7500   0.0806
  0.7650   0.0801
  0.7800   0.0797
  0.7950   0.0795
  0.8100   0.0801
  0.8250   0.0814
  0.8400   0.0830
  0.8550   0.0852
  0.8700   0.0880
  0.8850   0.0913
  0.9000   0.0952
  0.9150   0.0995
  0.9300   0.1041
  0.9450   0.1086
  0.9600   0.1126
  0.9750   0.1157
  0.9900   0.1175
/}\relax
\endpicture
}   
\setbox\figurethree=\vbox{\hsize=\xfiglen 
\beginpicture
\scriptsize
  \setcoordinatesystem units <\xfigdim,\yfigdim>  point at 0 0
  \setplotarea x from 0 to 1, y from -0.2 to 1.4
  \axis bottom shiftedto y=0.0 ticks short numbered from 0 to 1 by 0.2 /
  \axis left ticks short numbered from 0.2 to 1.4 by 0.4 /
  \put {$p(x)$} [lt] at 0.025 1.4
\setlinear
\setsolid 
\linethickness=0.5pt
\linethickness=0.5pt
\Black{\relax  
\plot
  0.0000   0.1000
  0.0150   0.1000
  0.0300   0.1000
  0.0450   0.1000
  0.0600   0.1000
  0.0750   0.1000
  0.0900   0.1000
  0.1050   0.1000
  0.1200   0.1000
  0.1350   0.1000
  0.1500   0.1000
  0.1650   0.1000
  0.1800   0.1000
  0.1950   0.1000
  0.2100   0.1000
  0.2250   0.1000
  0.2400   0.1000
  0.2550   0.1000
  0.2700   0.1000
  0.2850   0.1000
  0.3000   0.1000
  0.3150   0.1000
  0.3300   0.1000
  0.3450   0.1000
  0.3600   0.1000
  0.3750   0.1000
  0.3900   0.1000
  0.4050   0.1975
  0.4200   0.4600
  0.4350   0.6775
  0.4500   0.8500
  0.4650   0.9775
  0.4800   1.0600
  0.4950   1.0975
  0.5100   1.0900
  0.5250   1.0375
  0.5400   0.9400
  0.5550   0.7975
  0.5700   0.6100
  0.5850   0.3775
  0.6000   0.1000
  0.6150   0.1000
  0.6300   0.1000
  0.6450   0.1000
  0.6600   0.1000
  0.6750   0.1000
  0.6900   0.1000
  0.7050   0.1000
  0.7200   0.1000
  0.7350   0.1000
  0.7500   0.1000
  0.7650   0.1000
  0.7800   0.1000
  0.7950   0.1000
  0.8100   0.1000
  0.8250   0.1000
  0.8400   0.1000
  0.8550   0.1000
  0.8700   0.1000
  0.8850   0.1000
  0.9000   0.1000
  0.9150   0.1000
  0.9300   0.1000
  0.9450   0.1000
  0.9600   0.1000
  0.9750   0.1000
  0.9900   0.1000
/}\relax
\setsolid
\linethickness=0.15pt
\thinlines
\Red{\relax   
\plot
 0.0000  -0.2226
  0.0150  -0.2225
  0.0300  -0.2216
  0.0450  -0.2202
  0.0600  -0.2181
  0.0750  -0.2156
  0.0900  -0.2127
  0.1050  -0.2095
  0.1200  -0.2063
  0.1350  -0.2031
  0.1500  -0.2002
  0.1650  -0.1979
  0.1800  -0.1962
  0.1950  -0.1953
  0.2100  -0.1953
  0.2250  -0.1964
  0.2400  -0.1987
  0.2550  -0.2022
  0.2700  -0.2071
  0.2850  -0.2132
  0.3000  -0.2208
  0.3150  -0.2297
  0.3300  -0.2399
  0.3450  -0.2512
  0.3600  -0.2634
  0.3750  -0.2760
  0.3900  -0.2884
  0.4050  -0.2038
  0.4200   0.0464
  0.4350   0.2591
  0.4500   0.4384
  0.4650   0.5881
  0.4800   0.7101
  0.4950   0.8037
  0.5100   0.8653
  0.5250   0.8888
  0.5400   0.8676
  0.5550   0.7960
  0.5700   0.6704
  0.5850   0.4891
  0.6000   0.2521
  0.6150   0.2794
  0.6300   0.2981
  0.6450   0.3104
  0.6600   0.3183
  0.6750   0.3232
  0.6900   0.3265
  0.7050   0.3292
  0.7200   0.3321
  0.7350   0.3360
  0.7500   0.3413
  0.7650   0.3488
  0.7800   0.3591
  0.7950   0.3729
  0.8100   0.3917
  0.8250   0.4157
  0.8400   0.4448
  0.8550   0.4790
  0.8700   0.5180
  0.8850   0.5614
  0.9000   0.6081
  0.9150   0.6565
  0.9300   0.7045
  0.9450   0.7492
  0.9600   0.7874
  0.9750   0.8160
  0.9900   0.8322
/}\relax
\Green{\relax   
\plot
 0.0000   0.0339
  0.0150   0.0326
  0.0300   0.0292
  0.0450   0.0239
  0.0600   0.0174
  0.0750   0.0102
  0.0900   0.0031
  0.1050  -0.0036
  0.1200  -0.0096
  0.1350  -0.0146
  0.1500  -0.0186
  0.1650  -0.0217
  0.1800  -0.0238
  0.1950  -0.0250
  0.2100  -0.0253
  0.2250  -0.0249
  0.2400  -0.0235
  0.2550  -0.0212
  0.2700  -0.0177
  0.2850  -0.0129
  0.3000  -0.0063
  0.3150   0.0024
  0.3300   0.0139
  0.3450   0.0288
  0.3600   0.0479
  0.3750   0.0721
  0.3900   0.1020
  0.4050   0.2341
  0.4200   0.5346
  0.4350   0.7953
  0.4500   1.0135
  0.4650   1.1849
  0.4800   1.3044
  0.4950   1.3669
  0.5100   1.3693
  0.5250   1.3108
  0.5400   1.1939
  0.5550   1.0229
  0.5700   0.8024
  0.5850   0.5366
  0.6000   0.2284
  0.6150   0.1985
  0.6300   0.1729
  0.6450   0.1517
  0.6600   0.1346
  0.6750   0.1212
  0.6900   0.1111
  0.7050   0.1037
  0.7200   0.0989
  0.7350   0.0962
  0.7500   0.0955
  0.7650   0.0966
  0.7800   0.0996
  0.7950   0.1045
  0.8100   0.1120
  0.8250   0.1221
  0.8400   0.1344
  0.8550   0.1491
  0.8700   0.1661
  0.8850   0.1853
  0.9000   0.2061
  0.9150   0.2280
  0.9300   0.2499
  0.9450   0.2704
  0.9600   0.2881
  0.9750   0.3013
  0.9900   0.3089
/}\relax
\Blue{\relax  
\plot
 0.0000   0.1344
  0.0150   0.1323
  0.0300   0.1265
  0.0450   0.1175
  0.0600   0.1063
  0.0750   0.0938
  0.0900   0.0809
  0.1050   0.0684
  0.1200   0.0569
  0.1350   0.0466
  0.1500   0.0377
  0.1650   0.0303
  0.1800   0.0243
  0.1950   0.0196
  0.2100   0.0162
  0.2250   0.0140
  0.2400   0.0130
  0.2550   0.0132
  0.2700   0.0147
  0.2850   0.0177
  0.3000   0.0225
  0.3150   0.0293
  0.3300   0.0387
  0.3450   0.0510
  0.3600   0.0670
  0.3750   0.0872
  0.3900   0.1120
  0.4050   0.2375
  0.4200   0.5299
  0.4350   0.7805
  0.4500   0.9868
  0.4650   1.1450
  0.4800   1.2506
  0.4950   1.2998
  0.5100   1.2904
  0.5250   1.2229
  0.5400   1.1000
  0.5550   0.9259
  0.5700   0.7048
  0.5850   0.4404
  0.6000   0.1347
  0.6150   0.1080
  0.6300   0.0858
  0.6450   0.0678
  0.6600   0.0536
  0.6750   0.0426
  0.6900   0.0343
  0.7050   0.0282
  0.7200   0.0240
  0.7350   0.0212
  0.7500   0.0197
  0.7650   0.0194
  0.7800   0.0202
  0.7950   0.0221
  0.8100   0.0258
  0.8250   0.0312
  0.8400   0.0380
  0.8550   0.0463
  0.8700   0.0561
  0.8850   0.0673
  0.9000   0.0796
  0.9150   0.0928
  0.9300   0.1062
  0.9450   0.1188
  0.9600   0.1297
  0.9750   0.1380
  0.9900   0.1427
/}\relax
\endpicture
}  
\setbox\figureten=\hbox to \hsize{\copy\figureone\hss\copy\figuretwo\hss\copy\figurethree}
%
%
\setbox\figurefour=\vbox{\hsize=\xfiglen
\beginpicture
\scriptsize
  \setcoordinatesystem units <\xfigdim,0.56\yfigdim>  point at 0 0
  \setplotarea x from 0 to 1, y from -0.5 to 2.2
  \axis bottom shiftedto y=0 ticks short numbered from 0 to 1 by 0.2 /
  \axis left ticks short numbered from 0.0 to 2.0 by 1.0 /
  \put {$q(x)$} [lt] at 0.025 2.2
\setlinear
\setsolid 
\Black{\relax  
\plot
0.0000   0.1345
0.0150   0.1301
0.0300   0.1259
0.0450   0.1219
0.0600   0.1181
0.0750   0.1146
0.0900   0.1115
0.1050   0.1086
0.1200   0.1062
0.1350   0.1041
0.1500   0.1024
0.1650   0.1012
0.1800   0.1004
0.1950   0.1000
0.2100   0.1001
0.2250   0.1006
0.2400   0.1016
0.2550   0.1030
0.2700   0.1048
0.2850   0.1070
0.3000   0.1095
0.3150   0.1125
0.3300   0.1158
0.3450   0.1194
0.3600   0.1232
0.3750   0.1273
0.3900   0.1316
0.4050   0.1361
0.4200   0.1406
0.4350   0.1453
0.4500   0.1500
0.4650   0.1547
0.4800   0.1594
0.4950   0.1639
0.5100   0.1684
0.5250   0.1727
0.5400   0.1768
0.5550   0.1806
0.5700   0.1842
0.5850   0.1875
0.6000   0.1905
0.6150   0.7287
0.6300   1.1910
0.6450   1.5714
0.6600   1.8652
0.6750   2.0686
0.6900   2.1789
0.7050   2.1947
0.7200   2.1158
0.7350   1.9432
0.7500   1.6792
0.7650   1.3271
0.7800   0.8916
0.7950   0.3779
0.8100   0.1885
0.8250   0.1854
0.8400   0.1819
0.8550   0.1781
0.8700   0.1741
0.8850   0.1699
0.9000   0.1655
0.9150   0.1609
0.9300   0.1563
0.9450   0.1516
0.9600   0.1469
0.9750   0.1422
0.9900   0.1376
/}
\setsolid
\Red{\relax   
\!\!\plot
0.0000   0.4766
0.0150   0.4709
0.0300   0.4640
0.0450   0.4556
0.0600   0.4460
0.0750   0.4353
0.0900   0.4240
0.1050   0.4123
0.1200   0.4004
0.1350   0.3887
0.1500   0.3772
0.1650   0.3662
0.1800   0.3558
0.1950   0.3461
0.2100   0.3372
0.2250   0.3291
0.2400   0.3218
0.2550   0.3152
0.2700   0.3094
0.2850   0.3042
0.3000   0.2996
0.3150   0.2954
0.3300   0.2914
0.3450   0.2877
0.3600   0.2839
0.3750   0.2799
0.3900   0.2756
0.4050   0.2711
0.4200   0.2664
0.4350   0.2618
0.4500   0.2579
0.4650   0.2553
0.4800   0.2547
0.4950   0.2568
0.5100   0.2617
0.5250   0.2691
0.5400   0.2785
0.5550   0.2889
0.5700   0.2995
0.5850   0.3097
0.6000   0.3191
0.6150   0.8632
0.6300   1.3308
0.6450   1.7160
0.6600   2.0141
0.6750   2.2215
0.6900   2.3358 
0.7050   2.3555
0.7200   2.2807
0.7350   2.1124
0.7500   1.8530
0.7650   1.5059
0.7800   1.0757
0.7950   0.5678
0.8100   0.3845
0.8250   0.3879
0.8400   0.3912
0.8550   0.3945
0.8700   0.3976
0.8850   0.4005
0.9000   0.4030
0.9150   0.4050
0.9300   0.4064
0.9450   0.4068
0.9600   0.4063
0.9750   0.4045
0.9900   0.4016
/}\relax
\Green{\relax   
\plot
0.0000   0.3752
0.0150   0.3698
0.0300   0.3634
0.0450   0.3558
0.0600   0.3474
0.0750   0.3383
0.0900   0.3288
0.1050   0.3192
0.1200   0.3097
0.1350   0.3005
0.1500   0.2918
0.1650   0.2836
0.1800   0.2762
0.1950   0.2696
0.2100   0.2638
0.2250   0.2589
0.2400   0.2549
0.2550   0.2518
0.2700   0.2497
0.2850   0.2484
0.3000   0.2481
0.3150   0.2487
0.3300   0.2502
0.3450   0.2526
0.3600   0.2558
0.3750   0.2600
0.3900   0.2649
0.4050   0.2707
0.4200   0.2771
0.4350   0.2840
0.4500   0.2911
0.4650   0.2981
0.4800   0.3046
0.4950   0.3102
0.5100   0.3148
0.5250   0.3182
0.5400   0.3207
0.5550   0.3227
0.5700   0.3243
0.5850   0.3259
0.6000   0.3276
0.6150   0.8652
0.6300   1.3273
0.6450   1.7083
0.6600   2.0032
0.6750   2.2083
0.6900   2.3209
0.7050   2.3394
0.7200   2.2638
0.7350   2.0949
0.7500   1.8351
0.7650   1.4877
0.7800   1.0573
0.7950   0.5492
0.8100   0.3659
0.8250   0.3691
0.8400   0.3724
0.8550   0.3756
0.8700   0.3788
0.8850   0.3818
0.9000   0.3844
0.9150   0.3866
0.9300   0.3881
0.9450   0.3888
0.9600   0.3885
0.9750   0.3870
0.9900   0.3841
/}\relax
\Blue{\relax  
\plot
0.0000   0.1913
0.0150   0.1864
0.0300   0.1814
0.0450   0.1762
0.0600   0.1709
0.0750   0.1656
0.0900   0.1605
0.1050   0.1556
0.1200   0.1510
0.1350   0.1469
0.1500   0.1433
0.1650   0.1402
0.1800   0.1377
0.1950   0.1359
0.2100   0.1346
0.2250   0.1339
0.2400   0.1339
0.2550   0.1346
0.2700   0.1359
0.2850   0.1378
0.3000   0.1405
0.3150   0.1438
0.3300   0.1479
0.3450   0.1527
0.3600   0.1583
0.3750   0.1646
0.3900   0.1718
0.4050   0.1797
0.4200   0.1883
0.4350   0.1972
0.4500   0.2062
0.4650   0.2147
0.4800   0.2223
0.4950   0.2284
0.5100   0.2326
0.5250   0.2352
0.5400   0.2362
0.5550   0.2362
0.5700   0.2356
0.5850   0.2348
0.6000   0.2341
0.6150   0.7693
0.6300   1.2290
0.6450   1.6075
0.6600   1.8998
0.6750   2.1023
0.6900   2.2120
0.7050   2.2277
0.7200   2.1489
0.7350   1.9767
0.7500   1.7134
0.7650   1.3622
0.7800   0.9277
0.7950   0.4154
0.8100   0.2275
0.8250   0.2260
0.8400   0.2243
0.8550   0.2225
0.8700   0.2206
0.8850   0.2185
0.9000   0.2163
0.9150   0.2139
0.9300   0.2112
0.9450   0.2083
0.9600   0.2051
0.9750   0.2015
0.9900   0.1975
/}\relax
\endpicture
}   
\setbox\figuresix=\vbox{\hsize=\xfiglen
\beginpicture
\scriptsize
  \setcoordinatesystem units <\xfigdim,0.56\yfigdim>  point at 0 0
  \setplotarea x from 0 to 1, y from -0.5 to 2.2
  \axis bottom shiftedto y=0 ticks short numbered from 0 to 1 by 0.2 /
  \axis left ticks short numbered from 0.0 to 2.0 by 1.0 /
  \put {$q(x)$} [lt] at 0.025 2.2
\setlinear
\setsolid 
\Black{\relax  
\plot
0.0000   0.1345
0.0150   0.1301
0.0300   0.1259
0.0450   0.1219
0.0600   0.1181
0.0750   0.1146
0.0900   0.1115
0.1050   0.1086
0.1200   0.1062
0.1350   0.1041
0.1500   0.1024
0.1650   0.1012
0.1800   0.1004
0.1950   0.1000
0.2100   0.1001
0.2250   0.1006
0.2400   0.1016
0.2550   0.1030
0.2700   0.1048
0.2850   0.1070
0.3000   0.1095
0.3150   0.1125
0.3300   0.1158
0.3450   0.1194
0.3600   0.1232
0.3750   0.1273
0.3900   0.1316
0.4050   0.1361
0.4200   0.1406
0.4350   0.1453
0.4500   0.1500
0.4650   0.1547
0.4800   0.1594
0.4950   0.1639
0.5100   0.1684
0.5250   0.1727
0.5400   0.1768
0.5550   0.1806
0.5700   0.1842
0.5850   0.1875
0.6000   0.1905
0.6150   0.7287
0.6300   1.1910
0.6450   1.5714
0.6600   1.8652
0.6750   2.0686
0.6900   2.1789
0.7050   2.1947
0.7200   2.1158
0.7350   1.9432
0.7500   1.6792
0.7650   1.3271
0.7800   0.8916
0.7950   0.3779
0.8100   0.1885
0.8250   0.1854
0.8400   0.1819
0.8550   0.1781
0.8700   0.1741
0.8850   0.1699
0.9000   0.1655
0.9150   0.1609
0.9300   0.1563
0.9450   0.1516
0.9600   0.1469
0.9750   0.1422
0.9900   0.1376
/}\relax
\setsolid
\Red{\relax   
\plot
0.0000   0.0523
0.0150   0.0387
0.0300   0.0070
0.0450  -0.0400
0.0600  -0.0992
0.0750  -0.1665
0.0900  -0.2382
0.1050  -0.3107
0.1200  -0.3812
0.1350  -0.4473
0.1500  -0.5075
0.1650  -0.5605
0.1800  -0.6056
0.1950  -0.6423
0.2100  -0.6700
0.2250  -0.6885
0.2400  -0.6972
0.2550  -0.6957
0.2700  -0.6832
0.2850  -0.6588
0.3000  -0.6214
0.3150  -0.5695
0.3300  -0.5015
0.3450  -0.4153
0.3600  -0.3090
0.3750  -0.1810
0.3900  -0.0301
0.4050   0.1339
0.4200   0.3030
0.4350   0.4917
0.4500   0.6896
0.4650   0.8813
0.4800   1.0477
0.4950   1.1699
0.5100   1.2348
0.5250   1.2390
0.5400   1.1896
0.5550   1.1011
0.5700   0.9905
0.5850   0.8734
0.6000   0.7614
0.6150   1.1757
0.6300   1.5256
0.6450   1.8076
0.6600   2.0176
0.6750   2.1515
0.6900   2.2062
0.7050   2.1793
0.7200   2.0700
0.7350   1.8787
0.7500   1.6072
0.7650   1.2590
0.7800   0.8386
0.7950   0.3517
0.8100   0.2067
0.8250   0.2636
0.8400   0.3324
0.8550   0.4123
0.8700   0.5020
0.8850   0.5995
0.9000   0.7021
0.9150   0.8059
0.9300   0.9062
0.9450   0.9973
0.9600   1.0731
0.9750   1.1279
0.9900   1.1568
/}\relax\relax
\Green{\relax   
\plot
0.0000   0.4851
0.0150   0.4695
0.0300   0.4325
0.0450   0.3768
0.0600   0.3067
0.0750   0.2269
0.0900   0.1419
0.1050   0.0560
0.1200  -0.0272
0.1350  -0.1050
0.1500  -0.1754
0.1650  -0.2371
0.1800  -0.2889
0.1950  -0.3303
0.2100  -0.3606
0.2250  -0.3795
0.2400  -0.3863
0.2550  -0.3804
0.2700  -0.3610
0.2850  -0.3268
0.3000  -0.2766
0.3150  -0.2085
0.3300  -0.1206
0.3450  -0.0106
0.3600   0.1236
0.3750   0.2839
0.3900   0.4712
0.4050   0.6752
0.4200   0.8856
0.4350   1.1116
0.4500   1.3369
0.4650   1.5394
0.4800   1.6943
0.4950   1.7796
0.5100   1.7829
0.5250   1.7058
0.5400   1.5630
0.5550   1.3772
0.5700   1.1727
0.5850   0.9699
0.6000   0.7835
0.6150   1.1358
0.6300   1.4354
0.6450   1.6775
0.6600   1.8564
0.6750   1.9665
0.6900   2.0030
0.7050   1.9627
0.7200   1.8434
0.7350   1.6448
0.7500   1.3680
0.7650   1.0153
0.7800   0.5906
0.7950   0.0990
0.8100  -0.0521
0.8250  -0.0032
0.8400   0.0560
0.8550   0.1249
0.8700   0.2021
0.8850   0.2861
0.9000   0.3743
0.9150   0.4636
0.9300   0.5497
0.9450   0.6278
0.9600   0.6928
0.9750   0.7395
0.9900   0.7638
/}\relax\relax
\Blue{
\plot
0.0000   0.5807
0.0150   0.5647
0.0300   0.5263
0.0450   0.4686
0.0600   0.3957
0.0750   0.3125
0.0900   0.2238
0.1050   0.1341
0.1200   0.0469
0.1350  -0.0349
0.1500  -0.1093
0.1650  -0.1748
0.1800  -0.2304
0.1950  -0.2754
0.2100  -0.3094
0.2250  -0.3318
0.2400  -0.3422
0.2550  -0.3400
0.2700  -0.3245
0.2850  -0.2945
0.3000  -0.2489
0.3150  -0.1861
0.3300  -0.1042
0.3450  -0.0014
0.3600   0.1244
0.3750   0.2748
0.3900   0.4503
0.4050   0.6404
0.4200   0.8349
0.4350   1.0432
0.4500   1.2494
0.4650   1.4325
0.4800   1.5692
0.4950   1.6389
0.5100   1.6306
0.5250   1.5466
0.5400   1.4015
0.5550   1.2174
0.5700   1.0172
0.5850   0.8205
0.6000   0.6410
0.6150   1.0001
0.6300   1.3063
0.6450   1.5542
0.6600   1.7379
0.6750   1.8520
0.6900   1.8918
 0.7050   1.8537
0.7200   1.7359
0.7350   1.5379
0.7500   1.2608
0.7650   0.9072
0.7800   0.4809
0.7950  -0.0131
0.8100  -0.1675
0.8250  -0.1225
0.8400  -0.0679
0.8550  -0.0043
0.8700   0.0670
0.8850   0.1445
0.9000   0.2261
0.9150   0.3085
0.9300   0.3880
0.9450   0.4601
0.9600   0.5199
0.9750   0.5629
0.9900   0.5851
/}\relax
\endpicture
}
\setbox\figurefive=\vbox{\hsize=\xfiglen 
\beginpicture
\scriptsize
  \setcoordinatesystem units <\xfigdim,0.56\yfigdim>  point at 0 0
  \setplotarea x from 0 to 1, y from -0.5 to 2.2
  \axis bottom shiftedto y=0 ticks short numbered from 0 to 1 by 0.2 /
  \axis left ticks short numbered from 0.0 to 2.0 by 1.0 /
  \put {$q(x)$} [lt] at 0.025 2.2
\setlinear
\setsolid
\Black{\relax  
\plot
0.0000   0.1345
0.0150   0.1301
0.0300   0.1259
0.0450   0.1219
0.0600   0.1181
0.0750   0.1146
0.0900   0.1115
0.1050   0.1086
0.1200   0.1062
0.1350   0.1041
0.1500   0.1024
0.1650   0.1012
0.1800   0.1004
0.1950   0.1000
0.2100   0.1001
0.2250   0.1006
0.2400   0.1016
0.2550   0.1030
0.2700   0.1048
0.2850   0.1070
0.3000   0.1095
0.3150   0.1125
0.3300   0.1158
0.3450   0.1194
0.3600   0.1232
0.3750   0.1273
0.3900   0.1316
0.4050   0.1361
0.4200   0.1406
0.4350   0.1453
0.4500   0.1500
0.4650   0.1547
0.4800   0.1594
0.4950   0.1639
0.5100   0.1684
0.5250   0.1727
0.5400   0.1768
0.5550   0.1806
0.5700   0.1842
0.5850   0.1875
0.6000   0.1905
0.6150   0.7287
0.6300   1.1910
0.6450   1.5714
0.6600   1.8652
0.6750   2.0686
0.6900   2.1789
0.7050   2.1947
0.7200   2.1158
0.7350   1.9432
0.7500   1.6792
0.7650   1.3271
0.7800   0.8916
0.7950   0.3779
0.8100   0.1885
0.8250   0.1854
0.8400   0.1819
0.8550   0.1781
0.8700   0.1741
0.8850   0.1699
0.9000   0.1655
0.9150   0.1609
0.9300   0.1563
0.9450   0.1516
0.9600   0.1469
0.9750   0.1422
0.9900   0.1376
/}
\setsolid
\Red{\relax   
\!\!\!\!
\plot
0.0000   0.0883
0.0150   0.0848
0.0300   0.0750
0.0450   0.0599
0.0600   0.0411
0.0750   0.0203
0.0900  -0.0011
0.1050  -0.0217
0.1200  -0.0407
0.1350  -0.0574
0.1500  -0.0718
0.1650  -0.0836
0.1800  -0.0931
0.1950  -0.1004
0.2100  -0.1056
0.2250  -0.1089
0.2400  -0.1102
0.2550  -0.1095
0.2700  -0.1067
0.2850  -0.1014
0.3000  -0.0932
0.3150  -0.0814
0.3300  -0.0653
0.3450  -0.0439
0.3600  -0.0160
0.3750   0.0194
0.3900   0.0631
0.4050   0.2102
0.4200   0.5260
0.4350   0.8047
0.4500   1.0431
0.4650   1.2354
0.4800   1.3744
0.4950   1.4529
0.5100   1.4658
0.5250   1.4118
0.5400   1.2935
0.5550   1.1166
0.5700   0.8876
0.5850   0.6122
0.6000   0.2945
0.6150   0.2529
0.6300   0.2163
0.6450   0.1852
0.6600   0.1596
0.6750   0.1390
0.6900   0.1230
0.7050   0.1110
0.7200   0.1027
0.7350   0.0978
0.7500   0.0960
0.7650   0.0974
0.7800   0.1020
0.7950   0.1100
0.8100   0.1234
0.8250   0.1418
0.8400   0.1647
0.8550   0.1922
0.8700   0.2242
0.8850   0.2604
0.9000   0.3000
0.9150   0.3416
0.9300   0.3832
0.9450   0.4223
0.9600   0.4560
0.9750   0.4813
0.9900   0.4958
/}\relax
\Green{\relax   
\!\!
\plot
0.0000  -0.1550
0.0150  -0.1550
0.0300  -0.1543
0.0450  -0.1533
0.0600  -0.1517
0.0750  -0.1495
0.0900  -0.1469
0.1050  -0.1437
0.1200  -0.1402
0.1350  -0.1365
0.1500  -0.1325
0.1650  -0.1285
0.1800  -0.1246
0.1950  -0.1207
0.2100  -0.1168
0.2250  -0.1130
0.2400  -0.1092
0.2550  -0.1051
0.2700  -0.1007
0.2850  -0.0955
0.3000  -0.0893
0.3150  -0.0814
0.3300  -0.0712
0.3450  -0.0578
0.3600  -0.0404
0.3750  -0.0177
0.3900   0.0113
0.4050   0.1436
0.4200   0.4462
0.4350   0.7115
0.4500   0.9373
0.4650   1.1200
0.4800   1.2542
0.4950   1.3343
0.5100   1.3555
0.5250   1.3159
0.5400   1.2161
0.5550   1.0596
0.5700   0.8506
0.5850   0.5935
0.6000   0.2914
0.6150   0.2657
0.6300   0.2428
0.6450   0.2234
0.6600   0.2075
0.6750   0.1951
0.6900   0.1859
0.7050   0.1798
0.7200   0.1764
0.7350   0.1756
0.7500   0.1774
0.7650   0.1816
0.7800   0.1884
0.7950   0.1979
0.8100   0.2109
0.8250   0.2275
0.8400   0.2473
0.8550   0.2705
0.8700   0.2969
0.8850   0.3263
0.9000   0.3580
0.9150   0.3910
0.9300   0.4238
0.9450   0.4544
0.9600   0.4805
0.9750   0.5001
0.9900   0.5113
/}\relax
\Blue{\relax  
\!\!
\plot
0.0000   0.2619
0.0150   0.2572
0.0300   0.2439
0.0450   0.2233
0.0600   0.1974
0.0750   0.1685
0.0900   0.1387
0.1050   0.1096
0.1200   0.0826
0.1350   0.0583
0.1500   0.0373
0.1650   0.0197
0.1800   0.0054
0.1950  -0.0058
0.2100  -0.0139
0.2250  -0.0191
0.2400  -0.0215
0.2550  -0.0210
0.2700  -0.0173
0.2850  -0.0102
0.3000   0.0008
0.3150   0.0165
0.3300   0.0376
0.3450   0.0653
0.3600   0.1006
0.3750   0.1445
0.3900   0.1977
0.4050   0.3552
0.4200   0.6811
0.4350   0.9682
0.4500   1.2110
0.4650   1.4018
0.4800   1.5316
0.4950   1.5925
0.5100   1.5804
0.5250   1.4961
0.5400   1.3453
0.5550   1.1364
0.5700   0.8782
0.5850   0.5778
0.6000   0.2399
0.6150   0.1827
0.6300   0.1345
0.6450   0.0951
0.6600   0.0635
0.6750   0.0388
0.6900   0.0199
0.7050   0.0058
0.7200  -0.0042
0.7350  -0.0108
0.7500  -0.0144
0.7650  -0.0152
0.7800  -0.0135
0.7950  -0.0092
0.8100  -0.0008
0.8250   0.0114
0.8400   0.0269
0.8550   0.0457
0.8700   0.0679
0.8850   0.0932
0.9000   0.1211
0.9150   0.1507
0.9300   0.1805
0.9450   0.2087
0.9600   0.2331
0.9750   0.2515
0.9900   0.2620
/}\relax
\endpicture
}
%
%
\vskip20pt
\setbox\figureutwo =\vbox{
\hbox to \hsize{\copy\figureone\hss\copy\figuretwo\hss\copy\figurethree}
\hbox to \hsize{\copy\figurefour\hss\copy\figurefive\hss\copy\figuresix}
} 
%
%
%
\setbox\figureone=\vbox{\hsize=\xfiglen
\beginpicture
\scriptsize
  \setcoordinatesystem units <\xfigdim,\yfigdim>  point at 0 0
  \setplotarea x from 0 to 1, y from -0.2 to 1.4
  \axis bottom shiftedto y=0 ticks short numbered from 0 to 1 by 0.2 /
  \axis left ticks short numbered from 0.2 to 1.4 by 0.4 /
 \put {\phantom{xx}} [l] at 1 1.1    
  \put {$p(x)$} [lt] at 0.025 1.4
\setquadratic
\setlinear
\setsolid
\Black{\relax  
\plot
  0.0000   0.1000
  0.0150   0.1000
  0.0300   0.1000
  0.0450   0.1000
  0.0600   0.1000
  0.0750   0.1000
  0.0900   0.1000
  0.1050   0.1000
  0.1200   0.1000
  0.1350   0.1000
  0.1500   0.1000
  0.1650   0.1000
  0.1800   0.1000
  0.1950   0.1000
  0.2100   0.1000
  0.2250   0.1000
  0.2400   0.1000
  0.2550   0.1000
  0.2700   0.1000
  0.2850   0.1000
  0.3000   0.1000
  0.3150   0.1000
  0.3300   0.1000
  0.3450   0.1000
  0.3600   0.1000
  0.3750   0.1000
  0.3900   0.1000
  0.4050   0.1975
  0.4200   0.4600
  0.4350   0.6775
  0.4500   0.8500
  0.4650   0.9775
  0.4800   1.0600
  0.4950   1.0975
  0.5100   1.0900
  0.5250   1.0375
  0.5400   0.9400
  0.5550   0.7975
  0.5700   0.6100
  0.5850   0.3775
  0.6000   0.1000
  0.6150   0.1000
  0.6300   0.1000
  0.6450   0.1000
  0.6600   0.1000
  0.6750   0.1000
  0.6900   0.1000
  0.7050   0.1000
  0.7200   0.1000
  0.7350   0.1000
  0.7500   0.1000
  0.7650   0.1000
  0.7800   0.1000
  0.7950   0.1000
  0.8100   0.1000
  0.8250   0.1000
  0.8400   0.1000
  0.8550   0.1000
  0.8700   0.1000
  0.8850   0.1000
  0.9000   0.1000
  0.9150   0.1000
  0.9300   0.1000
  0.9450   0.1000
  0.9600   0.1000
  0.9750   0.1000
  0.9900   0.1000
/}\relax
%
\Red{\relax   
\plot
  0.0000   0.1143
  0.0150   0.1142
  0.0300   0.1138
  0.0450   0.1133
  0.0600   0.1125
  0.0750   0.1117
  0.0900   0.1108
  0.1050   0.1099
  0.1200   0.1090
  0.1350   0.1082
  0.1500   0.1075
  0.1650   0.1068
  0.1800   0.1063
  0.1950   0.1058
  0.2100   0.1054
  0.2250   0.1050
  0.2400   0.1047
  0.2550   0.1045
  0.2700   0.1043
  0.2850   0.1042
  0.3000   0.1041
  0.3150   0.1041
  0.3300   0.1041
  0.3450   0.1041
  0.3600   0.1042
  0.3750   0.1043
  0.3900   0.1045
  0.4050   0.2022
  0.4200   0.4651
  0.4350   0.6829
  0.4500   0.8558
  0.4650   0.9836
  0.4800   1.0663
  0.4950   1.1039
  0.5100   1.0964
  0.5250   1.0437
  0.5400   0.9460
  0.5550   0.8031
  0.5700   0.6152
  0.5850   0.3823
  0.6000   0.1044
  0.6150   0.1041
  0.6300   0.1039
  0.6450   0.1037
  0.6600   0.1036
  0.6750   0.1035
  0.6900   0.1034
  0.7050   0.1034
  0.7200   0.1034
  0.7350   0.1035
  0.7500   0.1036
  0.7650   0.1038
  0.7800   0.1040
  0.7950   0.1042
  0.8100   0.1045
  0.8250   0.1049
  0.8400   0.1053
  0.8550   0.1058
  0.8700   0.1064
  0.8850   0.1070
  0.9000   0.1076
  0.9150   0.1083
  0.9300   0.1089
  0.9450   0.1096
  0.9600   0.1101
  0.9750   0.1105
  0.9900   0.1107
/}\relax
%
\Green{\relax   
\plot
  0.0000   0.1033
  0.0150   0.1033
  0.0300   0.1032
  0.0450   0.1030
  0.0600   0.1028
  0.0750   0.1026
  0.0900   0.1024
  0.1050   0.1021
  0.1200   0.1019
  0.1350   0.1017
  0.1500   0.1015
  0.1650   0.1014
  0.1800   0.1013
  0.1950   0.1011
  0.2100   0.1011
  0.2250   0.1010
  0.2400   0.1009
  0.2550   0.1009
  0.2700   0.1009
  0.2850   0.1008
  0.3000   0.1008
  0.3150   0.1009
  0.3300   0.1009
  0.3450   0.1010
  0.3600   0.1010
  0.3750   0.1011
  0.3900   0.1012
  0.4050   0.1989
  0.4200   0.4615
  0.4350   0.6792
  0.4500   0.8518
  0.4650   0.9795
  0.4800   1.0621
  0.4950   1.0996
  0.5100   1.0921
  0.5250   1.0396
  0.5400   0.9420
  0.5550   0.7993
  0.5700   0.6117
  0.5850   0.3791
  0.6000   0.1014
  0.6150   0.1013
  0.6300   0.1012
  0.6450   0.1011
  0.6600   0.1011
  0.6750   0.1010
  0.6900   0.1010
  0.7050   0.1010
  0.7200   0.1010
  0.7350   0.1010
  0.7500   0.1010
  0.7650   0.1011
  0.7800   0.1011
  0.7950   0.1012
  0.8100   0.1013
  0.8250   0.1014
  0.8400   0.1016
  0.8550   0.1018
  0.8700   0.1020
  0.8850   0.1022
  0.9000   0.1024
  0.9150   0.1027
  0.9300   0.1029
  0.9450   0.1031
  0.9600   0.1033
  0.9750   0.1035
  0.9900   0.1035
/}\relax
%
\Blue{\relax  
\plot
  0.0000   0.1003
  0.0150   0.1003
  0.0300   0.1002
  0.0450   0.1002
  0.0600   0.1002
  0.0750   0.1001
  0.0900   0.1001
  0.1050   0.1000
  0.1200   0.1000
  0.1350   0.1000
  0.1500   0.0999
  0.1650   0.0999
  0.1800   0.0999
  0.1950   0.0999
  0.2100   0.0999
  0.2250   0.0999
  0.2400   0.0999
  0.2550   0.0999
  0.2700   0.0999
  0.2850   0.0999
  0.3000   0.0999
  0.3150   0.0999
  0.3300   0.0999
  0.3450   0.0999
  0.3600   0.1000
  0.3750   0.1000
  0.3900   0.1001
  0.4050   0.1977
  0.4200   0.4602
  0.4350   0.6778
  0.4500   0.8504
  0.4650   0.9780
  0.4800   1.0605
  0.4950   1.0981
  0.5100   1.0906
  0.5250   1.0380
  0.5400   0.9405
  0.5550   0.7979
  0.5700   0.6103
  0.5850   0.3777
  0.6000   0.1001
  0.6150   0.1001
  0.6300   0.1000
  0.6450   0.1000
  0.6600   0.0999
  0.6750   0.0999
  0.6900   0.0999
  0.7050   0.0999
  0.7200   0.0999
  0.7350   0.0999
  0.7500   0.0999
  0.7650   0.0999
  0.7800   0.0999
  0.7950   0.0999
  0.8100   0.0999
  0.8250   0.0999
  0.8400   0.0999
  0.8550   0.0999
  0.8700   0.1000
  0.8850   0.1000
  0.9000   0.1000
  0.9150   0.1001
  0.9300   0.1001
  0.9450   0.1002
  0.9600   0.1002
  0.9750   0.1003
  0.9900   0.1003
/}\relax
\endpicture
}   
\setbox\figuretwo=\vbox{\hsize=\xfiglen 
\beginpicture
\scriptsize
  \setcoordinatesystem units <\xfigdim,\yfigdim>  point at 0 0
  \setplotarea x from 0 to 1, y from -0.2 to 1.4
  \axis bottom shiftedto y=0.0 ticks short numbered from 0 to 1 by 0.2 /
  \axis left ticks short numbered from -0.2 to 1.4 by 0.4 /
  \put {$p(x)$} [lt] at 0.025 1.4
\setlinear
\setsolid
\Black{\relax  
\plot
0.0000   0.1000
0.0150   0.1000
0.0300   0.1000
0.0450   0.1000
0.0600   0.1000
0.0750   0.1000
0.0900   0.1000
0.1050   0.1000
0.1200   0.1000
0.1350   0.1000
0.1500   0.1000
0.1650   0.1000
0.1800   0.1000
0.1950   0.1000
0.2100   0.1000
0.2250   0.1000
0.2400   0.1000
0.2550   0.1000
0.2700   0.1000
0.2850   0.1000
0.3000   0.1000
0.3150   0.1000
0.3300   0.1000
0.3450   0.1000
0.3600   0.1000
0.3750   0.1000
0.3900   0.1000
0.4050   0.1975
0.4200   0.4600
0.4350   0.6775
0.4500   0.8500
0.4650   0.9775
0.4800   1.0600
0.4950   1.0975
0.5100   1.0900
0.5250   1.0375
0.5400   0.9400
0.5550   0.7975
0.5700   0.6100
0.5850   0.3775
0.6000   0.1000
0.6150   0.1000
0.6300   0.1000
0.6450   0.1000
0.6600   0.1000
0.6750   0.1000
0.6900   0.1000
0.7050   0.1000
0.7200   0.1000
0.7350   0.1000
0.7500   0.1000
0.7650   0.1000
0.7800   0.1000
0.7950   0.1000
0.8100   0.1000
0.8250   0.1000
0.8400   0.1000
0.8550   0.1000
0.8700   0.1000
0.8850   0.1000
0.9000   0.1000
0.9150   0.1000
0.9300   0.1000
0.9450   0.1000
0.9600   0.1000
0.9750   0.1000
0.9900   0.1000
/}\relax
\setsolid
%
\Red{\relax   
\plot
0.0000   0.0586
0.0150   0.0584
0.0300   0.0577
0.0450   0.0567
0.0600   0.0556
0.0750   0.0545
0.0900   0.0537
0.1050   0.0530
0.1200   0.0527
0.1350   0.0525
0.1500   0.0526
0.1650   0.0528
0.1800   0.0531
0.1950   0.0534
0.2100   0.0536
0.2250   0.0538
0.2400   0.0539
0.2550   0.0539
0.2700   0.0538
0.2850   0.0537
0.3000   0.0536
0.3150   0.0535
0.3300   0.0536
0.3450   0.0541
0.3600   0.0550
0.3750   0.0567
0.3900   0.0592
0.4050   0.1584
0.4200   0.4204
0.4350   0.6400
0.4500   0.8173
0.4650   0.9519
0.4800   1.0433
0.4950   1.0904
0.5100   1.0921
0.5250   1.0476
0.5400   0.9566
0.5550   0.8191
0.5700   0.6355
0.5850   0.4064
0.6000   0.1323
0.6150   0.1306
0.6300   0.1285
0.6450   0.1263
0.6600   0.1243
0.6750   0.1224
0.6900   0.1208
0.7050   0.1195
0.7200   0.1185
0.7350   0.1178
0.7500   0.1175
0.7650   0.1177
0.7800   0.1183
0.7950   0.1196
0.8100   0.1222
0.8250   0.1262
0.8400   0.1312
0.8550   0.1374
0.8700   0.1448
0.8850   0.1535
0.9000   0.1632
0.9150   0.1737
0.9300   0.1845
0.9450   0.1949
0.9600   0.2039
0.9750   0.2108
0.9900   0.2148
/}\relax
%
\Green{\relax   
\plot
0.0000   0.1222
0.0150   0.1214
0.0300   0.1190
0.0450   0.1154
0.0600   0.1109
0.0750   0.1061
0.0900   0.1014
0.1050   0.0970
0.1200   0.0931
0.1350   0.0898
0.1500   0.0871
0.1650   0.0850
0.1800   0.0834
0.1950   0.0823
0.2100   0.0816
0.2250   0.0812
0.2400   0.0811
0.2550   0.0813
0.2700   0.0820
0.2850   0.0830
0.3000   0.0845
0.3150   0.0868
0.3300   0.0898
0.3450   0.0939
0.3600   0.0993
0.3750   0.1063
0.3900   0.1151
0.4050   0.2214
0.4200   0.4911
0.4350   0.7181
0.4500   0.9013
0.4650   1.0394
0.4800   1.1304
0.4950   1.1729
0.5100   1.1658
0.5250   1.1092
0.5400   1.0041
0.5550   0.8523
0.5700   0.6554
0.5850   0.4147
0.6000   0.1312
0.6150   0.1223
0.6300   0.1149
0.6450   0.1089
0.6600   0.1043
0.6750   0.1006
0.6900   0.0979
0.7050   0.0958
0.7200   0.0943
0.7350   0.0932
0.7500   0.0925
0.7650   0.0922
0.7800   0.0921
0.7950   0.0924
0.8100   0.0937
0.8250   0.0958
0.8400   0.0986
0.8550   0.1021
0.8700   0.1064
0.8850   0.1114
0.9000   0.1172
0.9150   0.1235
0.9300   0.1301
0.9450   0.1365
0.9600   0.1421
0.9750   0.1465
0.9900   0.1490
/}\relax
%
\Blue{\relax  
\plot
0.0000   0.1185
0.0150   0.1177
0.0300   0.1154
0.0450   0.1119
0.0600   0.1076
0.0750   0.1030
0.0900   0.0984
0.1050   0.0941
0.1200   0.0904
0.1350   0.0872
0.1500   0.0846
0.1650   0.0826
0.1800   0.0810
0.1950   0.0799
0.2100   0.0791
0.2250   0.0786
0.2400   0.0785
0.2550   0.0785
0.2700   0.0789
0.2850   0.0797
0.3000   0.0808
0.3150   0.0825
0.3300   0.0850
0.3450   0.0883
0.3600   0.0927
0.3750   0.0986
0.3900   0.1059
0.4050   0.2105
0.4200   0.4783
0.4350   0.7031
0.4500   0.8840
0.4650   1.0198
0.4800   1.1087
0.4950   1.1495
0.5100   1.1413
0.5250   1.0843
0.5400   0.9795
0.5550   0.8285
0.5700   0.6328
0.5850   0.3936
0.6000   0.1116
0.6150   0.1041
0.6300   0.0981
0.6450   0.0934
0.6600   0.0898
0.6750   0.0870
0.6900   0.0850
0.7050   0.0834
0.7200   0.0822
0.7350   0.0813
0.7500   0.0806
0.7650   0.0801
0.7800   0.0797
0.7950   0.0795
0.8100   0.0801
0.8250   0.0814
0.8400   0.0830
0.8550   0.0852
0.8700   0.0880
0.8850   0.0913
0.9000   0.0952
0.9150   0.0995
0.9300   0.1041
0.9450   0.1086
0.9600   0.1126
0.9750   0.1157
/}\relax
\endpicture
}   
\setbox\figurethree=\vbox{\hsize=\xfiglen 
\beginpicture
\scriptsize
  \setcoordinatesystem units <\xfigdim,\yfigdim>  point at 0 0
  \setplotarea x from 0 to 1, y from -0.2 to 1.4
  \axis bottom shiftedto y=0.0 ticks short numbered from 0 to 1 by 0.2 /
  \axis left ticks short numbered from -0.2 to 1.4 by 0.4 /
  \put {$p(x)$} [lt] at 0.025 1.4
\setlinear
\setsolid 
\Black{\relax  
\plot
0.0000   0.1000
0.0150   0.1000
0.0300   0.1000
0.0450   0.1000
0.0600   0.1000
0.0750   0.1000
0.0900   0.1000
0.1050   0.1000
0.1200   0.1000
0.1350   0.1000
0.1500   0.1000
0.1650   0.1000
0.1800   0.1000
0.1950   0.1000
0.2100   0.1000
0.2250   0.1000
0.2400   0.1000
0.2550   0.1000
0.2700   0.1000
0.2850   0.1000
0.3000   0.1000
0.3150   0.1000
0.3300   0.1000
0.3450   0.1000
0.3600   0.1000
0.3750   0.1000
0.3900   0.1000
0.4050   0.1975
0.4200   0.4600
0.4350   0.6775
0.4500   0.8500
0.4650   0.9775
0.4800   1.0600
0.4950   1.0975
0.5100   1.0900
0.5250   1.0375
0.5400   0.9400
0.5550   0.7975
0.5700   0.6100
0.5850   0.3775
0.6000   0.1000
0.6150   0.1000
0.6300   0.1000
0.6450   0.1000
0.6600   0.1000
0.6750   0.1000
0.6900   0.1000
0.7050   0.1000
0.7200   0.1000
0.7350   0.1000
0.7500   0.1000
0.7650   0.1000
0.7800   0.1000
0.7950   0.1000
0.8100   0.1000
0.8250   0.1000
0.8400   0.1000
0.8550   0.1000
0.8700   0.1000
0.8850   0.1000
0.9000   0.1000
0.9150   0.1000
0.9300   0.1000
0.9450   0.1000
0.9600   0.1000
0.9750   0.1000
0.9900   0.1000
/}
%
\Red{\relax   
\!\!\!\!
\plot
0.0000  -0.0597
0.0150  -0.0586
0.0300  -0.0553
0.0450  -0.0502
0.0600  -0.0438
0.0750  -0.0364
0.0900  -0.0288
0.1050  -0.0212
0.1200  -0.0140
0.1350  -0.0074
0.1500  -0.0017
0.1650   0.0030
0.1800   0.0069
0.1950   0.0097
0.2100   0.0115
0.2250   0.0124
0.2400   0.0121
0.2550   0.0108
0.2700   0.0084
0.2850   0.0046
0.3000  -0.0004
0.3150  -0.0070
0.3300  -0.0153
0.3450  -0.0255
0.3600  -0.0377
0.3750  -0.0521
0.3900  -0.0685
0.4050   0.0097
0.4200   0.2504
0.4350   0.4482
0.4500   0.6064
0.4650   0.7282
0.4800   0.8165
0.4950   0.8727
0.5100   0.8957
0.5250   0.8825
0.5400   0.8285
0.5550   0.7294
0.5700   0.5817
0.5850   0.3834
0.6000   0.1339
0.6150   0.1528
0.6300   0.1662
0.6450   0.1752
0.6600   0.1810
0.6750   0.1847
0.6900   0.1869
0.7050   0.1884
0.7200   0.1897
0.7350   0.1910
0.7500   0.1928
0.7650   0.1952
0.7800   0.1987
0.7950   0.2035
0.8100   0.2103
0.8250   0.2191
0.8400   0.2301
0.8550   0.2432
0.8700   0.2586
0.8850   0.2761
0.9000   0.2953
0.9150   0.3157
0.9300   0.3363
0.9450   0.3557
0.9600   0.3726
0.9750   0.3854
0.9900   0.3927
/}\relax
\Green{\relax   
\plot
0.0000  -0.0051
0.0150  -0.0043
0.0300  -0.0021
0.0450   0.0015
0.0600   0.0060
0.0750   0.0111
0.0900   0.0165
0.1050   0.0219
0.1200   0.0272
0.1350   0.0320
0.1500   0.0364
0.1650   0.0403
0.1800   0.0436
0.1950   0.0464
0.2100   0.0488
0.2250   0.0507
0.2400   0.0522
0.2550   0.0533
0.2700   0.0541
0.2850   0.0548
0.3000   0.0552
0.3150   0.0557
0.3300   0.0562
0.3450   0.0571
0.3600   0.0584
0.3750   0.0607
0.3900   0.0642
0.4050   0.1657
0.4200   0.4325
0.4350   0.6569
0.4500   0.8387
0.4650   0.9774
0.4800   1.0720
0.4950   1.1211
0.5100   1.1235
0.5250   1.0783
0.5400   0.9853
0.5550   0.8448
0.5700   0.6576
0.5850   0.4244
0.6000   0.1460
0.6150   0.1425
0.6300   0.1391
0.6450   0.1360
0.6600   0.1334
0.6750   0.1313
0.6900   0.1298
0.7050   0.1287
0.7200   0.1281
0.7350   0.1281
0.7500   0.1285
0.7650   0.1295
0.7800   0.1311
0.7950   0.1334
0.8100   0.1366
0.8250   0.1409
0.8400   0.1461
0.8550   0.1524
0.8700   0.1597
0.8850   0.1681
0.9000   0.1773
0.9150   0.1872
0.9300   0.1972
0.9450   0.2067
0.9600   0.2150
0.9750   0.2212
0.9900   0.2248
/}\relax
\setlinear
\Blue{\relax  
\plot
0.0000   0.0979
0.0150   0.0976
0.0300   0.0968
0.0450   0.0955
0.0600   0.0940
0.0750   0.0924
0.0900   0.0908
0.1050   0.0894
0.1200   0.0882
0.1350   0.0872
0.1500   0.0865
0.1650   0.0859
0.1800   0.0855
0.1950   0.0852
0.2100   0.0850
0.2250   0.0849
0.2400   0.0848
0.2550   0.0849
0.2700   0.0850
0.2850   0.0852
0.3000   0.0855
0.3150   0.0861
0.3300   0.0869
0.3450   0.0881
0.3600   0.0898
0.3750   0.0920
0.3900   0.0950
0.4050   0.1952
0.4200   0.4595
0.4350   0.6799
0.4500   0.8559
0.4650   0.9869
0.4800   1.0722
0.4950   1.1112
0.5100   1.1034
0.5250   1.0489
0.5400   0.9482
0.5550   0.8021
0.5700   0.6113
0.5850   0.3763
0.6000   0.0974
0.6150   0.0944
0.6300   0.0921
0.6450   0.0903
0.6600   0.0890
0.6750   0.0881
0.6900   0.0874
0.7050   0.0869
0.7200   0.0865
0.7350   0.0862
0.7500   0.0860
0.7650   0.0857
0.7800   0.0855
0.7950   0.0853
0.8100   0.0854
0.8250   0.0858
0.8400   0.0863
0.8550   0.0870
0.8700   0.0878
0.8850   0.0889
0.9000   0.0903
0.9150   0.0918
0.9300   0.0934
0.9450   0.0950
0.9600   0.0965
0.9750   0.0976
0.9900   0.0983
/}\relax
\endpicture
}
%
%
%
\setbox\figurefour=\vbox{\hsize=\xfiglen
\beginpicture
\scriptsize
  \setcoordinatesystem units <\xfigdim,0.17\yfigdim>  point at 0 0
  \setplotarea x from 0 to 1, y from -1 to 7
  \axis bottom shiftedto y=0 ticks short numbered from 0 to 1 by 0.2 /
  \axis left ticks short unlabeled from -1 to 7 by 1 /
  \axis left ticks short numbered from 0 to 7 by 2 /
  \put {$q(x)$} [lt] at 0.025 7
\setlinear
\setsolid 
\Black{\relax  
\plot
0.0000   0.2716
0.0150   0.2541
0.0300   0.2378
0.0450   0.2229
0.0600   0.2093
0.0750   0.1972
0.0900   0.1864
0.1050   0.1771
0.1200   0.1691
0.1350   0.1626
0.1500   0.1574
0.1650   0.1536
0.1800   0.1512
0.1950   0.1501
0.2100   0.1503
0.2250   0.1519
0.2400   0.1547
0.2550   0.1590
0.2700   0.1646
0.2850   0.1716
0.3000   0.1800
0.3150   0.1898
0.3300   0.2010
0.3450   0.2137
0.3600   0.2277
0.3750   0.2431
0.3900   0.2598
0.4050   0.2776
0.4200   0.2967
0.4350   0.3167
0.4500   0.3375
0.4650   0.3590
0.4800   0.3810
0.4950   0.4032
0.5100   0.4254
0.5250   0.4474
0.5400   0.4688
0.5550   0.4895
0.5700   0.5091
0.5850   0.5274
0.6000   0.5441
0.6150   2.1100
0.6300   3.4879
0.6450   4.6446
0.6600   5.5518
0.6750   6.1867
0.6900   6.5335
0.7050   6.5834
0.7200   6.3349
0.7350   5.7946
0.7500   4.9760
0.7650   3.8996
0.7800   2.5919
0.7950   1.0848
0.8100   0.5331
0.8250   0.5153
0.8400   0.4962
0.8550   0.4758
0.8700   0.4546
0.8850   0.4328
0.9000   0.4106
0.9150   0.3884
0.9300   0.3663
0.9450   0.3446
0.9600   0.3235
0.9750   0.3032
0.9900   0.2839
/}
\setsolid
\Red{\relax   
\!\!\!\!
\plot
0.0000  -0.9174
0.0150  -0.9326
0.0300  -0.9431
0.0450  -0.9491
0.0600  -0.9514
0.0750  -0.9509
0.0900  -0.9487
0.1050  -0.9457
0.1200  -0.9427
0.1350  -0.9405
0.1500  -0.9397
0.1650  -0.9407
0.1800  -0.9440
0.1950  -0.9498
0.2100  -0.9586
0.2250  -0.9707
0.2400  -0.9865
0.2550  -1.0064
0.2700  -1.0309
0.2850  -1.0607
0.3000  -1.0964
0.3150  -1.1388
0.3300  -1.1887
0.3450  -1.2469
0.3600  -1.3139
0.3750  -1.3899
0.3900  -1.4743
0.4050  -1.5749
0.4200  -1.6889
0.4350  -1.7777
0.4500  -1.8220
0.4650  -1.8025
0.4800  -1.7037
0.4950  -1.5192
0.5100  -1.2555
0.5250  -0.9312
0.5400  -0.5724
0.5550  -0.2057
0.5700   0.1474
0.5850   0.4725
0.6000   0.7616
0.6150   2.5337
0.6300   4.0762
0.6450   5.3618
0.6600   6.3681
0.6750   7.0781
0.6900   7.4811
0.7050   7.5724
0.7200   7.3548
0.7350   6.8379
0.7500   6.0386
0.7650   4.9803
0.7800   3.6926
0.7950   2.2101
0.8100   1.6944
0.8250   1.7234
0.8400   1.7599
0.8550   1.8032
0.8700   1.8527
0.8850   1.9070
0.9000   1.9644
0.9150   2.0221
0.9300   2.0771
0.9450   2.1254
0.9600   2.1628
0.9750   2.1856
0.9900   2.1905
/}\relax
\Green{\relax   
\plot
0.0000  -0.5738
0.0150  -0.5891
0.0300  -0.5999
0.0450  -0.6062
0.0600  -0.6085
0.0750  -0.6075
0.0900  -0.6038
0.1050  -0.5980
0.1200  -0.5907
0.1350  -0.5822
0.1500  -0.5728
0.1650  -0.5627
0.1800  -0.5519
0.1950  -0.5404
0.2100  -0.5280
0.2250  -0.5145
0.2400  -0.4998
0.2550  -0.4834
0.2700  -0.4648
0.2850  -0.4436
0.3000  -0.4192
0.3150  -0.3906
0.3300  -0.3568
0.3450  -0.3167
0.3600  -0.2688
0.3750  -0.2115
0.3900  -0.1428
0.4050  -0.0709
0.4200  -0.0001
0.4350   0.0905
0.4500   0.1976
0.4650   0.3162
0.4800   0.4389
0.4950   0.5572
0.5100   0.6631
0.5250   0.7513
0.5400   0.8203
0.5550   0.8723
0.5700   0.9118
0.5850   0.9444
0.6000   0.9749
0.6150   2.5291
0.6300   3.8938
0.6450   5.0379
0.6600   5.9340
0.6750   6.5604
0.6900   6.9013
0.7050   6.9483
0.7200   6.7002
0.7350   6.1635
0.7500   5.3520
0.7650   4.2863
0.7800   2.9933
0.7950   1.5048
0.8100   0.9795
0.8250   0.9944
0.8400   1.0124
0.8550   1.0333
0.8700   1.0570
0.8850   1.0827
0.9000   1.1096
0.9150   1.1362
0.9300   1.1607
0.9450   1.1811
0.9600   1.1951
0.9750   1.2005
0.9900   1.1956
/}\relax
\Blue{\relax  
\!\!
\plot
0.0000   0.2425
0.0150   0.2234
0.0300   0.2013
0.0450   0.1772
0.0600   0.1518
0.0750   0.1259
0.0900   0.1003
0.1050   0.0756
0.1200   0.0526
0.1350   0.0317
0.1500   0.0132
0.1650  -0.0025
0.1800  -0.0154
0.1950  -0.0252
0.2100  -0.0319
0.2250  -0.0352
0.2400  -0.0352
0.2550  -0.0316
0.2700  -0.0243
0.2850  -0.0130
0.3000   0.0025
0.3150   0.0226
0.3300   0.0477
0.3450   0.0782
0.3600   0.1144
0.3750   0.1568
0.3900   0.2055
0.4050   0.2504
0.4200   0.2854
0.4350   0.3296
0.4500   0.3787
0.4650   0.4277
0.4800   0.4711
0.4950   0.5040
0.5100   0.5240
0.5250   0.5313
0.5400   0.5288
0.5550   0.5211
0.5700   0.5133
0.5850   0.5099
0.6000   0.5140
0.6150   2.0497
0.6300   3.4016
0.6450   4.5365
0.6600   5.4258
0.6750   6.0463
0.6900   6.3817
0.7050   6.4226
0.7200   6.1674
0.7350   5.6222
0.7500   4.8004
0.7650   3.7225
0.7800   2.4148
0.7950   0.9090
0.8100   0.3635
0.8250   0.3553
0.8400   0.3473
0.8550   0.3397
0.8700   0.3324
0.8850   0.3255
0.9000   0.3188
0.9150   0.3119
0.9300   0.3046
0.9450   0.2963
0.9600   0.2863
0.9750   0.2742
0.9900   0.2594
/}\relax
\endpicture
}   
\setbox\figurefive=\vbox{\hsize=\xfiglen
\beginpicture
\scriptsize
  \setcoordinatesystem units <\xfigdim,0.17\yfigdim>  point at 0 0
  \setplotarea x from 0 to 1, y from -1 to 7
  \axis bottom shiftedto y=0 ticks short numbered from 0 to 1 by 0.2 /
  \axis left ticks short unlabeled from -1 to 7.0 by 1 /
  \axis left ticks short numbered from 0 to 6 by 2 /
  \put {$q(x)$} [lt] at 0.025 7
\setlinear
\setsolid 
\Black{\relax  
\plot
0.0000   0.2716
0.0150   0.2541
0.0300   0.2378
0.0450   0.2229
0.0600   0.2093
0.0750   0.1972
0.0900   0.1864
0.1050   0.1771
0.1200   0.1691
0.1350   0.1626
0.1500   0.1574
0.1650   0.1536
0.1800   0.1512
0.1950   0.1501
0.2100   0.1503
0.2250   0.1519
0.2400   0.1547
0.2550   0.1590
0.2700   0.1646
0.2850   0.1716
0.3000   0.1800
0.3150   0.1898
0.3300   0.2010
0.3450   0.2137
0.3600   0.2277
0.3750   0.2431
0.3900   0.2598
0.4050   0.2776
0.4200   0.2967
0.4350   0.3167
0.4500   0.3375
0.4650   0.3590
0.4800   0.3810
0.4950   0.4032
0.5100   0.4254
0.5250   0.4474
0.5400   0.4688
0.5550   0.4895
0.5700   0.5091
0.5850   0.5274
0.6000   0.5441
0.6150   2.1100
0.6300   3.4879
0.6450   4.6446
0.6600   5.5518
0.6750   6.1867
0.6900   6.5335
0.7050   6.5834
0.7200   6.3349
0.7350   5.7946
0.7500   4.9760
0.7650   3.8996
0.7800   2.5919
0.7950   1.0848
0.8100   0.5331
0.8250   0.5153
0.8400   0.4962
0.8550   0.4758
0.8700   0.4546
0.8850   0.4328
0.9000   0.4106
0.9150   0.3884
0.9300   0.3663
0.9450   0.3446
0.9600   0.3235
0.9750   0.3032
0.9900   0.2839
/}
\setsolid
\relax
\Red{\relax   
\!\!\!\!
\plot
0.0000  -0.0724
0.0150  -0.0927
0.0300  -0.1181
0.0450  -0.1475
0.0600  -0.1797
0.0750  -0.2136
0.0900  -0.2480
0.1050  -0.2818
0.1200  -0.3141
0.1350  -0.3441
0.1500  -0.3713
0.1650  -0.3952
0.1800  -0.4156
0.1950  -0.4322
0.2100  -0.4449
0.2250  -0.4535
0.2400  -0.4579
0.2550  -0.4578
0.2700  -0.4531
0.2850  -0.4435
0.3000  -0.4286
0.3150  -0.4080
0.3300  -0.3813
0.3450  -0.3479
0.3600  -0.3071
0.3750  -0.2586
0.3900  -0.2022
0.4050  -0.1569
0.4200  -0.1328
0.4350  -0.0871
0.4500  -0.0215
0.4650   0.0605
0.4800   0.1532
0.4950   0.2496
0.5100   0.3427
0.5250   0.4276
0.5400   0.5022
0.5550   0.5677
0.5700   0.6274
0.5850   0.6858
0.6000   0.7470
0.6150   2.3093
0.6300   3.6789
0.6450   4.8245
0.6600   5.7192
0.6750   6.3414
0.6900   6.6756
0.7050   6.7137
0.7200   6.4549
0.7350   5.9061
0.7500   5.0814
0.7650   4.0021
0.7800   2.6956
0.7950   1.1945
0.8100   0.6623
0.8250   0.6745
0.8400   0.6920
0.8550   0.7143
0.8700   0.7408
0.8850   0.7707
0.9000   0.8027
0.9150   0.8351
0.9300   0.8655
0.9450   0.8916
0.9600   0.9104
0.9750   0.9195
0.9900   0.9166
/}\relax
\relax
\Green{\relax   
\plot
0.0000   0.3797
0.0150   0.3577
0.0300   0.3274
0.0450   0.2905
0.0600   0.2484
0.0750   0.2033
0.0900   0.1570
0.1050   0.1112
0.1200   0.0676
0.1350   0.0272
0.1500  -0.0089
0.1650  -0.0400
0.1800  -0.0658
0.1950  -0.0856
0.2100  -0.0994
0.2250  -0.1066
0.2400  -0.1071
0.2550  -0.1005
0.2700  -0.0862
0.2850  -0.0639
0.3000  -0.0328
0.3150   0.0079
0.3300   0.0589
0.3450   0.1213
0.3600   0.1960
0.3750   0.2836
0.3900   0.3844
0.4050   0.4786
0.4200   0.5537
0.4350   0.6461
0.4500   0.7471
0.4650   0.8455
0.4800   0.9292
0.4950   0.9876
0.5100   1.0146
0.5250   1.0104
0.5400   0.9816
0.5550   0.9382
0.5700   0.8918
0.5850   0.8525
0.6000   0.8280
0.6150   2.3188
0.6300   3.6302
0.6450   4.7299
0.6600   5.5892
0.6750   6.1846
0.6900   6.4995
0.7050   6.5241
0.7200   6.2565
0.7350   5.7025
0.7500   4.8753
0.7650   3.7952
0.7800   2.4885
0.7950   0.9868
0.8100   0.4526
0.8250   0.4609
0.8400   0.4728
0.8550   0.4881
0.8700   0.5063
0.8850   0.5268
0.9000   0.5487
0.9150   0.5706
0.9300   0.5909
0.9450   0.6075
0.9600   0.6184
0.9750   0.6215
0.9900   0.6153
/}\relax
\Blue{\relax  
\plot
0.0000   0.3912
0.0150   0.3691
0.0300   0.3386
0.0450   0.3012
0.0600   0.2586
0.0750   0.2127
0.0900   0.1655
0.1050   0.1187
0.1200   0.0739
0.1350   0.0322
0.1500  -0.0053
0.1650  -0.0380
0.1800  -0.0654
0.1950  -0.0871
0.2100  -0.1029
0.2250  -0.1124
0.2400  -0.1154
0.2550  -0.1116
0.2700  -0.1007
0.2850  -0.0821
0.3000  -0.0553
0.3150  -0.0198
0.3300   0.0253
0.3450   0.0808
0.3600   0.1473
0.3750   0.2254
0.3900   0.3152
0.4050   0.3970
0.4200   0.4581
0.4350   0.5353
0.4500   0.6201
0.4650   0.7023
0.4800   0.7707
0.4950   0.8161
0.5100   0.8333
0.5250   0.8231
0.5400   0.7919
0.5550   0.7493
0.5700   0.7059
0.5850   0.6711
0.6000   0.6520
0.6150   2.1485
0.6300   3.4656
0.6450   4.5705
0.6600   5.4344
0.6750   6.0337
0.6900   6.3516
0.7050   6.3785
0.7200   6.1125
0.7350   5.5592
0.7500   4.7319
0.7650   3.6509
0.7800   2.3426
0.7950   0.8384
0.8100   0.3008
0.8250   0.3049
0.8400   0.3118
0.8550   0.3214
0.8700   0.3332
0.8850   0.3469
0.9000   0.3617
0.9150   0.3764
0.9300   0.3897
0.9450   0.4001
0.9600   0.4057
0.9750   0.4050
0.9900   0.3966
/}\relax
\endpicture
}
\setbox\figuresix=\vbox{\hsize=\xfiglen 
\beginpicture
\scriptsize
  \setcoordinatesystem units <\xfigdim,0.17\yfigdim>  point at 0 0
  \setplotarea x from 0 to 1, y from -1 to 7
  \axis bottom shiftedto y=0 ticks short numbered from 0 to 1 by 0.2 /
  \axis left ticks short unlabeled from -1 to 7.0 by 1 /
  \axis left ticks short numbered from 0 to 6 by 2 /
  \put {$q(x)$} [lt] at 0.025 7
\setlinear
\setsolid
\Black{\relax  
\plot
0.0000   0.2716
0.0150   0.2541
0.0300   0.2378
0.0450   0.2229
0.0600   0.2093
0.0750   0.1972
0.0900   0.1864
0.1050   0.1771
0.1200   0.1691
0.1350   0.1626
0.1500   0.1574
0.1650   0.1536
0.1800   0.1512
0.1950   0.1501
0.2100   0.1503
0.2250   0.1519
0.2400   0.1547
0.2550   0.1590
0.2700   0.1646
0.2850   0.1716
0.3000   0.1800
0.3150   0.1898
0.3300   0.2010
0.3450   0.2137
0.3600   0.2277
0.3750   0.2431
0.3900   0.2598
0.4050   0.2776
0.4200   0.2967
0.4350   0.3167
0.4500   0.3375
0.4650   0.3590
0.4800   0.3810
0.4950   0.4032
0.5100   0.4254
0.5250   0.4474
0.5400   0.4688
0.5550   0.4895
0.5700   0.5091
0.5850   0.5274
0.6000   0.5441
0.6150   2.1100
0.6300   3.4879
0.6450   4.6446
0.6600   5.5518
0.6750   6.1867
0.7050   6.5834
0.7200   6.3349
0.7350   5.7946
0.7500   4.9760
0.7650   3.8996
0.7800   2.5919
0.7950   1.0848
0.8100   0.5331
0.8250   0.5153
0.8400   0.4962
0.8550   0.4758
0.8700   0.4546
0.8850   0.4328
0.9000   0.4106
0.9150   0.3884
0.9300   0.3663
0.9450   0.3446
0.9600   0.3235
0.9750   0.3032
0.9900   0.2839
/}
\setsolid
%
\Red{\relax   
\!\!\!\!
\plot
0.0000  -0.0597
0.0150  -0.0586
0.0300  -0.0553
0.0450  -0.0502
0.0600  -0.0438
0.0750  -0.0364
0.0900  -0.0288
0.1050  -0.0212
0.1200  -0.0140
0.1350  -0.0074
0.1500  -0.0017
0.1650   0.0030
0.1800   0.0069
0.1950   0.0097
0.2100   0.0115
0.2250   0.0124
0.2400   0.0121
0.2550   0.0108
0.2700   0.0084
0.2850   0.0046
0.3000  -0.0004
0.3150  -0.0070
0.3300  -0.0153
0.3450  -0.0255
0.3600  -0.0377
0.3750  -0.0521
0.3900  -0.0685
0.4050   0.0097
0.4200   0.2504
0.4350   0.4482
0.4500   0.6064
0.4650   0.7282
0.4800   0.8165
0.4950   0.8727
0.5100   0.8957
0.5250   0.8825
0.5400   0.8285
0.5550   0.7294
0.5700   0.5817
0.5850   0.3834
0.6000   0.1339
0.6150   0.1528
0.6300   0.1662
0.6450   0.1752
0.6600   0.1810
0.6750   0.1847
0.6900   0.1869
0.7050   0.1884
0.7200   0.1897
0.7350   0.1910
0.7500   0.1928
0.7650   0.1952
0.7800   0.1987
0.7950   0.2035
0.8100   0.2103
0.8250   0.2191
0.8400   0.2301
0.8550   0.2432
0.8700   0.2586
0.8850   0.2761
0.9000   0.2953
0.9150   0.3157
0.9300   0.3363
0.9450   0.3557
0.9600   0.3726
0.9750   0.3854
0.9900   0.3927
/}\relax
\Green{\relax   
\plot
0.0000  -0.5738
0.0150  -0.5891
0.0300  -0.5999
0.0450  -0.6062
0.0600  -0.6085
0.0750  -0.6075
0.0900  -0.6038
0.1050  -0.5980
0.1200  -0.5907
0.1350  -0.5822
0.1500  -0.5728
0.1650  -0.5627
0.1800  -0.5519
0.1950  -0.5404
0.2100  -0.5280
0.2250  -0.5145
0.2400  -0.4998
0.2550  -0.4834
0.2700  -0.4648
0.2850  -0.4436
0.3000  -0.4192
0.3150  -0.3906
0.3300  -0.3568
0.3450  -0.3167
0.3600  -0.2688
0.3750  -0.2115
0.3900  -0.1428
0.4050  -0.0709
0.4200  -0.0001
0.4350   0.0905
0.4500   0.1976
0.4650   0.3162
0.4800   0.4389
0.4950   0.5572
0.5100   0.6631
0.5250   0.7513
0.5400   0.8203
0.5550   0.8723
0.5700   0.9118
0.5850   0.9444
0.6000   0.9749
0.6150   2.5291
0.6300   3.8938
0.6450   5.0379
0.6600   5.9340
0.6750   6.5604
0.6900   6.9013
0.7050   6.9483
0.7200   6.7002
0.7350   6.1635
0.7500   5.3520
0.7650   4.2863
0.7800   2.9933
0.7950   1.5048
0.8100   0.9795
0.8250   0.9944
0.8400   1.0124
0.8550   1.0333
0.8700   1.0570
0.8850   1.0827
0.9000   1.1096
0.9150   1.1362
0.9300   1.1607
0.9450   1.1811
0.9600   1.1951
0.9750   1.2005
0.9900   1.1956
/}\relax
\Blue{\relax  
\plot
0.0000   0.2425
0.0150   0.2234
0.0300   0.2013
0.0450   0.1772
0.0600   0.1518
0.0750   0.1259
0.0900   0.1003
0.1050   0.0756
0.1200   0.0526
0.1350   0.0317
0.1500   0.0132
0.1650  -0.0025
0.1800  -0.0154
0.1950  -0.0252
0.2100  -0.0319
0.2250  -0.0352
0.2400  -0.0352
0.2550  -0.0316
0.2700  -0.0243
0.2850  -0.0130
0.3000   0.0025
0.3150   0.0226
0.3300   0.0477
0.3450   0.0782
0.3600   0.1144
0.3750   0.1568
0.3900   0.2055
0.4050   0.2504
0.4200   0.2854
0.4350   0.3296
0.4500   0.3787
0.4650   0.4277
0.4800   0.4711
0.4950   0.5040
0.5100   0.5240
0.5250   0.5313
0.5400   0.5288
0.5550   0.5211
0.5700   0.5133
0.5850   0.5099
0.6000   0.5140
0.6150   2.0497
0.6300   3.4016
0.6450   4.5365
0.6600   5.4258
0.6750   6.0463
0.7050   6.4226
0.7200   6.1674
0.7350   5.6222
0.7500   4.8004
0.7650   3.7225
0.7800   2.4148
0.7950   0.9090
0.8100   0.3635
0.8250   0.3553
0.8400   0.3473
0.8550   0.3397
0.8700   0.3324
0.8850   0.3255
0.9000   0.3188
0.9150   0.3119
0.9300   0.3046
0.9450   0.2963
0.9600   0.2863
0.9750   0.2742
0.9900   0.2594
/}\relax
\endpicture
}
%
\setbox\figureuthree =\vbox{
\hbox to \hsize{\copy\figureone\hss\copy\figuretwo\hss\copy\figurethree}
\vskip20pt
\hbox to \hsize{\copy\figurefour\phantom{2}\hss\copy\figurefive\hss\copy\figuresix}
}
\setbox\figureone=\vbox{\hsize=\xfiglen
\beginpicture
\scriptsize
  \setcoordinatesystem units <\xfigdim,\yfigdim>  point at 0 0
  \setplotarea x from 0 to 1, y from -0.2 to 1.4
  \axis bottom shiftedto y=0 ticks short numbered from 0 to 1 by 0.2 /
  \axis left ticks short numbered from 0.2 to 1.4 by 0.4 /
  \put {$p(x)$} [lt] at 0.025 1.4
\setquadratic
\setlinear
\setsolid
\Black{\relax  
\plot
0.0000   0.1000
0.0150   0.1000
0.0300   0.1000
0.0450   0.1000
0.0600   0.1000
0.0750   0.1000
0.0900   0.1000
0.1050   0.1000
0.1200   0.1000
0.1350   0.1000
0.1500   0.1000
0.1650   0.1000
0.1800   0.1000
0.1950   0.1000
0.2100   0.1000
0.2250   0.1000
0.2400   0.1000
0.2550   0.1000
0.2700   0.1000
0.2850   0.1000
0.3000   0.1000
0.3150   0.1000
0.3300   0.1000
0.3450   0.1000
0.3600   0.1000
0.3750   0.1000
0.3900   0.1000
0.4050   0.1975
0.4200   0.4600
0.4350   0.6775
0.4500   0.8500
0.4650   0.9775
0.4800   1.0600
0.4950   1.0975
0.5100   1.0900
0.5250   1.0375
0.5400   0.9400
0.5550   0.7975
0.5700   0.6100
0.5850   0.3775
0.6000   0.1000
0.6150   0.1000
0.6300   0.1000
0.6450   0.1000
0.6600   0.1000
0.6750   0.1000
0.6900   0.1000
0.7050   0.1000
0.7200   0.1000
0.7350   0.1000
0.7500   0.1000
0.7650   0.1000
0.7800   0.1000
0.7950   0.1000
0.8100   0.1000
0.8250   0.1000
0.8400   0.1000
0.8550   0.1000
0.8700   0.1000
0.8850   0.1000
0.9000   0.1000
0.9150   0.1000
0.9300   0.1000
0.9450   0.1000
0.9600   0.1000
0.9750   0.1000
0.9900   0.1000
/}\relax
%
\Red {\relax   
\!\!
\plot
0.0000   0.1220
0.0150   0.1218
0.0300   0.1215
0.0450   0.1209
0.0600   0.1201
0.0750   0.1192
0.0900   0.1182
0.1050   0.1172
0.1200   0.1161
0.1350   0.1151
0.1500   0.1142
0.1650   0.1133
0.1800   0.1125
0.1950   0.1118
0.2100   0.1111
0.2250   0.1106
0.2400   0.1101
0.2550   0.1097
0.2700   0.1093
0.2850   0.1090
0.3000   0.1088
0.3150   0.1086
0.3300   0.1084
0.3450   0.1083
0.3600   0.1083
0.3750   0.1082
0.3900   0.1082
0.4050   0.2058
0.4200   0.4684
0.4350   0.6861
0.4500   0.8587
0.4650   0.9864
0.4800   1.0691
0.4950   1.1068
0.5100   1.0994
0.5250   1.0470
0.5400   0.9494
0.5550   0.8068
0.5700   0.6191
0.5850   0.3863
0.6000   0.1086
0.6150   0.1083
0.6300   0.1081
0.6450   0.1079
0.6600   0.1078
0.6750   0.1077
0.6900   0.1076
0.7050   0.1075
0.7200   0.1076
0.7350   0.1076
0.7500   0.1078
0.7650   0.1080
0.7800   0.1082
0.7950   0.1085
0.8100   0.1089
0.8250   0.1094
0.8400   0.1099
0.8550   0.1106
0.8700   0.1112
0.8850   0.1119
0.9000   0.1127
0.9150   0.1134
0.9300   0.1141
0.9450   0.1148
0.9600   0.1153
0.9750   0.1157
0.9900   0.1160
/}\relax
\Green{\relax   
\plot
0.0000   0.1073
0.0150   0.1072
0.0300   0.1071
0.0450   0.1068
0.0600   0.1065
0.0750   0.1062
0.0900   0.1058
0.1050   0.1054
0.1200   0.1050
0.1350   0.1046
0.1500   0.1043
0.1650   0.1040
0.1800   0.1037
0.1950   0.1035
0.2100   0.1033
0.2250   0.1031
0.2400   0.1030
0.2550   0.1029
0.2700   0.1028
0.2850   0.1028
0.3000   0.1028
0.3150   0.1028
0.3300   0.1029
0.3450   0.1030
0.3600   0.1032
0.3750   0.1034
0.3900   0.1036
0.4050   0.2015
0.4200   0.4643
0.4350   0.6822
0.4500   0.8551
0.4650   0.9829
0.4800   1.0657
0.4950   1.1033
0.5100   1.0959
0.5250   1.0433
0.5400   0.9455
0.5550   0.8028
0.5700   0.6149
0.5850   0.3821
0.6000   0.1043
0.6150   0.1041
0.6300   0.1038
0.6450   0.1037
0.6600   0.1035
0.6750   0.1034
0.6900   0.1034
0.7050   0.1034
0.7200   0.1034
0.7350   0.1035
0.7500   0.1035
0.7650   0.1037
0.7800   0.1038
0.7950   0.1040
0.8100   0.1042
0.8250   0.1045
0.8400   0.1048
0.8550   0.1052
0.8700   0.1056
0.8850   0.1060
0.9000   0.1064
0.9150   0.1069
0.9300   0.1073
0.9450   0.1077
0.9600   0.1080
0.9750   0.1082
0.9900   0.1083
/}\relax
\Blue{\relax  
\plot
0.0000   0.1006
0.0150   0.1005
0.0300   0.1005
0.0450   0.1004
0.0600   0.1004
0.0750   0.1003
0.0900   0.1002
0.1050   0.1001
0.1200   0.1000
0.1350   0.0999
0.1500   0.0999
0.1650   0.0998
0.1800   0.0998
0.1950   0.0997
0.2100   0.0997
0.2250   0.0997
0.2400   0.0997
0.2550   0.0997
0.2700   0.0997
0.2850   0.0997
0.3000   0.0997
0.3150   0.0998
0.3300   0.0998
0.3450   0.0999
0.3600   0.1000
0.3750   0.1001
0.3900   0.1002
0.4050   0.1979
0.4200   0.4606
0.4350   0.6783
0.4500   0.8510
0.4650   0.9787
0.4800   1.0613
0.4950   1.0989
0.5100   1.0914
0.5250   1.0388
0.5400   0.9411
0.5550   0.7984
0.5700   0.6107
0.5850   0.3780
0.6000   0.1003
0.6150   0.1002
0.6300   0.1001
0.6450   0.0999
0.6600   0.0999
0.6750   0.0998
0.6900   0.0998
0.7050   0.0997
0.7200   0.0997
0.7350   0.0997
0.7500   0.0997
0.7650   0.0997
0.7800   0.0997
0.7950   0.0997
0.8100   0.0998
0.8250   0.0998
0.8400   0.0999
0.8550   0.0999
0.8700   0.1000
0.8850   0.1001
0.9000   0.1001
0.9150   0.1002
0.9300   0.1003
0.9450   0.1004
0.9600   0.1005
0.9750   0.1005
0.9900   0.1006
/}\relax
\endpicture
}   
\setbox\figuretwo=\vbox{\hsize=\xfiglen 
\beginpicture
\scriptsize
  \setcoordinatesystem units <\xfigdim,\yfigdim>  point at 0 0
  \setplotarea x from 0 to 1, y from -0.2 to 1.4
  \axis bottom shiftedto y=0.0 ticks short numbered from 0 to 1 by 0.2 /
  \axis left ticks short numbered from 0.2 to 1.4 by 0.4 /
  \put {$p(x)$} [lt] at 0.025 1.4
\setlinear
\setsolid
\Black{\relax  
\plot
0.0000   0.1000
0.0150   0.1000
0.0300   0.1000
0.0450   0.1000
0.0600   0.1000
0.0750   0.1000
0.0900   0.1000
0.1050   0.1000
0.1200   0.1000
0.1350   0.1000
0.1500   0.1000
0.1650   0.1000
0.1800   0.1000
0.1950   0.1000
0.2100   0.1000
0.2250   0.1000
0.2400   0.1000
0.2550   0.1000
0.2700   0.1000
0.2850   0.1000
0.3000   0.1000
0.3150   0.1000
0.3300   0.1000
0.3450   0.1000
0.3600   0.1000
0.3750   0.1000
0.3900   0.1000
0.4050   0.1975
0.4200   0.4600
0.4350   0.6775
0.4500   0.8500
0.4650   0.9775
0.4800   1.0600
0.4950   1.0975
0.5100   1.0900
0.5250   1.0375
0.5400   0.9400
0.5550   0.7975
0.5700   0.6100
0.5850   0.3775
0.6000   0.1000
0.6150   0.1000
0.6300   0.1000
0.6450   0.1000
0.6600   0.1000
0.6750   0.1000
0.6900   0.1000
0.7050   0.1000
0.7200   0.1000
0.7350   0.1000
0.7500   0.1000
0.7650   0.1000
0.7800   0.1000
0.7950   0.1000
0.8100   0.1000
0.8250   0.1000
0.8400   0.1000
0.8550   0.1000
0.8700   0.1000
0.8850   0.1000
0.9000   0.1000
0.9150   0.1000
0.9300   0.1000
0.9450   0.1000
0.9600   0.1000
0.9750   0.1000
0.9900   0.1000
/}\relax
\setsolid
%
\Red{\relax   
\plot
0.0000  -0.1064
0.0150  -0.1056
0.0300  -0.1033
0.0450  -0.0996
0.0600  -0.0949
0.0750  -0.0894
0.0900  -0.0836
0.1050  -0.0778
0.1200  -0.0722
0.1350  -0.0671
0.1500  -0.0626
0.1650  -0.0589
0.1800  -0.0561
0.1950  -0.0543
0.2100  -0.0534
0.2250  -0.0536
0.2400  -0.0550
0.2550  -0.0576
0.2700  -0.0614
0.2850  -0.0667
0.3000  -0.0734
0.3150  -0.0818
0.3300  -0.0919
0.3450  -0.1038
0.3600  -0.1176
0.3750  -0.1332
0.3900  -0.1504
0.4050  -0.0724
0.4200   0.1690
0.4350   0.3704
0.4500   0.5365
0.4650   0.6716
0.4800   0.7790
0.4950   0.8597
0.5100   0.9114
0.5250   0.9290
0.5400   0.9058
0.5550   0.8356
0.5700   0.7136
0.5850   0.5371
0.6000   0.3053
0.6150   0.3375
0.6300   0.3605
0.6450   0.3762
0.6600   0.3864
0.6750   0.3929
0.6900   0.3968
0.7050   0.3994
0.7200   0.4015
0.7350   0.4038
0.7500   0.4070
0.7650   0.4116
0.7800   0.4183
0.7950   0.4276
0.8100   0.4410
0.8250   0.4585
0.8400   0.4798
0.8550   0.5048
0.8700   0.5333
0.8850   0.5648
0.9000   0.5984
0.9150   0.6331
0.9300   0.6671
0.9450   0.6986
0.9600   0.7254
0.9750   0.7452
0.9900   0.7565
/}\relax
\Green{\relax   
\plot
0.0000   0.0165
0.0150   0.0168
0.0300   0.0176
0.0450   0.0189
0.0600   0.0206
0.0750   0.0227
0.0900   0.0250
0.1050   0.0275
0.1200   0.0301
0.1350   0.0327
0.1500   0.0354
0.1650   0.0380
0.1800   0.0405
0.1950   0.0431
0.2100   0.0456
0.2250   0.0483
0.2400   0.0510
0.2550   0.0540
0.2700   0.0573
0.2850   0.0611
0.3000   0.0657
0.3150   0.0713
0.3300   0.0781
0.3450   0.0866
0.3600   0.0973
0.3750   0.1105
0.3900   0.1269
0.4050   0.2429
0.4200   0.5253
0.4350   0.7666
0.4500   0.9661
0.4650   1.1219
0.4800   1.2317
0.4950   1.2928
0.5100   1.3029
0.5250   1.2607
0.5400   1.1663
0.5550   1.0209
0.5700   0.8265
0.5850   0.5850
0.6000   0.2978
0.6150   0.2854
0.6300   0.2736
0.6450   0.2629
0.6600   0.2536
0.6750   0.2458
0.6900   0.2395
0.7050   0.2348
0.7200   0.2315
0.7350   0.2298
0.7500   0.2294
0.7650   0.2305
0.7800   0.2331
0.7950   0.2373
0.8100   0.2439
0.8250   0.2526
0.8400   0.2633
0.8550   0.2758
0.8700   0.2901
0.8850   0.3058
0.9000   0.3227
0.9150   0.3401
0.9300   0.3573
0.9450   0.3731
0.9600   0.3866
0.9750   0.3966
0.9900   0.4023
/}\relax
\Blue{\relax  
\plot
  0.0000   0.1344
  0.0150   0.1323
  0.0300   0.1265
  0.0450   0.1175
  0.0600   0.1063
  0.0750   0.0938
  0.0900   0.0809
  0.1050   0.0684
  0.1200   0.0569
  0.1350   0.0466
  0.1500   0.0377
  0.1650   0.0303
  0.1800   0.0243
  0.1950   0.0196
  0.2100   0.0162
  0.2250   0.0140
  0.2400   0.0130
  0.2550   0.0132
  0.2700   0.0147
  0.2850   0.0177
  0.3000   0.0225
  0.3150   0.0293
  0.3300   0.0387
  0.3450   0.0510
  0.3600   0.0670
  0.3750   0.0872
  0.3900   0.1120
  0.4050   0.2375
  0.4200   0.5299
  0.4350   0.7805
  0.4500   0.9868
  0.4650   1.1450
  0.4800   1.2506
  0.4950   1.2998
  0.5100   1.2904
  0.5250   1.2229
  0.5400   1.1000
  0.5550   0.9259
  0.5700   0.7048
  0.5850   0.4404
  0.6000   0.1347
  0.6150   0.1080
  0.6300   0.0858
  0.6450   0.0678
  0.6600   0.0536
  0.6750   0.0426
  0.6900   0.0343
  0.7050   0.0282
  0.7200   0.0240
  0.7350   0.0212
  0.7500   0.0197
  0.7650   0.0194
  0.7800   0.0202
  0.7950   0.0221
  0.8100   0.0258
  0.8250   0.0312
  0.8400   0.0380
  0.8550   0.0463
  0.8700   0.0561
  0.8850   0.0673
  0.9000   0.0796
  0.9150   0.0928
  0.9300   0.1062
  0.9450   0.1188
  0.9600   0.1297
  0.9750   0.1380
  0.9900   0.1427
/}\relax
\endpicture
}   
\setbox\figurethree=\vbox{\hsize=\xfiglen 
\beginpicture
\scriptsize
  \setcoordinatesystem units <\xfigdim,\yfigdim>  point at 0 0
  \setplotarea x from 0 to 1, y from -0.2 to 1.4
  \axis bottom shiftedto y=0.0 ticks short numbered from 0 to 1 by 0.2 /
  \axis left ticks short numbered from 0.2 to 1.4 by 0.4 /
  \put {$p(x)$} [lt] at 0.025 1.4
\setlinear
\setsolid 
\Black{\relax  
\plot
0.0000   0.1000
0.0150   0.1000
0.0300   0.1000
0.0450   0.1000
0.0600   0.1000
0.0750   0.1000
0.0900   0.1000
0.1050   0.1000
0.1200   0.1000
0.1350   0.1000
0.1500   0.1000
0.1650   0.1000
0.1800   0.1000
0.1950   0.1000
0.2100   0.1000
0.2250   0.1000
0.2400   0.1000
0.2550   0.1000
0.2700   0.1000
0.2850   0.1000
0.3000   0.1000
0.3150   0.1000
0.3300   0.1000
0.3450   0.1000
0.3600   0.1000
0.3750   0.1000
0.3900   0.1000
0.4050   0.1975
0.4200   0.4600
0.4350   0.6775
0.4500   0.8500
0.4650   0.9775
0.4800   1.0600
0.4950   1.0975
0.5100   1.0900
0.5250   1.0375
0.5400   0.9400
0.5550   0.7975
0.5700   0.6100
0.5850   0.3775
0.6000   0.1000
0.6150   0.1000
0.6300   0.1000
0.6450   0.1000
0.6600   0.1000
0.6750   0.1000
0.6900   0.1000
0.7050   0.1000
0.7200   0.1000
0.7350   0.1000
0.7500   0.1000
0.7650   0.1000
0.7800   0.1000
0.7950   0.1000
0.8100   0.1000
0.8250   0.1000
0.8400   0.1000
0.8550   0.1000
0.8700   0.1000
0.8850   0.1000
0.9000   0.1000
0.9150   0.1000
0.9300   0.1000
0.9450   0.1000
0.9600   0.1000
0.9750   0.1000
0.9900   0.1000
/}
\setsolid
\Red{\relax   
\!\!\!\!
\plot
0.0000  -0.1173
0.0150  -0.1166
0.0300  -0.1143
0.0450  -0.1106
0.0600  -0.1057
0.0750  -0.1000
0.0900  -0.0936
0.1050  -0.0870
0.1200  -0.0804
0.1350  -0.0741
0.1500  -0.0684
0.1650  -0.0634
0.1800  -0.0593
0.1950  -0.0561
0.2100  -0.0541
0.2250  -0.0532
0.2400  -0.0534
0.2550  -0.0549
0.2700  -0.0577
0.2850  -0.0619
0.3000  -0.0674
0.3150  -0.0744
0.3300  -0.0828
0.3450  -0.0927
0.3600  -0.1038
0.3750  -0.1160
0.3900  -0.1288
0.4050  -0.0452
0.4200   0.2028
0.4350   0.4100
0.4500   0.5792
0.4650   0.7128
0.4800   0.8123
0.4950   0.8776
0.5100   0.9064
0.5250   0.8950
0.5400   0.8393
0.5550   0.7357
0.5700   0.5821
0.5850   0.3775
0.6000   0.1220
0.6150   0.1358
0.6300   0.1450
0.6450   0.1509
0.6600   0.1546
0.6750   0.1569
0.6900   0.1585
0.7050   0.1598
0.7200   0.1612
0.7350   0.1631
0.7500   0.1658
0.7650   0.1695
0.7800   0.1746
0.7950   0.1815
0.8100   0.1909
0.8250   0.2031
0.8400   0.2182
0.8550   0.2365
0.8700   0.2581
0.8850   0.2830
0.9000   0.3108
0.9150   0.3407
0.9300   0.3712
0.9450   0.4005
0.9600   0.4261
0.9750   0.4457
0.9900   0.4569
/
}\relax
\Green{\relax   
\plot
0.0000   0.0339
0.0150   0.0326
0.0300   0.0292
0.0450   0.0239
0.0600   0.0174
0.0750   0.0102
0.0900   0.0031
0.1050  -0.0036
0.1200  -0.0096
0.1350  -0.0146
0.1500  -0.0186
0.1650  -0.0217
0.1800  -0.0238
0.1950  -0.0250
0.2100  -0.0253
0.2250  -0.0249
0.2400  -0.0235
0.2550  -0.0212
0.2700  -0.0177
0.2850  -0.0129
0.3000  -0.0063
0.3150   0.0024
0.3300   0.0139
0.3450   0.0288
0.3600   0.0479
0.3750   0.0721
0.3900   0.1020
0.4050   0.2341
0.4200   0.5346
0.4350   0.7953
0.4500   1.0135
0.4650   1.1849
0.4800   1.3044
0.4950   1.3669
0.5100   1.3693
0.5250   1.3108
0.5400   1.1939
0.5550   1.0229
0.5700   0.8024
0.5850   0.5366
0.6000   0.2284
0.6150   0.1985
0.6300   0.1729
0.6450   0.1517
0.6600   0.1346
0.6750   0.1212
0.6900   0.1111
0.7050   0.1037
0.7200   0.0989
0.7350   0.0962
0.7500   0.0955
0.7650   0.0966
0.7800   0.0996
0.7950   0.1045
0.8100   0.1120
0.8250   0.1221
0.8400   0.1344
0.8550   0.1491
0.8700   0.1661
0.8850   0.1853
0.9000   0.2061
0.9150   0.2280
0.9300   0.2499
0.9450   0.2704
0.9600   0.2881
0.9750   0.3013
0.9900   0.3089
/
}\relax
\Blue{\relax  
\plot
0.0000   0.1220
0.0150   0.1214
0.0300   0.1197
0.0450   0.1171
0.0600   0.1139
0.0750   0.1103
0.0900   0.1065
0.1050   0.1028
0.1200   0.0994
0.1350   0.0963
0.1500   0.0936
0.1650   0.0914
0.1800   0.0895
0.1950   0.0881
0.2100   0.0870
0.2250   0.0864
0.2400   0.0861
0.2550   0.0862
0.2700   0.0868
0.2850   0.0878
0.3000   0.0893
0.3150   0.0915
0.3300   0.0944
0.3450   0.0982
0.3600   0.1031
0.3750   0.1091
0.3900   0.1163
0.4050   0.2210
0.4200   0.4895
0.4350   0.7144
0.4500   0.8948
0.4650   1.0297
0.4800   1.1181
0.4950   1.1587
0.5100   1.1511
0.5250   1.0952
0.5400   0.9919
0.5550   0.8423
0.5700   0.6474
0.5850   0.4083
0.6000   0.1255
0.6150   0.1187
0.6300   0.1128
0.6450   0.1079
0.6600   0.1039
0.6750   0.1008
0.7050   0.0962
0.7200   0.0947
0.7350   0.0935
0.7500   0.0927
0.7650   0.0921
0.7800   0.0918
0.7950   0.0918
0.8100   0.0928
0.8250   0.0945
0.8400   0.0967
0.8550   0.0994
0.8700   0.1025
0.8850   0.1060
0.9000   0.1099
0.9150   0.1140
0.9300   0.1180
0.9450   0.1219
0.9600   0.1252
0.9750   0.1276
0.9900   0.1290
/}\relax
\endpicture
}  
\setbox\figurefour=\vbox{\hsize=\xfiglen
\beginpicture
\scriptsize
  \setcoordinatesystem units <\xfigdim,0.12\yfigdim>  point at 0 0
  \setplotarea x from 0 to 1, y from 0 to 12
  \axis bottom shiftedto y=0 ticks short numbered from 0 to 1 by 0.2 /
  \axis left ticks short unlabeled from 0 to 12 by 2 /
  \put {0} [r] at -0.04 0
  \put {4} [r] at -0.04 4
  \put {8} [r] at -0.04 8
  \put {12} [r] at -0.04 12
  \put {$q(x)$} [lt] at 0.025 12.
\setlinear
\setsolid 
\Black{\relax  
\plot
0.0000   0.5382
0.0150   0.5206
0.0300   0.5036
0.0450   0.4876
0.0600   0.4725
0.0750   0.4586
0.0900   0.4459
0.1050   0.4346
0.1200   0.4247
0.1350   0.4164
0.1500   0.4098
0.1650   0.4048
0.1800   0.4016
0.1950   0.4001
0.2100   0.4004
0.2250   0.4025
0.2400   0.4063
0.2550   0.4118
0.2700   0.4190
0.2850   0.4279
0.3000   0.4382
0.3150   0.4500
0.3300   0.4631
0.3450   0.4774
0.3600   0.4928
0.3750   0.5092
0.3900   0.5264
0.4050   0.5442
0.4200   0.5625
0.4350   0.5812
0.4500   0.6000
0.4650   0.6188
0.4800   0.6375
0.4950   0.6558
0.5100   0.6736
0.5250   0.6908
0.5400   0.7072
0.5550   0.7226
0.5700   0.7369
0.5850   0.7500
0.6000   0.7618
0.6150   2.9149
0.6300   4.7639
0.6450   6.2857
0.6600   7.4609
0.6750   8.2745
0.6900   8.7157
0.7050   8.7789
0.7200   8.4633
0.7350   7.7729
0.7500   6.7168
0.7650   5.3086
0.7800   3.5662
0.7950   1.5117
0.8100   0.7541
0.8250   0.7414
0.8400   0.7275
0.8550   0.7124
0.8700   0.6964
0.8850   0.6794
0.9000   0.6618
0.9150   0.6436
0.9300   0.6251
0.9450   0.6063
0.9600   0.5874
0.9750   0.5687
0.9900   0.5503
/}
\Red{\relax
\!\!\!\!
\plot
0.0000   1.1685
0.0150   1.1490
0.0300   1.1271
0.0450   1.1031
0.0600   1.0774
0.0750   1.0506
0.0900   1.0234
0.1050   0.9963
0.1200   0.9699
0.1350   0.9447
0.1500   0.9212
0.1650   0.8997
0.1800   0.8806
0.1950   0.8639
0.2100   0.8499
0.2250   0.8387
0.2400   0.8302
0.2550   0.8246
0.2700   0.8217
0.2850   0.8215
0.3000   0.8239
0.3150   0.8286
0.3300   0.8355
0.3450   0.8443
0.3600   0.8549
0.3750   0.8670
0.3900   0.8802
0.4050   0.8946
0.4200   0.9101
0.4350   0.9266
0.4500   0.9437
0.4650   0.9613
0.4800   0.9789
0.4950   0.9959
0.5100   1.0120
0.5250   1.0265
0.5400   1.0392
0.5550   1.0499
0.5700   1.0586
0.5850   1.0657
0.6000   1.0712
0.6150   3.2182
0.6300   5.0614
0.6450   6.5777
0.6600   7.7480
0.6750   8.5572
0.7050   9.0559
0.7200   8.7389
0.7350   8.0483
0.7500   6.9930
0.7650   5.5868
0.7800   3.8477
0.7950   1.7976
0.8100   1.0456
0.8250   1.0397
0.8400   1.0335
0.8550   1.0268
0.8700   1.0197
0.8850   1.0120
0.9000   1.0035
0.9150   0.9943
0.9300   0.9839
0.9450   0.9724
0.9600   0.9594
0.9750   0.9449
0.9900   0.9288
/}
\Green{\relax
\!\!\!\!
\plot
0.0000   0.7102
0.0150   0.6918
0.0300   0.6732
0.0450   0.6544
0.0600   0.6358
0.0750   0.6176
0.0900   0.6001
0.1050   0.5836
0.1200   0.5684
0.1350   0.5548
0.1500   0.5428
0.1650   0.5326
0.1800   0.5245
0.1950   0.5184
0.2100   0.5144
0.2250   0.5125
0.2400   0.5128
0.2550   0.5153
0.2700   0.5199
0.2850   0.5266
0.3000   0.5354
0.3150   0.5462
0.3300   0.5590
0.3450   0.5737
0.3600   0.5903
0.3750   0.6086
0.3900   0.6286
0.4050   0.6500
0.4200   0.6727
0.4350   0.6963
0.4500   0.7204
0.4650   0.7442
0.4800   0.7669
0.4950   0.7877
0.5100   0.8062
0.5250   0.8223
0.5400   0.8360
0.5550   0.8480
0.5700   0.8586
0.5850   0.8683
0.6000   0.8772
0.6150   3.0281
0.6300   4.8757
0.6450   6.3968
0.6600   7.5720
0.6750   8.3861
0.6900   8.8285
0.7050   8.8933
0.7200   8.5798
0.7350   7.8920
0.7500   6.8388
0.7650   5.4338
0.7800   3.6951
0.7950   1.6447
0.8100   0.8915
0.8250   0.8837
0.8400   0.8749
0.8550   0.8652
0.8700   0.8547
0.8850   0.8435
0.9000   0.8314
0.9150   0.8185
0.9300   0.8049
0.9450   0.7904
0.9600   0.7751
0.9750   0.7588
0.9900   0.7418
/}
\Blue{\relax
\!\!\!\!
\plot
0.0000   0.5423
0.0150   0.5243
0.0300   0.5072
0.0450   0.4907
0.0600   0.4751
0.0750   0.4606
0.0900   0.4473
0.1050   0.4354
0.1200   0.4250
0.1350   0.4161
0.1500   0.4090
0.1650   0.4035
0.1800   0.3999
0.1950   0.3981
0.2100   0.3981
0.2250   0.3999
0.2400   0.4036
0.2550   0.4091
0.2700   0.4163
0.2850   0.4253
0.3000   0.4359
0.3150   0.4481
0.3300   0.4619
0.3450   0.4772
0.3600   0.4938
0.3750   0.5118
0.3900   0.5310
0.4050   0.5512
0.4200   0.5725
0.4350   0.5944
0.4500   0.6164
0.4650   0.6382
0.4800   0.6591
0.4950   0.6785
0.5100   0.6961
0.5250   0.7118
0.5400   0.7257
0.5550   0.7380
0.5700   0.7491
0.5850   0.7591
0.6000   0.7681
0.6150   2.9189
0.6300   4.7661
0.6450   6.2865
0.6600   7.4606
0.6750   8.2733
0.6900   8.7140
 0.7050 8.7768
0.7200   8.4610
0.7350   7.7705
0.7500   6.7144
0.7650   5.3063
0.7800   3.5641
0.7950   1.5099
0.8100   0.7526
0.8250   0.7403
0.8400   0.7268
0.8550   0.7122
0.8700   0.6967
0.8850   0.6803
0.9000   0.6633
0.9150   0.6458
0.9300   0.6278
0.9450   0.6096
0.9600   0.5913
0.9750   0.5729
0.9900   0.5546
/}
\endpicture
} 
\setbox\figurefive=\vbox{\hsize=\xfiglen
\beginpicture
\scriptsize
  \setcoordinatesystem units <\xfigdim,0.12\yfigdim>  point at 0 0
  \setplotarea x from 0 to 1, y from -2 to 12
  \axis bottom shiftedto y=0 ticks short numbered from 0 to 1 by 0.2 /
  \axis left ticks short unlabeled from -2 to 12 by 2 /
  \axis bottom shiftedto y=0 ticks short numbered from 0 to 1 by 0.2 /
  \put {0} [r] at -0.04 0
  \put {4} [r] at -0.04 4
  \put {8} [r] at -0.04 8
  \put {12} [rt] at -0.04 12
  \put {-2} [r] at -0.04 -2
  \put {$q(x)$} [lt] at 0.025 12
\setlinear
\setsolid 
\Black{\relax  
\plot
0.0000   0.5382
0.0150   0.5206
0.0300   0.5036
0.0450   0.4876
0.0600   0.4725
0.0750   0.4586
0.0900   0.4459
0.1050   0.4346
0.1200   0.4247
0.1350   0.4164
0.1500   0.4098
0.1650   0.4048
0.1800   0.4016
0.1950   0.4001
0.2100   0.4004
0.2250   0.4025
0.2400   0.4063
0.2550   0.4118
0.2700   0.4190
0.2850   0.4279
0.3000   0.4382
0.3150   0.4500
0.3300   0.4631
0.3450   0.4774
0.3600   0.4928
0.3750   0.5092
0.3900   0.5264
0.4050   0.5442
0.4200   0.5625
0.4350   0.5812
0.4500   0.6000
0.4650   0.6188
0.4800   0.6375
0.4950   0.6558
0.5100   0.6736
0.5250   0.6908
0.5400   0.7072
0.5550   0.7226
0.5700   0.7369
0.5850   0.7500
0.6000   0.7618
0.6150   2.9149
0.6300   4.7639
0.6450   6.2857
0.6600   7.4609
0.6750   8.2745
0.6900   8.7157
0.7050   8.7789
0.7200   8.4633
0.7350   7.7729
0.7500   6.7168
0.7650   5.3086
0.7800   3.5662
0.7950   1.5117
0.8100   0.7541
0.8250   0.7414
0.8400   0.7275
0.8550   0.7124
0.8700   0.6964
0.8850   0.6794
0.9000   0.6618
0.9150   0.6436
0.9300   0.6251
0.9450   0.6063
0.9600   0.5874
0.9750   0.5687
0.9900   0.5503
/}
\Red{\relax   
\!\!\!\!
\plot
0.0000  -1.7845
0.0150  -1.7991
0.0300  -1.8078
0.0450  -1.8110
0.0600  -1.8097
0.0750  -1.8053
0.0900  -1.7993
0.1050  -1.7929
0.1200  -1.7875
0.1350  -1.7841
0.1500  -1.7837
0.1650  -1.7869
0.1800  -1.7944
0.1950  -1.8068
0.2100  -1.8247
0.2250  -1.8487
0.2400  -1.8795
0.2550  -1.9177
0.2700  -1.9644
0.2850  -2.0203
0.3000  -2.0867
0.3150  -2.1646
0.3300  -2.2550
0.3450  -2.3586
0.3600  -2.4758
0.3750  -2.6059
0.3900  -2.7464
0.4050  -2.9094
0.4200  -3.0861
0.4350  -3.2000
0.4500  -3.2057
0.4650  -3.0600
0.4800  -2.7313
0.4950  -2.2114
0.5100  -1.5223
0.5250  -0.7133
0.5400   0.1518
0.5550   1.0112
0.5700   1.8170
0.5850   2.5388
0.6000   3.1618
0.6150   5.7828
0.6300   7.9980
0.6450   9.7998
0.6600  11.1850
0.6750  12.1532
0.6900  12.7066
0.7050  12.8507
0.7200  12.5942
0.7350  11.9499
0.7500  10.9344
0.7650   9.5689
0.7800   7.8791
0.7950   5.8949
0.8100   5.2436
0.8250   5.3720
0.8400   5.5275
0.8550   5.7077
0.8700   5.9093
0.8850   6.1272
0.9000   6.3546
0.9150   6.5829
0.9300   6.8010
0.9450   6.9967
0.9600   7.1570
0.9750   7.2697
0.9900   7.3252
/}
\Green{\relax
\!\!\!\!
\plot
0.0000  -0.4673
0.0150  -0.4838
0.0300  -0.4978
0.0450  -0.5091
0.0600  -0.5176
0.0750  -0.5232
0.0900  -0.5259
0.1050  -0.5256
0.1200  -0.5223
0.1350  -0.5157
0.1500  -0.5057
0.1650  -0.4920
0.1800  -0.4742
0.1950  -0.4519
0.2100  -0.4244
0.2250  -0.3910
0.2400  -0.3509
0.2550  -0.3030
0.2700  -0.2460
0.2850  -0.1782
0.3000  -0.0976
0.3150  -0.0018
0.3300   0.1120
0.3450   0.2472
0.3600   0.4077
0.3750   0.5973
0.3900   0.8197
0.4050   1.0606
0.4200   1.3100
0.4350   1.6001
0.4500   1.9205
0.4650   2.2535
0.4800   2.5753
0.4950   2.8596
0.5100   3.0844
0.5250   3.2375
0.5400   3.3198
0.5550   3.3429
0.5700   3.3243
0.5850   3.2824
0.6000   3.2325
0.6150   5.2862
0.6300   7.0371
0.6450   8.4680
0.6600   9.5627
0.6750  10.3078
0.6900  10.6935
0.7050  10.7141
0.7200  10.3690
0.7350   9.6621
0.7500   8.6024
0.7650   7.2038
0.7800   5.4846
0.7950   3.4672
0.8100   2.7712
0.8250   2.8408
0.8400   2.9243
0.8550   3.0203
0.8700   3.1270
0.8850   3.2417
0.9000   3.3609
0.9150   3.4797
0.9300   3.5922
0.9450   3.6920
0.9600   3.7719
0.9750   3.8255
0.9900   3.8478
/}
\Blue{\relax
\!\!\!\!
\plot
0.0000   0.7754
0.0150   0.7526
0.0300   0.7200
0.0450   0.6792
0.0600   0.6321
0.0750   0.5810
0.0900   0.5282
0.1050   0.4757
0.1200   0.4254
0.1350   0.3787
0.1500   0.3368
0.1650   0.3006
0.1800   0.2706
0.1950   0.2473
0.2100   0.2310
0.2250   0.2221
0.2400   0.2208
0.2550   0.2275
0.2700   0.2425
0.2850   0.2663
0.3000   0.2996
0.3150   0.3430
0.3300   0.3974
0.3450   0.4637
0.3600   0.5426
0.3750   0.6349
0.3900   0.7405
0.4050   0.8419
0.4200   0.9264
0.4350   1.0249
0.4500   1.1289
0.4650   1.2273
0.4800   1.3077
0.4950   1.3594
0.5100   1.3763
0.5250   1.3592
0.5400   1.3152
0.5550   1.2553
0.5700   1.1913
0.5850   1.1334
0.6000   1.0889
0.6150   3.1627
0.6300   4.9417
0.6450   6.4033
0.6600   7.5277
0.6750   8.2990
0.7050   8.7410
0.7200   8.4035
0.7350   7.6964
0.7500   6.6284
0.7650   5.2127
0.7800   3.4671
0.7950   1.4134
0.8100   0.6705
0.8250   0.6828
0.8400   0.6989
0.8550   0.7185
0.8700   0.7409
0.8850   0.7656
0.9000   0.7912
0.9150   0.8165
0.9300   0.8396
0.9450   0.8585
0.9600   0.8711
0.9750   0.8754
0.9900   0.8698
/}
\endpicture
}
%
\setbox\figuresix=\vbox{\hsize=\xfiglen
\beginpicture
\scriptsize
  \setcoordinatesystem units <\xfigdim,0.12\yfigdim>  point at 0 0
  \setplotarea x from 0 to 1, y from -2 to 12
  \axis bottom shiftedto y=0 ticks short numbered from 0 to 1 by 0.2 /
  \axis left ticks short unlabeled from -2 to 12 by 2 /
  \put {0} [r] at -0.04 0
  \put {4} [r] at -0.04 4
  \put {8} [r] at -0.04 8
  \put {12} [r] at -0.04 12
  \put {$q(x)$} [lt] at 0.025 12
\setlinear
\setsolid 
\Black{\relax  
\plot
0.0000   0.5382
0.0150   0.5206
0.0300   0.5036
0.0450   0.4876
0.0600   0.4725
0.0750   0.4586
0.0900   0.4459
0.1050   0.4346
0.1200   0.4247
0.1350   0.4164
0.1500   0.4098
0.1650   0.4048
0.1800   0.4016
0.1950   0.4001
0.2100   0.4004
0.2250   0.4025
0.2400   0.4063
0.2550   0.4118
0.2700   0.4190
0.2850   0.4279
0.3000   0.4382
0.3150   0.4500
0.3300   0.4631
0.3450   0.4774
0.3600   0.4928
0.3750   0.5092
0.3900   0.5264
0.4050   0.5442
0.4200   0.5625
0.4350   0.5812
0.4500   0.6000
0.4650   0.6188
0.4800   0.6375
0.4950   0.6558
0.5100   0.6736
0.5250   0.6908
0.5400   0.7072
0.5550   0.7226
0.5700   0.7369
0.5850   0.7500
0.6000   0.7618
0.6150   2.9149
0.6300   4.7639
0.6450   6.2857
0.6600   7.4609
0.6750   8.2745
0.6900   8.7157
0.7050   8.7789
0.7200   8.4633
0.7350   7.7729
0.7500   6.7168
0.7650   5.3086
0.7800   3.5662
0.7950   1.5117
0.8100   0.7541
0.8250   0.7414
0.8400   0.7275
0.8550   0.7124
0.8700   0.6964
0.8850   0.6794
0.9000   0.6618
0.9150   0.6436
0.9300   0.6251
0.9450   0.6063
0.9600   0.5874
0.9750   0.5687
0.9900   0.5503
/}
\setsolid
\Red{\relax   
\!\!\!\!
\plot
0.0000   0.3203
0.0150   0.2947
0.0300   0.2542
0.0450   0.2011
0.0600   0.1382
0.0750   0.0688
0.0900  -0.0038
0.1050  -0.0764
0.1200  -0.1464
0.1350  -0.2116
0.1500  -0.2703
0.1650  -0.3212
0.1800  -0.3635
0.1950  -0.3963
0.2100  -0.4191
0.2250  -0.4315
0.2400  -0.4327
0.2550  -0.4223
0.2700  -0.3994
0.2850  -0.3630
0.3000  -0.3120
0.3150  -0.2449
0.3300  -0.1601
0.3450  -0.0556
0.3600   0.0705
0.3750   0.2198
0.3900   0.3931
0.4050   0.5567
0.4200   0.6904
0.4350   0.8673
0.4500   1.0833
0.4650   1.3281
0.4800   1.5852
0.4950   1.8339
0.5100   2.0547
0.5250   2.2342
0.5400   2.3685
0.5550   2.4626
0.5700   2.5274
0.5850   2.5758
0.6000   2.6191
0.6150   4.7266
0.6300   6.5196
0.6450   7.9809
0.6600   9.0952
0.6750   9.8503
0.6900  10.2373
0.7050  10.2519
0.7200   9.8944
0.7350   9.1700
0.7500   8.0889
0.7650   6.6663
0.7800   4.9218
0.7950   2.8798
0.8100   2.1727
0.8250   2.2413
0.8400   2.3281
0.8550   2.4314
0.8700   2.5487
0.8850   2.6768
0.9000   2.8113
0.9150   2.9466
0.9300   3.0758
0.9450   3.1911
0.9600   3.2843
0.9750   3.3478
0.9900   3.3756
/}\relax
\Green{\relax   
\plot
0.0000   1.4254
0.0150   1.3957
0.0300   1.3433
0.0450   1.2713
0.0600   1.1840
0.0750   1.0863
0.0900   0.9833
0.1050   0.8796
0.1200   0.7794
0.1350   0.6858
0.1500   0.6016
0.1650   0.5286
0.1800   0.4682
0.1950   0.4215
0.2100   0.3893
0.2250   0.3725
0.2400   0.3716
0.2550   0.3877
0.2700   0.4217
0.2850   0.4749
0.3000   0.5489
0.3150   0.6453
0.3300   0.7664
0.3450   0.9142
0.3600   1.0909
0.3750   1.2977
0.3900   1.5349
0.4050   1.7673
0.4200   1.9685
0.4350   2.1984
0.4500   2.4383
0.4650   2.6633
0.4800   2.8456
0.4950   2.9605
0.5100   2.9938
0.5250   2.9462
0.5400   2.8329
0.5550   2.6784
0.5700   2.5094
0.5850   2.3488
0.6000   2.2136
0.6150   4.1753
0.6300   5.8528
0.6450   7.2249
0.6600   8.2723
0.6750   8.9783
0.6900   9.3308
0.7050   9.3221
0.7200   8.9498
0.7350   8.2165
0.7500   7.1303
0.7650   5.7040
0.7800   3.9552
0.7950   1.9055
0.8100   1.1847
0.8250   1.2329
0.8400   1.2932
0.8550   1.3644
0.8700   1.4448
0.8850   1.5321
0.9000   1.6235
0.9150   1.7148
0.9300   1.8014
0.9450   1.8779
0.9600   1.9385
0.9750   1.9780
0.9900   1.9922
/}\relax
\Blue{\relax  
\plot
0.0000   0.8007
0.0150   0.7731
0.0300   0.7271
0.0450   0.6652
0.0600   0.5907
0.0750   0.5077
0.0900   0.4202
0.1050   0.3321
0.1200   0.2466
0.1350   0.1664
0.1500   0.0935
0.1650   0.0293
0.1800  -0.0249
0.1950  -0.0684
0.2100  -0.1005
0.2250  -0.1209
0.2400  -0.1290
0.2550  -0.1242
0.2700  -0.1060
0.2850  -0.0735
0.3000  -0.0257
0.3150   0.0385
0.3300   0.1205
0.3450   0.2217
0.3600   0.3436
0.3750   0.4871
0.3900   0.6522
0.4050   0.8045
0.4200   0.9194
0.4350   1.0587
0.4500   1.2078
0.4650   1.3480
0.4800   1.4588
0.4950   1.5228
0.5100   1.5305
0.5250   1.4837
0.5400   1.3949
0.5550   1.2828
0.5700   1.1675
0.5850   1.0662
0.6000   0.9916
0.6150   3.0121
0.6300   4.7445
0.6450   6.1663
0.6600   7.2573
0.6750   8.0008
0.6900   8.3847
0.7050   8.4015
0.7200   8.0491
0.7350   7.3305
0.7500   6.2536
0.7650   4.8316
0.7800   3.0819
0.7950   1.0264
0.8100   0.2946
0.8250   0.3267
0.8400   0.3663
0.8550   0.4125
0.8700   0.4643
0.8850   0.5202
0.9000   0.5784
0.9150   0.6361
0.9300   0.6901
0.9450   0.7370
0.9600   0.7728
0.9750   0.7941
0.9900   0.7980
/}\relax
\endpicture
} 
%

\bigskip
\setbox\figurefourutwo =\vbox{
\hbox to \hsize{\copy\figureone\hss\copy\figuretwo\hss\copy\figurethree}
\vskip20pt
\hbox to \hsize{\copy\figurefour\hss\copy\figurefive\hss\copy\figuresix}
}
%
\setbox\figureone=\vbox{\hsize=\xfiglen
\beginpicture
\scriptsize
  \setcoordinatesystem units <\xfigdim,\yfigdim>  point at 0 0
  \setplotarea x from 0 to 1, y from -0.2 to 1.4
  \axis bottom shiftedto y=0 ticks short numbered from 0 to 1 by 0.2 /
  \axis left ticks short numbered from 0.2 to 1.4 by 0.4 /
  \put {$p(x)$} [lt] at 0.025 1.4
\setquadratic
\setlinear
\setsolid
\Black{\relax  
\plot
  0.0000   0.1000
  0.0150   0.1000
  0.0300   0.1000
  0.0450   0.1000
  0.0600   0.1000
  0.0750   0.1000
  0.0900   0.1000
  0.1050   0.1000
  0.1200   0.1000
  0.1350   0.1000
  0.1500   0.1000
  0.1650   0.1000
  0.1800   0.1000
  0.1950   0.1000
  0.2100   0.1000
  0.2250   0.1000
  0.2400   0.1000
  0.2550   0.1000
  0.2700   0.1000
  0.2850   0.1000
  0.3000   0.1000
  0.3150   0.1000
  0.3300   0.1000
  0.3450   0.1000
  0.3600   0.1000
  0.3750   0.1000
  0.3900   0.1000
  0.4050   0.1975
  0.4200   0.4600
  0.4350   0.6775
  0.4500   0.8500
  0.4650   0.9775
  0.4800   1.0600
  0.4950   1.0975
  0.5100   1.0900
  0.5250   1.0375
  0.5400   0.9400
  0.5550   0.7975
  0.5700   0.6100
  0.5850   0.3775
  0.6000   0.1000
  0.6150   0.1000
  0.6300   0.1000
  0.6450   0.1000
  0.6600   0.1000
  0.6750   0.1000
  0.6900   0.1000
  0.7050   0.1000
  0.7200   0.1000
  0.7350   0.1000
  0.7500   0.1000
  0.7650   0.1000
  0.7800   0.1000
  0.7950   0.1000
  0.8100   0.1000
  0.8250   0.1000
  0.8400   0.1000
  0.8550   0.1000
  0.8700   0.1000
  0.8850   0.1000
  0.9000   0.1000
  0.9150   0.1000
  0.9300   0.1000
  0.9450   0.1000
  0.9600   0.1000
  0.9750   0.1000
  0.9900   0.1000
/}\relax
%
\Red {\relax   
\plot
  0.0000   0.1250
  0.0150   0.1247
  0.0300   0.1240
  0.0450   0.1231
  0.0600   0.1218
  0.0750   0.1203
  0.0900   0.1188
  0.1050   0.1172
  0.1200   0.1156
  0.1350   0.1142
  0.1500   0.1128
  0.1650   0.1116
  0.1800   0.1105
  0.1950   0.1095
  0.2100   0.1086
  0.2250   0.1079
  0.2400   0.1072
  0.2550   0.1066
  0.2700   0.1061
  0.2850   0.1056
  0.3000   0.1051
  0.3150   0.1046
  0.3300   0.1042
  0.3450   0.1037
  0.3600   0.1031
  0.3750   0.1025
  0.3900   0.1019
  0.4050   0.1987
  0.4200   0.4605
  0.4350   0.6773
  0.4500   0.8491
  0.4650   0.9760
  0.4800   1.0581
  0.4950   1.0954
  0.5100   1.0880
  0.5250   1.0358
  0.5400   0.9388
  0.5550   0.7970
  0.5700   0.6102
  0.5850   0.3783
  0.6000   0.1014
  0.6150   0.1020
  0.6300   0.1026
  0.6450   0.1031
  0.6600   0.1035
  0.6750   0.1039
  0.6900   0.1043
  0.7050   0.1047
  0.7200   0.1051
  0.7350   0.1055
  0.7500   0.1059
  0.7650   0.1064
  0.7800   0.1070
  0.7950   0.1077
  0.8100   0.1084
  0.8250   0.1093
  0.8400   0.1102
  0.8550   0.1113
  0.8700   0.1125
  0.8850   0.1138
  0.9000   0.1151
  0.9150   0.1165
  0.9300   0.1178
  0.9450   0.1190
  0.9600   0.1200
  0.9750   0.1208
  0.9900   0.1212
/}\relax
\Green{\relax   
\plot
  0.0000   0.1244
  0.0150   0.1241
  0.0300   0.1235
  0.0450   0.1224
  0.0600   0.1211
  0.0750   0.1196
  0.0900   0.1181
  0.1050   0.1165
  0.1200   0.1150
  0.1350   0.1136
  0.1500   0.1124
  0.1650   0.1113
  0.1800   0.1104
  0.1950   0.1096
  0.2100   0.1089
  0.2250   0.1084
  0.2400   0.1080
  0.2550   0.1077
  0.2700   0.1075
  0.2850   0.1074
  0.3000   0.1074
  0.3150   0.1076
  0.3300   0.1078
  0.3450   0.1081
  0.3600   0.1086
  0.3750   0.1092
  0.3900   0.1099
  0.4050   0.2083
  0.4200   0.4717
  0.4350   0.6902
  0.4500   0.8637
  0.4650   0.9921
  0.4800   1.0752
  0.4950   1.1130
  0.5100   1.1055
  0.5250   1.0526
  0.5400   0.9544
  0.5550   0.8110
  0.5700   0.6226
  0.5850   0.3891
  0.6000   0.1107
  0.6150   0.1099
  0.6300   0.1092
  0.6450   0.1087
  0.6600   0.1083
  0.6750   0.1080
  0.6900   0.1078
  0.7050   0.1077
  0.7200   0.1077
  0.7350   0.1078
  0.7500   0.1081
  0.7650   0.1084
  0.7800   0.1089
  0.7950   0.1095
  0.8100   0.1102
  0.8250   0.1110
  0.8400   0.1120
  0.8550   0.1132
  0.8700   0.1145
  0.8850   0.1160
  0.9000   0.1176
  0.9150   0.1192
  0.9300   0.1208
  0.9450   0.1222
  0.9600   0.1235
  0.9750   0.1244
  0.9900   0.1250
/
}
\Blue{\relax   
\!\!\!\!
\plot
  0.0000   0.1039
  0.0150   0.1038
  0.0300   0.1036
  0.0450   0.1034
  0.0600   0.1031
  0.0750   0.1027
  0.0900   0.1024
  0.1050   0.1020
  0.1200   0.1017
  0.1350   0.1014
  0.1500   0.1012
  0.1650   0.1010
  0.1800   0.1008
  0.1950   0.1007
  0.2100   0.1006
  0.2250   0.1005
  0.2400   0.1005
  0.2550   0.1004
  0.2700   0.1004
  0.2850   0.1005
  0.3000   0.1005
  0.3150   0.1006
  0.3300   0.1007
  0.3450   0.1008
  0.3600   0.1010
  0.3750   0.1013
  0.3900   0.1016
  0.4050   0.1995
  0.4200   0.4624
  0.4350   0.6803
  0.4500   0.8533
  0.4650   0.9812
  0.4800   1.0640
  0.4950   1.1017
  0.5100   1.0941
  0.5250   1.0414
  0.5400   0.9436
  0.5550   0.8007
  0.5700   0.6127
  0.5850   0.3798
  0.6000   0.1019
  0.6150   0.1015
  0.6300   0.1012
  0.6450   0.1010
  0.6600   0.1008
  0.6750   0.1007
  0.6900   0.1006
  0.7050   0.1005
  0.7200   0.1005
  0.7350   0.1005
  0.7500   0.1005
  0.7650   0.1005
  0.7800   0.1006
  0.7950   0.1007
  0.8100   0.1008
  0.8250   0.1010
  0.8400   0.1011
  0.8550   0.1014
  0.8700   0.1016
  0.8850   0.1019
  0.9000   0.1023
  0.9150   0.1026
  0.9300   0.1030
  0.9450   0.1033
  0.9600   0.1036
  0.9750   0.1039
  0.9900   0.1040
/}\relax
\endpicture
}
\setbox\figurefour=\vbox{\hsize=\xfiglen
\beginpicture
\scriptsize
  \setcoordinatesystem units <\xfigdim,0.8\yfigdim>  point at 0 0
  \setplotarea x from 0 to 1, y from -0.2 to 1.6
  \axis bottom shiftedto y=0 ticks short numbered from 0 to 1 by 0.2 /
  \axis left ticks short numbered from 0 to 1.6 by 0.4 /
  \put {$q(x)$} [lt] at 0.025 1.6
\put{\phantom{-0.2}} [r] at 0 -0.2
\setquadratic
\setlinear
\setsolid
\Black{\relax  
\plot
  0.0000   0.0679
  0.0150   0.0635
  0.0300   0.0595
  0.0450   0.0557
  0.0600   0.0523
  0.0750   0.0493
  0.0900   0.0466
  0.1050   0.0443
  0.1200   0.0423
  0.1350   0.0406
  0.1500   0.0394
  0.1650   0.0384
  0.1800   0.0378
  0.1950   0.0375
  0.2100   0.0376
  0.2250   0.0380
  0.2400   0.0387
  0.2550   0.0397
  0.2700   0.0412
  0.2850   0.0429
  0.3000   0.0450
  0.3150   0.0475
  0.3300   0.0503
  0.3450   0.0534
  0.3600   0.0569
  0.3750   0.0608
  0.3900   0.0649
  0.4050   0.0694
  0.4200   0.0742
  0.4350   0.0792
  0.4500   0.0844
  0.4650   0.0898
  0.4800   0.0952
  0.4950   0.1008
  0.5100   0.1064
  0.5250   0.1118
  0.5400   0.1172
  0.5550   0.1224
  0.5700   0.1273
  0.5850   0.1318
  0.6000   0.1360
  0.6150   0.5275
  0.6300   0.8720
  0.6450   1.1612
  0.6600   1.3879
  0.6750   1.5467
  0.6900   1.6334
  0.7050   1.6458
  0.7200   1.5837
  0.7350   1.4486
  0.7500   1.2440
  0.7650   0.9749
  0.7800   0.6480
  0.7950   0.2712
  0.8100   0.1333
  0.8250   0.1288
  0.8400   0.1240
  0.8550   0.1190
  0.8700   0.1136
  0.8850   0.1082
  0.9000   0.1027
  0.9150   0.0971
  0.9300   0.0916
  0.9450   0.0862
  0.9600   0.0809
  0.9750   0.0758
  0.9900   0.0710
/\relax}\relax
\Red{\plot
 0.0000   0.4296
  0.0150   0.4243
  0.0300   0.4175
  0.0450   0.4094
  0.0600   0.4001
  0.0750   0.3899
  0.0900   0.3790
  0.1050   0.3677
  0.1200   0.3563
  0.1350   0.3448
  0.1500   0.3336
  0.1650   0.3228
  0.1800   0.3125
  0.1950   0.3027
  0.2100   0.2937
  0.2250   0.2855
  0.2400   0.2780
  0.2550   0.2713
  0.2700   0.2654
  0.2850   0.2603
  0.3000   0.2559
  0.3150   0.2522
  0.3300   0.2492
  0.3450   0.2467
  0.3600   0.2447
  0.3750   0.2430
  0.3900   0.2417
  0.4050   0.2406
  0.4200   0.2399
  0.4350   0.2396
  0.4500   0.2399
  0.4650   0.2410
  0.4800   0.2433
  0.4950   0.2469
  0.5100   0.2520
  0.5250   0.2585
  0.5400   0.2659
  0.5550   0.2739
  0.5700   0.2822
  0.5850   0.2903
  0.6000   0.2981
  0.6150   0.6933
  0.6300   1.0413
  0.6450   1.3341
  0.6600   1.5645
  0.6750   1.7271
  0.6900   1.8178
  0.7050   1.8346
  0.7200   1.7771
  0.7350   1.6471
  0.7500   1.4479
  0.7650   1.1847
  0.7800   0.8641
  0.7950   0.4941
  0.8100   0.3634
  0.8250   0.3666
  0.8400   0.3696
  0.8550   0.3726
  0.8700   0.3753
  0.8850   0.3778
  0.9000   0.3800
  0.9150   0.3816
  0.9300   0.3825
  0.9450   0.3827
  0.9600   0.3819
  0.9750   0.3800
  0.9900   0.3769
/}
\Green{\relax   
\!\!\!\!
\plot
  0.0000   0.2807
  0.0150   0.2756
  0.0300   0.2698
  0.0450   0.2631
  0.0600   0.2559
  0.0750   0.2482
  0.0900   0.2403
  0.1050   0.2323
  0.1200   0.2244
  0.1350   0.2168
  0.1500   0.2094
  0.1650   0.2026
  0.1800   0.1963
  0.1950   0.1905
  0.2100   0.1854
  0.2250   0.1810
  0.2400   0.1773
  0.2550   0.1744
  0.2700   0.1722
  0.2850   0.1709
  0.3000   0.1704
  0.3150   0.1708
  0.3300   0.1720
  0.3450   0.1743
  0.3600   0.1775
  0.3750   0.1816
  0.3900   0.1868
  0.4050   0.1929
  0.4200   0.1999
  0.4350   0.2075
  0.4500   0.2156
  0.4650   0.2237
  0.4800   0.2312
  0.4950   0.2379
  0.5100   0.2435
  0.5250   0.2478
  0.5400   0.2512
  0.5550   0.2539
  0.5700   0.2563
  0.5850   0.2586
  0.6000   0.2610
  0.6150   0.6513
  0.6300   0.9952
  0.6450   1.2845
  0.6600   1.5120
  0.6750   1.6721
  0.6900   1.7607
  0.7050   1.7756
  0.7200   1.7164
  0.7350   1.5846
  0.7500   1.3838
  0.7650   1.1189
  0.7800   0.7965
  0.7950   0.4247
  0.8100   0.2921
  0.8250   0.2932
  0.8400   0.2943
  0.8550   0.2952
  0.8700   0.2960
  0.8850   0.2966
  0.9000   0.2970
  0.9150   0.2970
  0.9300   0.2966
  0.9450   0.2956
  0.9600   0.2939
  0.9750   0.2914
  0.9900   0.2880
/
}
%
\Blue{\relax   
\!\!\!\!
\plot
  0.0000   0.0904
  0.0150   0.0859
  0.0300   0.0816
  0.0450   0.0775
  0.0600   0.0736
  0.0750   0.0699
  0.0900   0.0666
  0.1050   0.0636
  0.1200   0.0609
  0.1350   0.0586
  0.1500   0.0566
  0.1650   0.0550
  0.1800   0.0538
  0.1950   0.0530
  0.2100   0.0525
  0.2250   0.0524
  0.2400   0.0527
  0.2550   0.0534
  0.2700   0.0546
  0.2850   0.0561
  0.3000   0.0581
  0.3150   0.0605
  0.3300   0.0635
  0.3450   0.0669
  0.3600   0.0708
  0.3750   0.0754
  0.3900   0.0804
  0.4050   0.0861
  0.4200   0.0922
  0.4350   0.0988
  0.4500   0.1056
  0.4650   0.1124
  0.4800   0.1191
  0.4950   0.1252
  0.5100   0.1307
  0.5250   0.1354
  0.5400   0.1396
  0.5550   0.1433
  0.5700   0.1466
  0.5850   0.1497
  0.6000   0.1526
  0.6150   0.5431
  0.6300   0.8868
  0.6450   1.1754
  0.6600   1.4018
  0.6750   1.5604
  0.6900   1.6471
  0.7050   1.6596
  0.7200   1.5977
  0.7350   1.4629
  0.7500   1.2586
  0.7650   0.9899
  0.7800   0.6634
  0.7950   0.2871
  0.8100   0.1498
  0.8250   0.1459
  0.8400   0.1418
  0.8550   0.1374
  0.8700   0.1328
  0.8850   0.1281
  0.9000   0.1233
  0.9150   0.1184
  0.9300   0.1136
  0.9450   0.1087
  0.9600   0.1040
  0.9750   0.0993
  0.9900   0.0946
/}
\endpicture
}
\setbox\figurefive=\vbox{\hsize=\xfiglen
\beginpicture
\scriptsize
  \setcoordinatesystem units <\xfigdim,0.8\yfigdim>  point at 0 0
  \setplotarea x from 0 to 1, y from 0 to 1.6
  \axis bottom shiftedto y=0 ticks short numbered from 0 to 1 by 0.2 /
  \axis left ticks short numbered from 0 to 1.6 by 0.4 /
  \put {$q(x)$} [lt] at 0.025 1.6
\setquadratic
\setlinear
\setsolid
\Black{\relax  
\plot
  0.0000   0.0679
  0.0150   0.0635
  0.0300   0.0595
  0.0450   0.0557
  0.0600   0.0523
  0.0750   0.0493
  0.0900   0.0466
  0.1050   0.0443
  0.1200   0.0423
  0.1350   0.0406
  0.1500   0.0394
  0.1650   0.0384
  0.1800   0.0378
  0.1950   0.0375
  0.2100   0.0376
  0.2250   0.0380
  0.2400   0.0387
  0.2550   0.0397
  0.2700   0.0412
  0.2850   0.0429
  0.3000   0.0450
  0.3150   0.0475
  0.3300   0.0503
  0.3450   0.0534
  0.3600   0.0569
  0.3750   0.0608
  0.3900   0.0649
  0.4050   0.0694
  0.4200   0.0742
  0.4350   0.0792
  0.4500   0.0844
  0.4650   0.0898
  0.4800   0.0952
  0.4950   0.1008
  0.5100   0.1064
  0.5250   0.1118
  0.5400   0.1172
  0.5550   0.1224
  0.5700   0.1273
  0.5850   0.1318
  0.6000   0.1360
  0.6150   0.5275
  0.6300   0.8720
  0.6450   1.1612
  0.6600   1.3879
  0.6750   1.5467
  0.6900   1.6334
  0.7050   1.6458
  0.7200   1.5837
  0.7350   1.4486
  0.7500   1.2440
  0.7650   0.9749
  0.7800   0.6480
  0.7950   0.2712
  0.8100   0.1333
  0.8250   0.1288
  0.8400   0.1240
  0.8550   0.1190
  0.8700   0.1136
  0.8850   0.1082
  0.9000   0.1027
  0.9150   0.0971
  0.9300   0.0916
  0.9450   0.0862
  0.9600   0.0809
  0.9750   0.0758
  0.9900   0.0710
/}
%
\Red{\relax
\!\!\!\!
\plot
  0.0000  -0.3701
  0.0150  -0.3745
  0.0300  -0.3797
  0.0450  -0.3852
  0.0600  -0.3910
  0.0750  -0.3970
  0.0900  -0.4031
  0.1050  -0.4093
  0.1200  -0.4155
  0.1350  -0.4217
  0.1500  -0.4280
  0.1650  -0.4342
  0.1800  -0.4405
  0.1950  -0.4468
  0.2100  -0.4533
  0.2250  -0.4599
  0.2400  -0.4666
  0.2550  -0.4736
  0.2700  -0.4809
  0.2850  -0.4886
  0.3000  -0.4966
  0.3150  -0.5049
  0.3300  -0.5137
  0.3450  -0.5226
  0.3600  -0.5317
  0.3750  -0.5405
  0.3900  -0.5486
  0.4050  -0.5579
  0.4200  -0.5679
  0.4350  -0.5684
  0.4500  -0.5552
  0.4650  -0.5247
  0.4800  -0.4750
  0.4950  -0.4071
  0.5100  -0.3247
  0.5250  -0.2339
  0.5400  -0.1414
  0.5550  -0.0529
  0.5700   0.0278
  0.5850   0.0988
  0.6000   0.1598
  0.6150   0.5909
  0.6300   0.9656
  0.6450   1.2775
  0.6600   1.5213
  0.6750   1.6930
  0.7050   1.8105
  0.7200   1.7558
  0.7350   1.6280
  0.7500   1.4314
  0.7650   1.1718
  0.7800   0.8565
  0.7950   0.4944
  0.8100   0.3758
  0.8250   0.3955
  0.8400   0.4189
  0.8550   0.4459
  0.8700   0.4764
  0.8850   0.5097
  0.9000   0.5451
  0.9150   0.5811
  0.9300   0.6161
  0.9450   0.6480
  0.9600   0.6743
  0.9750   0.6926
  0.9900   0.7011
/}
\Green{\relax
\!\!\!\!
\plot
  0.0000  -0.3607
  0.0150  -0.3648
  0.0300  -0.3690
  0.0450  -0.3730
  0.0600  -0.3768
  0.0750  -0.3802
  0.0900  -0.3832
  0.1050  -0.3859
  0.1200  -0.3881
  0.1350  -0.3900
  0.1500  -0.3913
  0.1650  -0.3921
  0.1800  -0.3923
  0.1950  -0.3919
  0.2100  -0.3908
  0.2250  -0.3888
  0.2400  -0.3858
  0.2550  -0.3816
  0.2700  -0.3760
  0.2850  -0.3687
  0.3000  -0.3592
  0.3150  -0.3472
  0.3300  -0.3319
  0.3450  -0.3126
  0.3600  -0.2885
  0.3750  -0.2587
  0.3900  -0.2220
  0.4050  -0.1804
  0.4200  -0.1351
  0.4350  -0.0810
  0.4500  -0.0204
  0.4650   0.0429
  0.4800   0.1036
  0.4950   0.1560
  0.5100   0.1956
  0.5250   0.2200
  0.5400   0.2300
  0.5550   0.2285
  0.5700   0.2196
  0.5850   0.2075
  0.6000   0.1955
  0.6150   0.5659
  0.6300   0.8915
  0.6450   1.1649
  0.6600   1.3790
  0.6750   1.5284
  0.6900   1.6088
  0.7050   1.6180
  0.7200   1.5554
  0.7350   1.4226
  0.7500   1.2229
  0.7650   0.9613
  0.7800   0.6446
  0.7950   0.2806
  0.8100   0.1590
  0.8250   0.1741
  0.8400   0.1916
  0.8550   0.2114
  0.8700   0.2334
  0.8850   0.2572
  0.9000   0.2821
  0.9150   0.3074
  0.9300   0.3317
  0.9450   0.3536
  0.9600   0.3714
  0.9750   0.3834
  0.9900   0.3882
/}
%
\Blue{\relax
\!\!\!\!
\plot
  0.0000  -0.0922
  0.0150  -0.0979
  0.0300  -0.1065
  0.0450  -0.1175
  0.0600  -0.1302
  0.0750  -0.1440
  0.0900  -0.1582
  0.1050  -0.1722
  0.1200  -0.1857
  0.1350  -0.1982
  0.1500  -0.2097
  0.1650  -0.2199
  0.1800  -0.2288
  0.1950  -0.2362
  0.2100  -0.2422
  0.2250  -0.2466
  0.2400  -0.2495
  0.2550  -0.2508
  0.2700  -0.2503
  0.2850  -0.2478
  0.3000  -0.2433
  0.3150  -0.2364
  0.3300  -0.2268
  0.3450  -0.2143
  0.3600  -0.1982
  0.3750  -0.1784
  0.3900  -0.1545
  0.4050  -0.1291
  0.4200  -0.1043
  0.4350  -0.0755
  0.4500  -0.0451
  0.4650  -0.0166
  0.4800   0.0065
  0.4950   0.0209
  0.5100   0.0247
  0.5250   0.0183
  0.5400   0.0040
  0.5550  -0.0150
  0.5700  -0.0350
  0.5850  -0.0531
  0.6000  -0.0674
  0.6150   0.3032
  0.6300   0.6303
  0.6450   0.9057
  0.6600   1.1218
  0.6750   1.2727
  0.6900   1.3538
  0.7050   1.3627
  0.7200   1.2988
  0.7350   1.1633
  0.7500   0.9596
  0.7650   0.6927
  0.7800   0.3691
  0.7950  -0.0033
  0.8100  -0.1349
  0.8250  -0.1315
  0.8400  -0.1274
  0.8550  -0.1226
  0.8700  -0.1171
  0.8850  -0.1110
  0.9000  -0.1046
  0.9150  -0.0982
  0.9300  -0.0921
  0.9450  -0.0869
  0.9600  -0.0833
  0.9750  -0.0817
  0.9900  -0.0829
/}
\endpicture
}
\setbox\figuresix=\vbox{\hsize=\xfiglen
\beginpicture
\scriptsize
  \setcoordinatesystem units <\xfigdim,0.8\yfigdim>  point at 0 0
  \setplotarea x from 0 to 1, y from -0.5 to 1.6
  \axis bottom shiftedto y=0 ticks short numbered from 0 to 1 by 0.2 /
  \axis left ticks short numbered from 0 to 1.6 by 0.4 /
  \put {$q(x)$} [lt] at 0.025 1.6
\setlinear
\setsolid 
\Black{\relax  
\plot
0.0000   0.0679
0.0150   0.0635
0.0300   0.0595
0.0450   0.0557
0.0600   0.0523
0.0750   0.0493
0.0900   0.0466
0.1050   0.0443
0.1200   0.0423
0.1350   0.0406
0.1500   0.0394
0.1650   0.0384
0.1800   0.0378
0.1950   0.0375
0.2100   0.0376
0.2250   0.0380
0.2400   0.0387
0.2550   0.0397
0.2700   0.0412
0.2850   0.0429
0.3000   0.0450
0.3150   0.0475
0.3300   0.0503
0.3450   0.0534
0.3600   0.0569
0.3750   0.0608
0.3900   0.0649
0.4050   0.0694
0.4200   0.0742
0.4350   0.0792
0.4500   0.0844
0.4650   0.0898
0.4800   0.0952
0.4950   0.1008
0.5100   0.1064
0.5250   0.1118
0.5400   0.1172
0.5550   0.1224
0.5700   0.1273
0.5850   0.1318
0.6000   0.1360
0.6150   0.5275
0.6300   0.8720
0.6450   1.1612
0.6600   1.3879
0.6750   1.5467
0.6900   1.6334
0.7050   1.6458
0.7200   1.5837
0.7350   1.4486
0.7500   1.2440
0.7650   0.9749
0.7800   0.6480
0.7950   0.2712
0.8100   0.1333
0.8250   0.1288
0.8400   0.1240
0.8550   0.1190
0.8700   0.1136
0.8850   0.1082
0.9000   0.1027
0.9150   0.0971
0.9300   0.0916
0.9450   0.0862
0.9600   0.0809
0.9750   0.0758
0.9900   0.0710
/}
\Red{\relax  
\!\!\!\!
\plot
0.0000  -0.0127
0.0150  -0.0212
0.0300  -0.0382
0.0450  -0.0622
0.0600  -0.0916
0.0750  -0.1244
0.0900  -0.1590
0.1050  -0.1936
0.1200  -0.2270
0.1350  -0.2583
0.1500  -0.2868
0.1650  -0.3119
0.1800  -0.3335
0.1950  -0.3511
0.2100  -0.3648
0.2250  -0.3743
0.2400  -0.3795
0.2550  -0.3800
0.2700  -0.3756
0.2850  -0.3660
0.3000  -0.3505
0.3150  -0.3286
0.3300  -0.2994
0.3450  -0.2621
0.3600  -0.2158
0.3750  -0.1596
0.3900  -0.0930
0.4050  -0.0213
0.4200   0.0515
0.4350   0.1343
0.4500   0.2221
0.4650   0.3074
0.4800   0.3811
0.4950   0.4345
0.5100   0.4616
0.5250   0.4612
0.5400   0.4371
0.5550   0.3964
0.5700   0.3473
0.5850   0.2971
0.6000   0.2514
0.6150   0.5862
0.6300   0.8801
0.6450   1.1258
0.6600   1.3160
0.6750   1.4450
0.6900   1.5082
0.7050   1.5030
0.7200   1.4287
0.7350   1.2866
0.7500   1.0798
0.7650   0.8134
0.7800   0.4941
0.7950   0.1300
0.8100   0.0118
0.8250   0.0334
0.8400   0.0598
0.8550   0.0908
0.8700   0.1259
0.8850   0.1645
0.9000   0.2053
0.9150   0.2471
0.9300   0.2876
0.9450   0.3246
0.9600   0.3552
0.9750   0.3767
0.9900   0.3869
/}
%
\Green{\relax  
\!\!\!\!
\plot
0.0000   0.1152
0.0150   0.1060
0.0300   0.0872
0.0450   0.0603
0.0600   0.0273
0.0750  -0.0097
0.0900  -0.0485
0.1050  -0.0874
0.1200  -0.1248
0.1350  -0.1596
0.1500  -0.1911
0.1650  -0.2187
0.1800  -0.2420
0.1950  -0.2608
0.2100  -0.2748
0.2250  -0.2839
0.2400  -0.2879
0.2550  -0.2865
0.2700  -0.2793
0.2850  -0.2658
0.3000  -0.2454
0.3150  -0.2174
0.3300  -0.1809
0.3450  -0.1348
0.3600  -0.0781
0.3750  -0.0099
0.3900   0.0703
0.4050   0.1570
0.4200   0.2455
0.4350   0.3425
0.4500   0.4408
0.4650   0.5303
0.4800   0.6000
0.4950   0.6398
0.5100   0.6443
0.5250   0.6141
0.5400   0.5558
0.5550   0.4796
0.5700   0.3965
0.5850   0.3156
0.6000   0.2433
0.6150   0.5563
0.6300   0.8329
0.6450   1.0650
0.6600   1.2449
0.6750   1.3661
0.6900   1.4235
0.7050   1.4142
0.7200   1.3370
0.7350   1.1927
0.7500   0.9845
0.7650   0.7169
0.7800   0.3964
0.7950   0.0307
0.8100  -0.0897
0.8250  -0.0710
0.8400  -0.0482
0.8550  -0.0215
0.8700   0.0086
0.8850   0.0416
0.9000   0.0766
0.9150   0.1122
0.9300   0.1468
0.9450   0.1781
0.9600   0.2039
0.9750   0.2219
0.9900   0.2301
/}
%
\Blue{\relax  
\!\!\!\!
\plot
0.0000   0.1695
0.0150   0.1600
0.0300   0.1404
0.0450   0.1122
0.0600   0.0776
0.0750   0.0387
0.0900  -0.0022
0.1050  -0.0432
0.1200  -0.0827
0.1350  -0.1196
0.1500  -0.1530
0.1650  -0.1824
0.1800  -0.2074
0.1950  -0.2278
0.2100  -0.2433
0.2250  -0.2538
0.2400  -0.2590
0.2550  -0.2588
0.2700  -0.2527
0.2850  -0.2403
0.3000  -0.2211
0.3150  -0.1943
0.3300  -0.1590
0.3450  -0.1144
0.3600  -0.0595
0.3750   0.0065
0.3900   0.0841
0.4050   0.1677
0.4200   0.2523
0.4350   0.3449
0.4500   0.4380
0.4650   0.5219
0.4800   0.5857
0.4950   0.6199
0.5100   0.6194
0.5250   0.5851
0.5400   0.5239
0.5550   0.4459
0.5700   0.3619
0.5850   0.2806
0.6000   0.2085
0.6150   0.5219
0.6300   0.7991
0.6450   1.0316
0.6600   1.2119
0.6750   1.3334
0.6900   1.3910
0.7050   1.3817
0.7200   1.3042
0.7350   1.1596
0.7500   0.9508
0.7650   0.6824
0.7800   0.3609
0.7950  -0.0059
0.8100  -0.1277
0.8250  -0.1107
0.8400  -0.0898
0.8550  -0.0653
0.8700  -0.0374
0.8850  -0.0069
0.9000   0.0255
0.9150   0.0584
0.9300   0.0904
0.9450   0.1194
0.9600   0.1432
0.9750   0.1597
0.9900   0.1670
/}
\endpicture
} 
\setbox\figurequarteruthree =\vbox{
\hbox to \hsize{\copy\figureone\hss\copy\figuretwo\hss\copy\figurethree}
\vskip20pt
\hbox to \hsize{\copy\figurefour\hss\copy\figurefive\hss\copy\figuresix}
}
%
%
%
%
%
%

%
\copy\figureutwo  
\copy\figurequarteruthree
\vfill\eject
\copy\figurefourutwo  
\copy\figureuthree

\vspace*{1cm}
It is clearly visible that the larger $f$ values in $f_3(u) = 4u^2$ and $f_4(u) = u^3$ put a larger emphasis on the difference between the $p(x)$ and the $q(x)$ term, thus allowing for a better reconstruction of both.

We started all our reconstructions with zero values of $p$ and $q$.
Reconstruction by the fixed point schemes worked remarkably well, even with this very crude initial guess.
However, the nonlinearity $f(u)$ needs to yield sufficiently large (in magnitude) values of the states $u$ in order to provide sufficient sensitivity of the observation with respect to the multiplier $p$. 
The reconstruction quality of $q$ heavily depends on how well $p$ has been recovered and the error in $q$ is usually larger than the one in $p$. This is due to the fact that the factor $f(u)$ of $p$ emphasises this coefficient more strongly than the factor $u$ of $q$.
The two-run version \eqref{obs} (left column) obviously incorporates more independent information, thus leading to substantially better reconstructions as compared to the one-run and two times observation \eqref{obs2} setting; in the latter, $T_1$ being closer to zero and further away from $T_2$ (middle column) yields better results.

\medskip 

In order to illustrate the influence of the time differentiation order, in Table~\ref{table:alpha} we list some iterates obtained with different values of $\alpha$, using $f(u)=u^3$ and taking measurements at $T=0.3$ for two different sources. Interestingly, the first iterate becomes worse for increasing $\alpha$, while quickly catching up during the iterates so that at the 6th step, relative errors are smaller for larger $\alpha$.
\begin{table}
\begin{center}
\begin{tabular}{||l||c|c||c|c||}
\hline
$\alpha$ & \multicolumn{2}{|c|}{p} & \multicolumn{2}{|c|}{q}\\
\hline
\hline
&it=1 & it=6&it=1 & it=6\\
\hline
0.5
&0.014132  
&0.000748  
%
&0.121003  
&0.001565  
\\ \hline
0.6
&0.014231  
&0.000763  
%
&0.120720  
&0.001590  
\\ \hline
0.7
&0.017053  
&0.000738  
%
&0.134141  
&0.001548  
\\ \hline
0.8
&0.021438  
&0.000695  
%
&0.147655  
&0.001429  
\\ \hline
0.9
&0.029297  
&0.000669  
%
&0.172068  
&0.001375  
\\ \hline
0.99
&0.040586  
&0.000638  
%
&0.208235  
&0.001321  
\\ \hline\hline
\end{tabular}
\end{center}
\caption{Relative errors for varying differentiation order $\alpha$; 
$f(u)=u^3$, two runs, $T=0.3$.
\label{table:alpha}}
\end{table}

\medskip 

Finally, we test the behaviour with noise, again in the scenario of two runs with different sources, observations at $T=0.3$, and $f(u)=u^3$; the value of $\alpha=0.8$ is the same as the one in the error and reconstruction plots above; see Table~\ref{table:delta}.
To this end, normally distributed random noise was added to the exact values of $u_1(T,x)$, $u_2(T,x)$ and the result was smoothed and rescaled, so that the $H^2(\Omega)$ norm matches the given percentages. This corresponds to the setting of Remark~\ref{rem:noisydata}. Correspondingly, we observe an $O(\delta)$ convergence rate in the 10th iterate. Note that no early stopping is needed; we only have to make sure that at least $\sim|\log(\delta)|$ fixed point steps are carried out.
\begin{table}
\begin{center}
\begin{tabular}{||l||c|c||}
\hline
$\delta$ & p & q\\
\hline
\hline
3\%& 0.034690 & 0.140247\\
1\%& 0.011105 & 0.044538\\
0.3\%& 0.003228 & 0.014412\\
0.1\%& 0.001208 & 0.005018
\\ \hline\hline
\end{tabular}
\end{center}
\caption{Relative errors at 10th iterate for varying noise level $\delta$; 
$f(u)=u^3$, two runs, $T=0.3$, $\alpha=0.8$.
\label{table:delta}}
\end{table}

\subsection{Possible alternative observation scenarios}

Our basic paradigm of having an input source  $r(t,x)$
and and measuring {\it census\/} data at later times $T_1$ and $T_2$ or
having a pair of sources $r_i(t,x)$, $i=1,2$ and making measurements for
each at a single time $T$ is clearly open to alternatives.
What is feasible depends of course on the application and what changes in
any experimental set-up are possible.
For population models it may very well be the case that measurements
(spatial census data) at a series of times $\{T_i\}$ are available while
changing other factors in the equation and initial/boundary conditions are not.
For the many applications of the two-term model with $p(x)$ and $q(x)$
unknown and where no additional source changes in $r(t,x)$ are possible,
it may be feasible, for example, to
re-run the experiment with two different initial values $u_{0,i}(x)$, $i=1,\,2$.
Alternatively, one might change the boundary conditions on $\partial\Omega$
to include both Neumann and Dirichlet types or, more generally, modify the
boundary conditions to be of impedance type
$\frac{\partial u}{\partial\nu} + \beta u = a$ and be able to modify
either $\beta$ or the source $a(t,x)$.

A reconstruction of 
two phantoms similar to the previous ones
for $p$ and $q$ is shown in the figure
below in the case of having two initial conditions 
$u_{0,1}(x)=\frac12(1+\cos(\pi(1-x)))$,
$u_{0,2}(x)=1 - 2(1-x)^3 + 3(1-x)^2$
and $f(u)=u^2$, as well as $\alpha=1$, $T=0.5$, and homo\-ge\-neous Neumann conditions.
Here we used as starting approximations the constant values of
$p(x)=0.5$ and $q(x)=4.0$.
Although these are approximately the midpoint value of the respective functions
they are quite far apart from the actual in any $L^\rrr\times L^\sss$ norm.
As expected from such a choice, the first iteration may give quite a poor 
update as we see in particular here.
\medskip

 \newbox\pqfigurelegend
\xfigdim = 2.0true in
\yfigdim = 1.2true in
%
\setbox\pqfigurelegend=\vbox{\hsize=\xfigdim   %
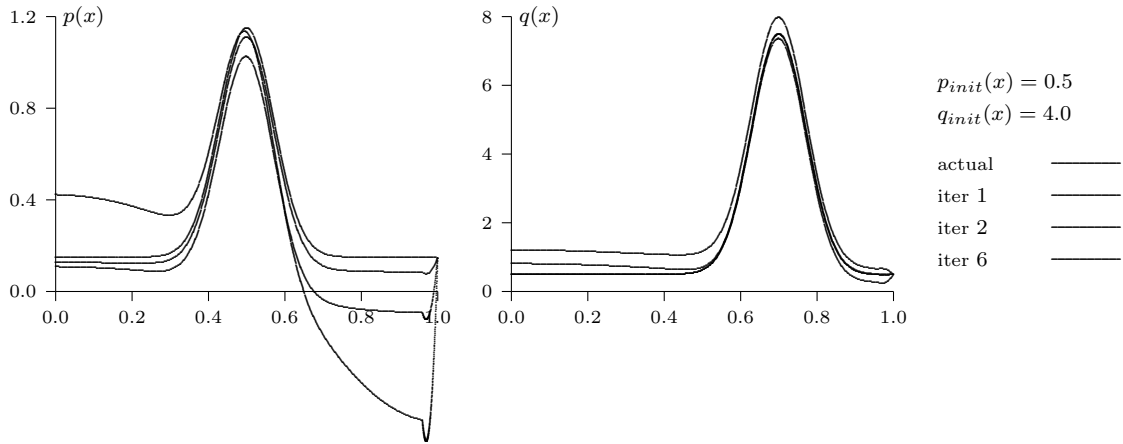
\beginpicture
  \setcoordinatesystem units <0.6\xfigdim,0.7\yfigdim>  point at 0 0
  \setplotarea x from 0 to 1, y from 0.0 to 1.4
\scriptsize
\sevenrm
\put{$p_{init}(x) = 0.5$} [l] at 0 1.3
\put{$q_{init}(x) = 4.0$} [l] at 0 1.1
\put{actual} [l] at 0 0.8 \plot 0.5 0.8 0.8 0.8 /
\put{iter $1$} [l] at 0 0.6  \Red{\plot 0.5 0.6 0.8 0.6 / }\relax
\put{iter $2$} [l] at 0 0.4  \Green{\plot 0.5 0.4 0.8 0.4 /}\relax
\put{iter $6$} [l] at 0 0.2  \Blue{\plot 0.5 0.2 0.8 0.2 /}\relax
\endpicture
}
\setbox\figureone=\vbox{\hsize=\xfigdim   
%
%
\beginpicture
  \setcoordinatesystem units <\xfigdim,\yfigdim>  point at 0 0
  \setplotarea x from 0 to 1.0, y from 0 to 1.2
\sevenrm
  \axis bottom shiftedto y=0 ticks short numbered from 0 to 1 by 0.2 /
  \axis left ticks short numbered from 0 to 1.2 by 0.4 /
\put{\scriptsize{$p(x)$}} [l] at 0.02 1.2
\setquadratic
\setdashes <3pt>
\setsolid
\Black{\relax  
\plot
         0    0.1500
    0.0200    0.1500
    0.0400    0.1500
    0.0600    0.1500
    0.0800    0.1500
    0.1000    0.1500
    0.1200    0.1500
    0.1400    0.1500
    0.1600    0.1500
    0.1800    0.1500
    0.2000    0.1501
    0.2200    0.1504
    0.2400    0.1512
    0.2600    0.1532
    0.2800    0.1579
    0.3000    0.1683
    0.3200    0.1892
    0.3400    0.2273
    0.3600    0.2909
    0.3800    0.3869
    0.4000    0.5179
    0.4200    0.6773
    0.4400    0.8477
    0.4600    1.0021
    0.4800    1.1108
    0.5000    1.1500
    0.5200    1.1108
    0.5400    1.0021
    0.5600    0.8477
    0.5800    0.6773
    0.6000    0.5179
    0.6200    0.3869
    0.6400    0.2909
    0.6600    0.2273
    0.6800    0.1892
    0.7000    0.1683
    0.7200    0.1579
    0.7400    0.1532
    0.7600    0.1512
    0.7800    0.1504
    0.8000    0.1501
    0.8200    0.1500
    0.8400    0.1500
    0.8600    0.1500
    0.8800    0.1500
    0.9000    0.1500
    0.9200    0.1500
    0.9400    0.1500
    0.9600    0.1500
    0.9800    0.1500
    1.0000    0.1500
/}\relax
\Red{\relax  
\plot
         0    0.4229
    0.0200    0.4210
    0.0400    0.4195
    0.0600    0.4170
    0.0800    0.4135
    0.1000    0.4091
    0.1200    0.4036
    0.1400    0.3973
    0.1600    0.3899
    0.1800    0.3817
    0.2000    0.3727
    0.2200    0.3629
    0.2400    0.3528
    0.2600    0.3430
    0.2800    0.3353
    0.3000    0.3324
    0.3200    0.3392
    0.3400    0.3626
    0.3600    0.4106
    0.3800    0.4904
    0.4000    0.6042
    0.4200    0.7455
    0.4400    0.8966
    0.4600    1.0304
    0.4800    1.1169
    0.5000    1.1321
    0.5200    1.0671
    0.5400    0.9306
    0.5600    0.7464
    0.5800    0.5444
    0.6000    0.3515
    0.6200    0.1854
    0.6400    0.0525
    0.6600   -0.0493
    0.6800   -0.1269
    0.7000   -0.1880
    0.7200   -0.2389
    0.7400   -0.2838
    0.7600   -0.3247
    0.7800   -0.3626
    0.8000   -0.3976
    0.8200   -0.4297
    0.8400   -0.4585
    0.8600   -0.4841
    0.8800   -0.5063
    0.9000   -0.5251
    0.9200   -0.5405
    0.9400   -0.5524
    0.9600   -0.5609
    0.9800   -0.5660
    1.0000    0.1476
/}\relax
\Green{\relax 
\plot
         0    0.1091
    0.0200    0.1076
    0.0400    0.1072
    0.0600    0.1065
    0.0800    0.1055
    0.1000    0.1043
    0.1200    0.1029
    0.1400    0.1011
    0.1600    0.0991
    0.1800    0.0969
    0.2000    0.0944
    0.2200    0.0919
    0.2400    0.0896
    0.2600    0.0883
    0.2800    0.0895
    0.3000    0.0961
    0.3200    0.1128
    0.3400    0.1467
    0.3600    0.2056
    0.3800    0.2969
    0.4000    0.4228
    0.4200    0.5769
    0.4400    0.7416
    0.4600    0.8902
    0.4800    0.9926
    0.5000    1.0252
    0.5200    0.9789
    0.5400    0.8629
    0.5600    0.7007
    0.5800    0.5223
    0.6000    0.3547
    0.6200    0.2154
    0.6400    0.1110
    0.6600    0.0393
    0.6800   -0.0065
    0.7000   -0.0345
    0.7200   -0.0514
    0.7400   -0.0619
    0.7600   -0.0689
    0.7800   -0.0739
    0.8000   -0.0778
    0.8200   -0.0809
    0.8400   -0.0835
    0.8600   -0.0857
    0.8800   -0.0875
    0.9000   -0.0890
    0.9200   -0.0902
    0.9400   -0.0911
    0.9600   -0.0917
    0.9800   -0.0921
    1.0000    0.1476
/}\relax
\Blue{\relax 
\plot
         0    0.1280
    0.0200    0.1266
    0.0400    0.1265
    0.0600    0.1264
    0.0800    0.1262
    0.1000    0.1260
    0.1200    0.1257
    0.1400    0.1254
    0.1600    0.1250
    0.1800    0.1246
    0.2000    0.1242
    0.2200    0.1239
    0.2400    0.1241
    0.2600    0.1255
    0.2800    0.1295
    0.3000    0.1392
    0.3200    0.1593
    0.3400    0.1966
    0.3600    0.2592
    0.3800    0.3544
    0.4000    0.4843
    0.4200    0.6427
    0.4400    0.8120
    0.4600    0.9653
    0.4800    1.0726
    0.5000    1.1105
    0.5200    1.0698
    0.5400    0.9596
    0.5600    0.8034
    0.5800    0.6313
    0.6000    0.4701
    0.6200    0.3373
    0.6400    0.2394
    0.6600    0.1741
    0.6800    0.1343
    0.7000    0.1119
    0.7200    0.1001
    0.7400    0.0941
    0.7600    0.0910
    0.7800    0.0893
    0.8000    0.0882
    0.8200    0.0875
    0.8400    0.0869
    0.8600    0.0864
    0.8800    0.0859
    0.9000    0.0856
    0.9200    0.0853
    0.9400    0.0851
    0.9600    0.0849
    0.9800    0.0848
    1.0000    0.1476
/}\relax
\endpicture
}
\yfigdim = 0.18true in
\setbox\figuretwo=\vbox{\hsize=\xfigdim   
\beginpicture
  \setcoordinatesystem units <\xfigdim,\yfigdim>  point at 0 0
  \setplotarea x from 0 to 1.0, y from 0 to 8
\sevenrm
  \axis bottom shiftedto y=0 ticks short numbered from 0 to 1 by 0.2 /
  \axis left ticks short numbered from 0 to 8 by 2 /
\put{\scriptsize{$q(x)$}} [l] at 0.02 8
\setquadratic
\setdashes <3pt>
\setsolid
\Black{\relax  
\plot
         0    0.5000
    0.0200    0.5000
    0.0400    0.5000
    0.0600    0.5000
    0.0800    0.5000
    0.1000    0.5000
    0.1200    0.5000
    0.1400    0.5000
    0.1600    0.5000
    0.1800    0.5000
    0.2000    0.5000
    0.2200    0.5000
    0.2400    0.5000
    0.2600    0.5000
    0.2800    0.5000
    0.3000    0.5000
    0.3200    0.5000
    0.3400    0.5000
    0.3600    0.5001
    0.3800    0.5002
    0.4000    0.5009
    0.4200    0.5028
    0.4400    0.5081
    0.4600    0.5221
    0.4800    0.5553
    0.5000    0.6282
    0.5200    0.7741
    0.5400    1.0411
    0.5600    1.4860
    0.5800    2.1585
    0.6000    3.0752
    0.6200    4.1910
    0.6400    5.3837
    0.6600    6.4650
    0.6800    7.2255
    0.7000    7.5000
    0.7200    7.2255
    0.7400    6.4650
    0.7600    5.3837
    0.7800    4.1910
    0.8000    3.0752
    0.8200    2.1585
    0.8400    1.4860
    0.8600    1.0411
    0.8800    0.7741
    0.9000    0.6282
    0.9200    0.5553
    0.9400    0.5221
    0.9600    0.5081
    0.9800    0.5028
    1.0000    0.5009
/}\relax
\Red{\relax  
\plot
         0    1.2095
    0.0200    1.2086
    0.0400    1.2072
    0.0600    1.2050
    0.0800    1.2019
    0.1000    1.1980
    0.1200    1.1933
    0.1400    1.1878
    0.1600    1.1815
    0.1800    1.1746
    0.2000    1.1671
    0.2200    1.1590
    0.2400    1.1505
    0.2600    1.1415
    0.2800    1.1323
    0.3000    1.1227
    0.3200    1.1131
    0.3400    1.1034
    0.3600    1.0939
    0.3800    1.0846
    0.4000    1.0761
    0.4200    1.0692
    0.4400    1.0664
    0.4600    1.0729
    0.4800    1.0995
    0.5000    1.1666
    0.5200    1.3078
    0.5400    1.5710
    0.5600    2.0129
    0.5800    2.6831
    0.6000    3.5975
    0.6200    4.7103
    0.6400    5.8981
    0.6600    6.9714
    0.6800    7.7197
    0.7000    7.9770
    0.7200    7.6800
    0.7400    6.8922
    0.7600    5.7799
    0.7800    4.5541
    0.8000    3.4043
    0.8200    2.4546
    0.8400    1.7510
    0.8600    1.2778
    0.8800    0.9857
    0.9000    0.8183
    0.9200    0.7277
    0.9400    0.6806
    0.9600    0.6567
    0.9800    0.6454
    1.0000    0.5002
/}\relax
\Green{\relax 
\plot
         0    0.8121
    0.0200    0.8113
    0.0400    0.8102
    0.0600    0.8083
    0.0800    0.8057
    0.1000    0.8023
    0.1200    0.7981
    0.1400    0.7933
    0.1600    0.7876
    0.1800    0.7813
    0.2000    0.7742
    0.2200    0.7664
    0.2400    0.7579
    0.2600    0.7487
    0.2800    0.7389
    0.3000    0.7283
    0.3200    0.7170
    0.3400    0.7051
    0.3600    0.6925
    0.3800    0.6794
    0.4000    0.6658
    0.4200    0.6528
    0.4400    0.6423
    0.4600    0.6394
    0.4800    0.6547
    0.5000    0.7084
    0.5200    0.8340
    0.5400    1.0795
    0.5600    1.5018
    0.5800    2.1509
    0.6000    3.0436
    0.6200    4.1355
    0.6400    5.3046
    0.6600    6.3632
    0.6800    7.1026
    0.7000    7.3578
    0.7200    7.0663
    0.7400    6.2912
    0.7600    5.1976
    0.7800    3.9947
    0.8000    2.8705
    0.8200    1.9470
    0.8400    1.2690
    0.8600    0.8194
    0.8800    0.5486
    0.9000    0.3995
    0.9200    0.3241
    0.9400    0.2888
    0.9600    0.2734
    0.9800    0.2672
    1.0000    0.5002
/}\relax
\Blue{\relax 
\plot
         0    0.5127
    0.0200    0.5123
    0.0400    0.5122
    0.0600    0.5121
    0.0800    0.5119
    0.1000    0.5117
    0.1200    0.5114
    0.1400    0.5111
    0.1600    0.5107
    0.1800    0.5103
    0.2000    0.5099
    0.2200    0.5094
    0.2400    0.5088
    0.2600    0.5082
    0.2800    0.5076
    0.3000    0.5070
    0.3200    0.5063
    0.3400    0.5056
    0.3600    0.5049
    0.3800    0.5043
    0.4000    0.5041
    0.4200    0.5052
    0.4400    0.5097
    0.4600    0.5228
    0.4800    0.5552
    0.5000    0.6271
    0.5200    0.7720
    0.5400    1.0380
    0.5600    1.4819
    0.5800    2.1534
    0.6000    3.0691
    0.6200    4.1841
    0.6400    5.3760
    0.6600    6.4564
    0.6800    7.2161
    0.7000    7.4897
    0.7200    7.2143
    0.7400    6.4528
    0.7600    5.3705
    0.7800    4.1769
    0.8000    3.0602
    0.8200    2.1429
    0.8400    1.4698
    0.8600    1.0244
    0.8800    0.7571
    0.9000    0.6108
    0.9200    0.5376
    0.9400    0.5041
    0.9600    0.4900
    0.9800    0.4845
    1.0000    0.5002
/}\relax
\endpicture
}
%

\noindent
\smallskip
\hbox to \hsize{\hss\copy\figureone\hss\copy\figuretwo\hss\copy\pqfigurelegend\hss}


\vskip50pt

The next figure shows reconstructions in the same setting $f(u)=u^2$, $\alpha=1$, $T=0.5$, 
from two different boundary conditions with the same initial condition $u_0(x)=1+\cos(\pi x)$,  
Dirichlet $u(t,0)=2-t$ at $x=0$ combined with homogeneous Neumann at $x=1$ and homogeneous Neumann at both $x=0$ and $x=1$.

 \newbox\pqfigurerates
\xfigdim = 2.0true in
\yfigdim = 1.2true in
\setbox\pqfigurelegend=\vbox{\hsize=\xfigdim   %
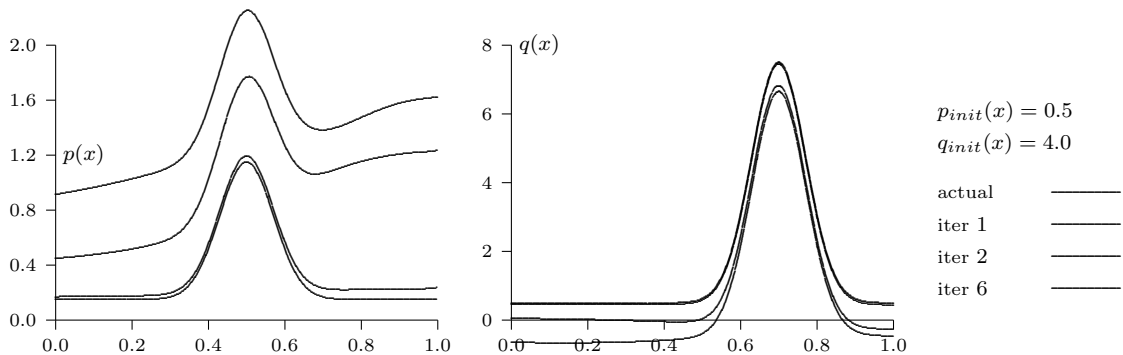
\beginpicture
  \setcoordinatesystem units <0.6\xfigdim,0.7\yfigdim>  point at 0 0
  \setplotarea x from 0 to 1, y from 0.0 to 1.4
\scriptsize
\sevenrm
\put{$p_{init}(x) = 0.5$} [l] at 0 1.3
\put{$q_{init}(x) = 4.0$} [l] at 0 1.1
\put{actual} [l] at 0 0.8 \plot 0.5 0.8 0.8 0.8 /
\put{iter $1$} [l] at 0 0.6  \Red{\plot 0.5 0.6 0.8 0.6 / }\relax
\put{iter $2$} [l] at 0 0.4  \Green{\plot 0.5 0.4 0.8 0.4 /}\relax
\put{iter $6$} [l] at 0 0.2  \Blue{\plot 0.5 0.2 0.8 0.2 /}\relax
\endpicture
}
\setbox\figureone=\vbox{\hsize=\xfigdim   
%
%
\beginpicture
  \setcoordinatesystem units <\xfigdim,0.6\yfigdim>  point at 0 0
  \setplotarea x from 0 to 1.0, y from 0 to 2
\sevenrm
  \axis bottom shiftedto y=0 ticks short numbered from 0 to 1 by 0.2 /
  \axis left ticks short numbered from 0 to 2 by 0.4 /
\put{\scriptsize{$p(x)$}} [l] at 0.02 1.2
\setquadratic
\setdashes <3pt>
\setsolid
\Black{\relax  
\plot
         0    0.1500
    0.0200    0.1500
    0.0400    0.1500
    0.0600    0.1500
    0.0800    0.1500
    0.1000    0.1500
    0.1200    0.1500
    0.1400    0.1500
    0.1600    0.1500
    0.1800    0.1500
    0.2000    0.1501
    0.2200    0.1504
    0.2400    0.1512
    0.2600    0.1532
    0.2800    0.1579
    0.3000    0.1683
    0.3200    0.1892
    0.3400    0.2273
    0.3600    0.2909
    0.3800    0.3869
    0.4000    0.5179
    0.4200    0.6773
    0.4400    0.8477
    0.4600    1.0021
    0.4800    1.1108
    0.5000    1.1500
    0.5200    1.1108
    0.5400    1.0021
    0.5600    0.8477
    0.5800    0.6773
    0.6000    0.5179
    0.6200    0.3869
    0.6400    0.2909
    0.6600    0.2273
    0.6800    0.1892
    0.7000    0.1683
    0.7200    0.1579
    0.7400    0.1532
    0.7600    0.1512
    0.7800    0.1504
    0.8000    0.1501
    0.8200    0.1500
    0.8400    0.1500
    0.8600    0.1500
    0.8800    0.1500
    0.9000    0.1500
    0.9200    0.1500
    0.9400    0.1500
    0.9600    0.1500
    0.9800    0.1500
    1.0000    0.1500
/}\relax
\Red{\relax  
\plot
         0    0.9127
    0.0200    0.9235
    0.0400    0.9333
    0.0600    0.9435
    0.0800    0.9542
    0.1000    0.9654
    0.1200    0.9770
    0.1400    0.9890
    0.1600    1.0015
    0.1800    1.0143
    0.2000    1.0276
    0.2200    1.0415
    0.2400    1.0563
    0.2600    1.0726
    0.2800    1.0920
    0.3000    1.1174
    0.3200    1.1534
    0.3400    1.2070
    0.3600    1.2861
    0.3800    1.3979
    0.4000    1.5446
    0.4200    1.7197
    0.4400    1.9055
    0.4600    2.0750
    0.4800    2.1982
    0.5000    2.2514
    0.5200    2.2253
    0.5400    2.1290
    0.5600    1.9860
    0.5800    1.8262
    0.6000    1.6766
    0.6200    1.5551
    0.6400    1.4687
    0.6600    1.4156
    0.6800    1.3896
    0.7000    1.3833
    0.7200    1.3903
    0.7400    1.4058
    0.7600    1.4263
    0.7800    1.4497
    0.8000    1.4741
    0.8200    1.4983
    0.8400    1.5213
    0.8600    1.5425
    0.8800    1.5615
    0.9000    1.5779
    0.9200    1.5914
    0.9400    1.6021
    0.9600    1.6098
    0.9800    1.6167
    1.0000    1.6258
/}\relax
\Green{\relax 
\plot
         0    0.4479
    0.0200    0.4538
    0.0400    0.4587
    0.0600    0.4641
    0.0800    0.4699
    0.1000    0.4763
    0.1200    0.4833
    0.1400    0.4908
    0.1600    0.4989
    0.1800    0.5077
    0.2000    0.5172
    0.2200    0.5276
    0.2400    0.5392
    0.2600    0.5529
    0.2800    0.5701
    0.3000    0.5939
    0.3200    0.6291
    0.3400    0.6825
    0.3600    0.7624
    0.3800    0.8758
    0.4000    1.0253
    0.4200    1.2044
    0.4400    1.3956
    0.4600    1.5719
    0.4800    1.7036
    0.5000    1.7670
    0.5200    1.7530
    0.5400    1.6707
    0.5600    1.5437
    0.5800    1.4019
    0.6000    1.2719
    0.6200    1.1711
    0.6400    1.1054
    0.6600    1.0718
    0.6800    1.0627
    0.7000    1.0693
    0.7200    1.0842
    0.7400    1.1024
    0.7600    1.1208
    0.7800    1.1381
    0.8000    1.1537
    0.8200    1.1676
    0.8400    1.1798
    0.8600    1.1905
    0.8800    1.1998
    0.9000    1.2076
    0.9200    1.2141
    0.9400    1.2191
    0.9600    1.2228
    0.9800    1.2272
    1.0000    1.2356
/}\relax
\Blue{\relax 
\plot
         0    0.1692
    0.0200    0.1710
    0.0400    0.1714
    0.0600    0.1718
    0.0800    0.1723
    0.1000    0.1727
    0.1200    0.1732
    0.1400    0.1738
    0.1600    0.1744
    0.1800    0.1750
    0.2000    0.1758
    0.2200    0.1768
    0.2400    0.1783
    0.2600    0.1812
    0.2800    0.1868
    0.3000    0.1982
    0.3200    0.2201
    0.3400    0.2593
    0.3600    0.3240
    0.3800    0.4213
    0.4000    0.5536
    0.4200    0.7145
    0.4400    0.8864
    0.4600    1.0425
    0.4800    1.1528
    0.5000    1.1938
    0.5200    1.1564
    0.5400    1.0497
    0.5600    0.8972
    0.5800    0.7289
    0.6000    0.5717
    0.6200    0.4429
    0.6400    0.3491
    0.6600    0.2878
    0.6800    0.2518
    0.7000    0.2330
    0.7200    0.2245
    0.7400    0.2215
    0.7600    0.2211
    0.7800    0.2218
    0.8000    0.2228
    0.8200    0.2238
    0.8400    0.2248
    0.8600    0.2256
    0.8800    0.2264
    0.9000    0.2270
    0.9200    0.2275
    0.9400    0.2279
    0.9600    0.2282
    0.9800    0.2306
    1.0000    0.2383
/}\relax
\endpicture
}
\yfigdim = 0.18true in
\setbox\figuretwo=\vbox{\hsize=\xfigdim   
\beginpicture
  \setcoordinatesystem units <\xfigdim,\yfigdim>  point at 0 0
  \setplotarea x from 0 to 1.0, y from 0 to 8
\sevenrm
  \axis bottom shiftedto y=0 ticks short numbered from 0 to 1 by 0.2 /
  \axis left ticks short numbered from 0 to 8 by 2 /
\put{\scriptsize{$q(x)$}} [l] at 0.02 8
\setquadratic
\setdashes <3pt>
\setsolid
\Black{\relax  
\plot
         0    0.5000
    0.0200    0.5000
    0.0400    0.5000
    0.0600    0.5000
    0.0800    0.5000
    0.1000    0.5000
    0.1200    0.5000
    0.1400    0.5000
    0.1600    0.5000
    0.1800    0.5000
    0.2000    0.5000
    0.2200    0.5000
    0.2400    0.5000
    0.2600    0.5000
    0.2800    0.5000
    0.3000    0.5000
    0.3200    0.5000
    0.3400    0.5000
    0.3600    0.5001
    0.3800    0.5002
    0.4000    0.5009
    0.4200    0.5028
    0.4400    0.5081
    0.4600    0.5221
    0.4800    0.5553
    0.5000    0.6282
    0.5200    0.7741
    0.5400    1.0411
    0.5600    1.4860
    0.5800    2.1585
    0.6000    3.0752
    0.6200    4.1910
    0.6400    5.3837
    0.6600    6.4650
    0.6800    7.2255
    0.7000    7.5000
    0.7200    7.2255
    0.7400    6.4650
    0.7600    5.3837
    0.7800    4.1910
    0.8000    3.0752
    0.8200    2.1585
    0.8400    1.4860
    0.8600    1.0411
    0.8800    0.7741
    0.9000    0.6282
    0.9200    0.5553
    0.9400    0.5221
    0.9600    0.5081
    0.9800    0.5028
    1.0000    0.5009
/}\relax
\Red{\relax  
\plot
        0   -0.6463
    0.0200   -0.6504
    0.0400   -0.6539
    0.0600   -0.6569
    0.0800   -0.6592
    0.1000   -0.6608
    0.1200   -0.6618
    0.1400   -0.6621
    0.1600   -0.6617
    0.1800   -0.6605
    0.2000   -0.6586
    0.2200   -0.6560
    0.2400   -0.6525
    0.2600   -0.6481
    0.2800   -0.6429
    0.3000   -0.6368
    0.3200   -0.6297
    0.3400   -0.6216
    0.3600   -0.6125
    0.3800   -0.6021
    0.4000   -0.5903
    0.4200   -0.5762
    0.4400   -0.5575
    0.4600   -0.5292
    0.4800   -0.4805
    0.5000   -0.3912
    0.5200   -0.2278
    0.5400    0.0576
    0.5600    0.5218
    0.5800    1.2143
    0.6000    2.1513
    0.6200    3.2871
    0.6400    4.4984
    0.6600    5.5958
    0.6800    6.3686
    0.7000    6.6508
    0.7200    6.3789
    0.7400    5.6160
    0.7600    4.5280
    0.7800    3.3252
    0.8000    2.1972
    0.8200    1.2674
    0.8400    0.5819
    0.8600    0.1246
    0.8800   -0.1537
    0.9000   -0.3095
    0.9200   -0.3905
    0.9400   -0.4303
    0.9600   -0.4488
    0.9800   -0.4611
    1.0000   -0.4770
/}\relax
\Green{\relax 
\plot
         0    0.0509
    0.0200    0.0488
    0.0400    0.0463
    0.0600    0.0436
    0.0800    0.0406
    0.1000    0.0373
    0.1200    0.0338
    0.1400    0.0300
    0.1600    0.0260
    0.1800    0.0216
    0.2000    0.0170
    0.2200    0.0121
    0.2400    0.0069
    0.2600    0.0015
    0.2800   -0.0043
    0.3000   -0.0104
    0.3200   -0.0167
    0.3400   -0.0234
    0.3600   -0.0303
    0.3800   -0.0373
    0.4000   -0.0442
    0.4200   -0.0501
    0.4400   -0.0528
    0.4600   -0.0471
    0.4800   -0.0222
    0.5000    0.0420
    0.5200    0.1792
    0.5400    0.4373
    0.5600    0.8731
    0.5800    1.5365
    0.6000    2.4439
    0.6200    3.5505
    0.6400    4.7339
    0.6600    5.8058
    0.6800    6.5570
    0.7000    6.8221
    0.7200    6.5383
    0.7400    5.7686
    0.7600    4.6785
    0.7800    3.4774
    0.8000    2.3536
    0.8200    1.4297
    0.8400    0.7507
    0.8600    0.2999
    0.8800    0.0278
    0.9000   -0.1226
    0.9200   -0.1991
    0.9400   -0.2353
    0.9600   -0.2513
    0.9800   -0.2620
    1.0000   -0.2774
/}\relax
\Blue{\relax 
\plot
         0    0.4689
    0.0200    0.4688
    0.0400    0.4686
    0.0600    0.4684
    0.0800    0.4682
    0.1000    0.4679
    0.1200    0.4677
    0.1400    0.4674
    0.1600    0.4671
    0.1800    0.4667
    0.2000    0.4664
    0.2200    0.4660
    0.2400    0.4657
    0.2600    0.4652
    0.2800    0.4648
    0.3000    0.4644
    0.3200    0.4639
    0.3400    0.4634
    0.3600    0.4629
    0.3800    0.4626
    0.4000    0.4626
    0.4200    0.4639
    0.4400    0.4686
    0.4600    0.4819
    0.4800    0.5145
    0.5000    0.5867
    0.5200    0.7319
    0.5400    0.9982
    0.5600    1.4423
    0.5800    2.1141
    0.6000    3.0300
    0.6200    4.1453
    0.6400    5.3374
    0.6600    6.4180
    0.6800    7.1779
    0.7000    7.4517
    0.7200    7.1764
    0.7400    6.4150
    0.7600    5.3329
    0.7800    4.1394
    0.8000    3.0227
    0.8200    2.1054
    0.8400    1.4324
    0.8600    0.9871
    0.8800    0.7197
    0.9000    0.5735
    0.9200    0.5003
    0.9400    0.4668
    0.9600    0.4527
    0.9800    0.4432
    1.0000    0.4281
/}\relax
\endpicture
}
\hbox to \hsize{\hss\copy\figureone\hss\copy\figuretwo\hss\copy\pqfigurelegend\hss}
%

\vskip50pt

As a general guidance, since for values of $p(x)$, $q(x)$, one has
to (iteratively) solve the forwards problem to obtain $ g_i(x) := u_t(x,T;p,q)$
for two distinct runs, it is important that the determinant formed from
$g_1,f(g_1)$ and $g_2,f(g_2)$ be nonsingular -- except perhaps
at small number of isolated points (and preferably in all of $\Omega$, cf. \eqref{defdet}, \eqref{detgbdd}).
Of course, if oversampling is possible 
with measurements at, say, $N>2$ time instances $T_i$,
for $i=1,\;\ldots\; N$, then this condition changes to the resulting matrix have rank two at each spatial point and becomes easier to satisfy.

\subsection*{Acknowledgement}
\vskip-4pt
\noindent
The work of the first author was funded in part by the Austrian Science Fund (FWF) 
[10.55776/P36318]. 

\noindent
The work of the second author was supported in part by the
National Science Foundation through award {\sc dms}-2111020.


\end{document}